\documentclass[STIX1COL]{WileyNJD-v2}

\articletype{Research Article}%

\received{-}
\revised{-}
\accepted{-}

\usepackage{moreverb}
\RequirePackage{times}
\RequirePackage{graphics}
\RequirePackage{graphicx}
\RequirePackage{afterpage}

\usepackage{cite}
\usepackage{amssymb}
\usepackage{latexsym}
\usepackage{bm}
\usepackage{graphics}
\usepackage{graphicx}
\usepackage{amsmath}

\usepackage{verbatim}

\begin{document}

\title{Parametric Topology Optimization with Multi-Resolution Finite Element Models}

\author[1]{Vahid Keshavarzzadeh$^{*,}$}

\author[1,2]{Robert M. Kirby}

\author[1,3]{Akil Narayan}

\authormark{Keshavarzzadeh {et al.}}

\address[1]{\orgdiv{Scientific Computing and Imaging Institute}, \orgname{University of Utah}, \orgaddress{\state{Utah}, \country{USA}}}

\address[2]{\orgdiv{School of Computing}, \orgname{University of Utah}, \orgaddress{\state{Utah}, \country{USA}}}

\address[3]{\orgdiv{Department of Mathematics}, \orgname{University of Utah}, \orgaddress{\state{Utah}, \country{USA}}}

\corres{{$^*$}Vahid Keshavarzzadeh, \email{vkeshava@sci.utah.edu}}



\abstract[Summary]{We present a methodical procedure for topology optimization under uncertainty with multi-resolution finite element models. We use our framework in a bi-fidelity setting where a coarse and a fine mesh corresponding to low- and high-resolution models are available. The inexpensive low-resolution model is used to explore the parameter space and approximate the parameterized high-resolution model and its sensitivity where parameters are considered in both structural load and stiffness. We provide error bounds for bi-fidelity finite element (FE) approximations and their sensitivities and conduct numerical studies to verify these theoretical estimates. We demonstrate our approach on benchmark compliance minimization problems where we show significant reduction in computational cost for expensive problems such as topology optimization under manufacturing variability while generating almost identical designs to those obtained with single resolution mesh. We also compute the parametric Von-Mises stress for the generated designs via our bi-fidelity FE approximation and compare them with standard Monte Carlo simulations. The implementation of our algorithm which extends the well-known 88-line topology optimization code in MATLAB is provided.}


\maketitle


\section{Introduction}\label{sec1}

Topology optimization is a systematic design framework for the distribution of given material resources within a specified spatial domain to achieve the maximum stiffness. This technique spawns from a seminal paper by Bends{\o}e and Kikuchi~\cite{Bendsoe88} in which  the structure layout, instead of structure boundaries as done in shape optimization, is optimized. Since then, in addition to solid mechanics~\cite{SIGMUND19971037} topology optimization has been developed and extended to various fields such as heat conduction, fluid dynamics, and multi-physics simulations~\cite{ALEXANDERSEN2018138,Dilgen2018,Lundgaard2018,Behrou17_ECS,Behrou17_IJNME}. A majority of existing works focus on deterministic analysis and optimization for such designs. However, the design performance varies due to inherent uncertainties in different parameters such as loading, boundary conditions, material properties and geometry.

This performance deficiency can be overcome by incorporating uncertainty analysis in the optimization process as done in robust design optimization (RDO) \cite{MartinezFrutos16, allaire15,deGournay08,KESHAVARZZADEH201647,KESHAVARZZADEH2017120} by minimizing the performance variation and reliability based design optimization (RBDO) \cite{Torii2017,MARTINEZFRUTOS2018180} by constraining the failure probability. The computational complexity is the outstanding challenge in these approaches due to requiring a considerable number of expensive simulations to capture variations in parameter/stochastic space.  Multi-resolution finite element models have been used in a number of studies to enhance the computational efficiency of topology optimization~\cite{Nguyen2010,PARK2015571,Filipov2016,KIM2000}. These multi-resolution topology optimization approaches are explored within a deterministic framework i.e. when the focus is only on a limited number of deterministic simulations throughout different mesh resolutions.

In this work, we adopt a different perspective in multi-resolution topology optimization and use coarse and fine finite element meshes within a parametric/stochastic framework. We use the inexpensive low-resolution model to traverse the parameter space and use that information to predict the stochastic response and sensitivity of the expensive high-resolution model. In this way the stochastic analysis is primarily performed via a low-resolution model which drastically decreases the computational complexity. Our method is non-intrusive i.e. it is implemented with minimal modification to the existing codes for topology optimization. We present our approach in the context of a generic density based topology optimization; however it is similarly applicable to a level-set based method.
The implementation of our approach which is the extension of ``Efficient topology optimization in MATLAB using 88 lines of code''~\cite{Andreassen2011} is provided in~\cite{Vahid2018}. We also provide error bounds for the bi-fidelity construction of compliance and its sensitivity which serves as a certificate for the convergence of our parametric topology optimization approach. We provide numerical results to delineate the error estimate for compliance and its sensitivity.

The present paper is organized as follows. Section \ref{S2} briefly describes the topology optimization including its deterministic and parametric forms. The details of our multi-resolution approach is presented in Section~\ref{S3}. Section~\ref{S4} presents numerical results for topology optimization under loading and manufacturing variability in conjunction with computational cost studies. Finally Section \ref{S5} contains concluding remarks.

\section{Topology Optimization}\label{S2}

\subsection{Notation and Setup}\label{S2_1}

We use bold characters to denote matrices, vectors and multivariate quantities e.g. $\bm x$ indicates a vector of variables in the domain of a multivariate function. We denote sets with uppercase letters e.g. $P$ is a set of sample parameters.

In this paper we mainly focus on parameterized elastostatics problems with the general form of
\begin{equation}
\begin{array}{l}
  \displaystyle \mathcal{L}\{\bm u(x,\bm p) \} = f(x,\bm p),\quad x \in \Omega, \quad \bm{p} \in \bm P \\
\\
\bm u(x,\bm p) =\bm u_{b}(\bm p), \quad x \in \partial \Omega, \quad \bm{p} \in \bm P,
\end{array}
\end{equation}
where $\mathcal{L}$ denotes a linear operator which will be replaced by the generic finite element global stiffness matrix shortly, $\Omega$ is the spatial domain, and the parameter $\bm{p} \in \bm P$ later will be treated as random variables. We consider two models: low-resolution model $\bm u^{L}: \bm P \rightarrow \bm U^L$ corresponding to the coarse mesh, and high-resolution model $\bm{u}^{H}: \bm P \rightarrow \bm U^H$ corresponding to the fine mesh. Here $\bm U^L$ and $\bm U^H$ are Hilbert spaces equipped with inner products $\langle.,.\rangle^L,\langle.,.\rangle^H$, respectively. For example, if $\bm U^L$ is finite-dimensional and two elements $a, b$ of this space are represented as coordinates in the vectors $\bm{a}$, $\bm{b}$, then one way to define an inner product is
\begin{equation}
\displaystyle \langle a, b \rangle^L = \frac{\bm a^T \bm b}{{\rm{dim}}~ {\bm a}}\\
\end{equation}
where ${\rm{dim}}$ denotes dimension or the size of vector ${\rm{dim}}~ {\bm a}={\rm{dim}}~ {\bm b}$. We hereafter assume that $\bm U^L$ and $\bm U^H$ are finite-dimensional, respectively of dimensions $N^L_{dof}$ and $N^H_{dof}$. Due to our coarse mesh/fine mesh assumptions, we have $N_{dof}^L < N_{dof}^H$.

In this paper we seek to find accurate approximations to the high resolution model $\bm u^H$ by using the low resolution model $\bm u^L$ in the parameter space. We use hat notation to denote approximations, e.g., $\hat{\bm u}$ is an approximation to $\bm u$. The low resolution solution can be computed on a collection of parametric samples $\Gamma_N=\{p_1,p_2,\ldots,p_N\} \subset \bm P$ in a relatively small amount of time. We denote the collection of these solution samples and their span as
\begin{equation}
\begin{array}{l}
\bm u^L (\Gamma_N) = \big\{\bm u^L(p_1), \bm u^L(p_2),\ldots, \bm u^L(p_N)   \big\},\\
\\
\bm U^L_{\Gamma_N}= {\rm{span}}~ \bm u^L(\Gamma_N) = {\rm{span}}~ \big\{\bm u^L(p_1), \bm u^L(p_2),\ldots, \bm u^L(p_N)\big\},
\end{array}
\end{equation}
and use similar notation for $\bm u^H$. The term {\rm{span}} above denotes the subspace that is formed with any linear combination of the solution samples $\bm u^L(p_i)$. We view the collection of samples as a matrix e.g. $\bm u^L({\Gamma_N}) = [\bm u^L(p_1) \bm u^L(p_2) \ldots \bm u^L(p_N)]_{ N_{dof}^L \times N}$ in our matrix computations.

Once we compute the low-resolution FE responses on the entire sample space, we find important samples and identify coefficients which ``relate'' the important samples to the rest of samples in $\Gamma_N$. Computing the high-resolution important samples which are few,  we then use the identified coefficients to estimate the high-resolution responses on the rest of samples. Figure~\ref{figschematic} shows the schematic representation of this bi-fidelity construction.

\begin{figure}[h!]
\centering
\includegraphics[width=5in]{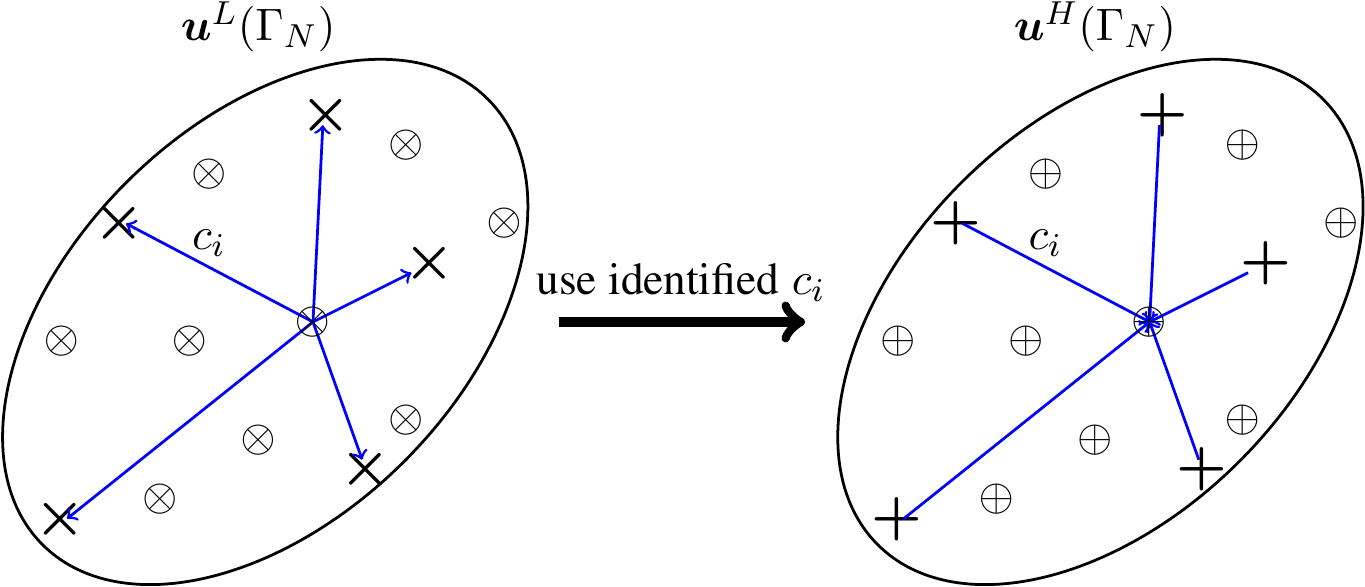}
\caption{{Schematic representation for bi-fidelity construction of parametric FE solutions. The low- and high-resolution important samples are denoted by $\times$ and $+$. To estimate the high-resolution samples on unknown locations in parameter space $\oplus$ we use the identified coefficients $c_i$ from the low-resolution sample space.}}\label{figschematic}
\end{figure}

\subsection{Deterministic Optimization}\label{S2_2}
Topology optimization in its original form is a constrained optimization problem which minimizes compliance subject to a volume constraint.  To find compliance, the structural analysis is typically performed via the Finite Element Method (FEM). We consider density based topology optimization in which the design space is characterized with element volume fractions. The optimization problem after finite element discretization is stated as
\begin{equation*}\label{to1_0_1}
  \begin{array}{r l l}
    \displaystyle \mathop{\min}_{\bm \rho} &  C(\bm \rho) = \bm{U}^{T} \bm{F} & \\
    \\
    \text{subject to} & V(\bm \rho) \leq \bar{V}
    \vspace{0.2cm}
    \\ & \bm{K}(\bm \rho) \bm{U}(\bm \rho) = \bm{F}(\bm \rho)
    \vspace{0.1cm}
    \\ &\rho_{min} \leq \bm \rho \leq 1,
   \end{array}
\end{equation*}
where $\bm{K}$, $\bm{U}$ and $\bm{F}$ denote the global finite element stiffness matrix, the displacement vector and force vector; and $\bm \rho$ is the vector of element volume fractions, $C$ is the compliance, $V$ is the volume and $0<\rho_{min} \ll 1$ is the lower bound for the volume fractions.

We process the design variables throughout the optimization in two ways: i) we impose a minimum length scale by using the filtered volume fraction to generate well-posed topology optimization formulation, and ii) we use Heaviside thresholding to generate more distinct interfaces and to model geometric variabilities.

The filtered volume fractions $\hat{\rho}$ are expressed via the cone kernel $K_F$,
\begin{equation}\label{to2}
\hat{\rho}(x_i) = \displaystyle \frac{ \sum_{j=1}^{n} K_F (x_i,x_j)  \rho(x_j)  }{ \sum_{j=1}^{n} K_F (x_i,x_j)}   , \qquad i=1,\ldots,n,
\end{equation}
where
\begin{equation}\label{to1}
K_F (x_i,x_j)  = \begin{cases} r_{min} - |x_i-x_j| & \mbox{if }  ~ |x_i-x_j|  \leq r_{min}    \\
 0 &\mbox{if } ~ |x_i-x_j|  > r_{min}.  \end{cases}
\end{equation}
In these expressions $r_{min}$ and $x_i$ denote the filter radius and the element $i$ centroid~\cite{Bruns01}. Ideally, the Heaviside step function is used to threshold the filtered volume fractions to $0$ and $1$ e.g. $\bar{\rho}=H(\hat{\rho}-0.5)$. However, to make the thresholding possible for sensitivity analysis and optimization a smooth approximation of a step function
\begin{equation}\label{to3}
\bar{\rho} = H_{\beta,\tau}(\hat{\rho}) = \displaystyle \frac{\tanh(\beta \tau)+\tanh(\beta (\hat{\rho}-\tau))}{\tanh(\beta \tau)+\tanh(\beta (1-\tau))},
\end{equation}
is used. In this approximation, $\beta$ controls the smoothness of transition and $\tau \in [0,1]$ serves as the threshold, i.e. $\lim_{\beta \rightarrow \infty} H_{\beta,\tau}(x) = H(x-\tau)$ pointwise for all $x \neq \tau$. We use the latter parameter $\tau$ to vary the boundary, i.e. geometry of the structure. It will be used as an uncertain parameter in our numerical examples to model the geometric tolerances which may arise in the manufacturing process.

Finally we use the Solid Isotropic Material with Penalization (SIMP) method to penalize intermediate volume fractions \cite{Bendsoe1989, Bendsoe1999}. As such, we compute the global stiffness matrix $\bm K$ by using the processed (thresholded-filtered) volume fractions $\bar{\rho}$,
\begin{equation}\label{to5}
\bm K =  \displaystyle \sum_{i=1}^{n} \bar{{\rho_i}}^\iota \bm K_i,
\end{equation}
where $n$ is the number of elements, $\iota=3$ is the penalization parameter and $\bm K_i$ is the nominal element $i$ stiffness matrix.
\subsection{Parametric Optimization}\label{S2_3}

We consider uncertainties in loading and geometry in our topology optimization statement problem by introducing the parameter $p$ in loading and structure stiffness
\begin{equation}\label{stc01}
  \bm{K}(\bm \rho,\bm p) \bm{U}(\bm \rho,\bm p) = \bm{F}(\bm \rho,\bm p), \quad \bm p \in \bm P
\end{equation}
where the parameters are treated as random variables. Since $U$ and subsequently $C$ are therefore random, we restate the optimization problem with a quantity of interest $Q$ which depends on statistical moments of $C$:
\begin{equation}\label{to1_0_1Q}
  \begin{array}{r l l}
    \displaystyle \mathop{\min}_{\bm \rho} & Q(\lambda)=\mu(\bm \rho) + \lambda \sigma (\bm \rho) & \\
    \\
    \text{subject to} & E[V(\bm \rho)] \leq \bar{V}
    \vspace{0.1cm}
    \\ &\rho_{min} \leq \bm \rho \leq 1,
   \end{array}
\end{equation}
where $\mu$ and $\sigma$ are the mean and standard deviation of the compliance $C$, cf. Section~\ref{S3_4}, and $\lambda$ is a weight factor associated with the standard deviation. We note that this formulation is pertinent to the case of geometric uncertainty where we consider the expected value for volume.

We consider a Karhunen-Loeve Expansion (KLE) to model uncertainties in both distributed load and spatial threshold parameters $\tau$.  We assume a covariance function
\begin{equation}\label{pcen212}
{\bf R}_{xx'}=\exp {\Big (}-{\frac {||x-x'||_2^{2}}{2l_c^{2}}}{\Big )},
\end{equation}
where $|| x-x' ||_2$ is the Euclidean distance between locations $x$ and $x'$ and $l_c$ is the correlation length.  We discretize this covariance function with $x$ and $x'$ as i) the finite element centroids in the case of spatial threshold and ii) finite element nodes that are under the influence of load in the case of distributed load to obtain the correlation matrix ${\bf R}$. We use $n_M$ first eigenvalue-eigenvector pairs $(\lambda_i, \bm \gamma_i)$ of the covariance matrix to generate the KL decomposition of a zero mean process as
\begin{equation}\label{pcen216}
Z(x,\bm p) = \gamma_0 + \sum_{i=1}^{n_M} \sqrt{\lambda_i} \bm\gamma_i(x) \varphi_i(\bm p)
\end{equation}
where $\varphi_i$ are uniform random variables and $\bm \gamma_i$ are eigenvectors, and $\gamma_0 > 0$ is a constant that is chosen to ensure a positive distributed load and avoid erratic distributions. We also post-process the random field $Z$ in the case of spatial threshold since $Z$ is not necessarily in the desirable range $Z \not\in [0,1]$. In this case to generate a value in the range $[0,1]$, we transform $Z$ into its Cumulative Distribution Function (CDF) i.e. $Z \gets \Phi(Z) \in [0,1]$ where $\Phi$ is the CDF of $Z$. For detailed discussion of this transformation see \cite{KESHAVARZZADEH2017120}. We finally use an affine map $\tau_i=a_1 Z(p_i) + a_2$ for suitable constants $a_1$ and $a_2$ to map values of $Z$ to an appropriate range.

The random processes are evaluated on Monte Carlo samples or quadrature points $Z(x,p_i)$ where each sample corresponds to a parametric load $\bm F(p_i)$ or stiffness matrix $\bm K(p_i)$. For a given finite element resolution the parametric analysis is summarized:

\begin{enumerate}
\item[--] Loading Uncertainty: For each parameter $p_i$ solve $\bm K \bm U = \bm F(p_i)$ to find the parametric $\bm U(p_i)$, parametric compliance $C(p_i)=\bm U(p_i)^T \bm K \bm U(p_i)$  and parametric compliance sensitivity $\partial C(p_i)/\partial \bm \rho = \bm \Lambda(p_i)^T (\partial \bm K(\bm \rho)/\partial \bm \rho) \bm U(p_i)$ where $\bm \Lambda = - \bm U$ is the adjoint sensitivity solution for compliance.
\item[--] Geometric Uncertainty: Solve $\bm K(p_i) \bm U = \bm F$ for each parameter $p_i$ to find the parametric displacement, compliance and its sensitivity similarly to loading uncertainty. Note that in this case, the derivative of stiffness matrix is dependent on the parameter.
\end{enumerate}

\begin{remark}
In the case of loading uncertainty, it is possible to compute the response only for the principal KL modes and use superposition to find the total response since the structure is linear. However as we mainly perform parametric analysis on the coarse mesh, solving the finite element system for a large number of samples does not pose a significant challenge.
\end{remark}

\section{Multiresolution Framework}\label{S3}

Our multi-resolution topology optimization framework has four major components summarized below. The detailed description of each component is provided in the following.

\begin{enumerate}[leftmargin=1in]
  \item[--] Translation: Given high resolution element density and parametric quantities, translate this data to the low resolution FE model cf. Section~\ref{S3_1}.
  \item[ --] Important Samples: Perform parametric analysis with the low resolution FE model on a large number of samples and determine important samples cf. Section \ref{S3_2}.
  \item[ --] Bi-Fidelity Construction: Compute interpolative coefficients in the parametric low resolution space and use those to estimate the parametric high resolution response and sensitivities cf. Section \ref{S3_3}.
  \item[--] Primal and Sensitivity Analyses: Compute statistical moments of compliance and their sensitivity and feed to optimizer cf. Section \ref{S3_4}.
\end{enumerate}

\subsection{Transition between High- and Low-resolution Models}\label{S3_1}

Our ultimate goal is to produce a design with fine resolution. The optimization progresses with the fine resolution model; however the information from the fine resolution should be translated to the low resolution model where the main FE computations are performed. Particularly, the KL model $Z(x,p_i)$ and densities $\bm \rho$ associated with the fine mesh should be translated to the low resolution mesh. The translation of KL model is trivial since it can be consistently generated for two resolutions by considering coarse and fine coordinates in the generation of covariance matrix. Similarly the KL modes for the low resolution mesh can be interpolated from the high resolution mesh.

To translate densities from the fine model to the coarse model we use a simple averaging
\begin{equation}\label{transform}
\rho^L =\frac{1}{n_H}  {\sum_{i=1}^{n_H} \rho_i^H}
\end{equation}
where $n_H$ is the number of fine mesh elements that can be placed in a coarse mesh considering we only use standard square elements. For example, four fine mesh elements with size $dx=dy=0.5$ cover a coarse element with $dx=dy=1$. Figure~\ref{figtranslate} shows the transformation of high resolution densities ($100 \times 100$ mesh) to low-resolution densities ($10 \times 10$ mesh).


\begin{figure}
  \centering
  \begin{tabular}{c c}
    \vspace{-1cm}
    \includegraphics[width=2.5in]{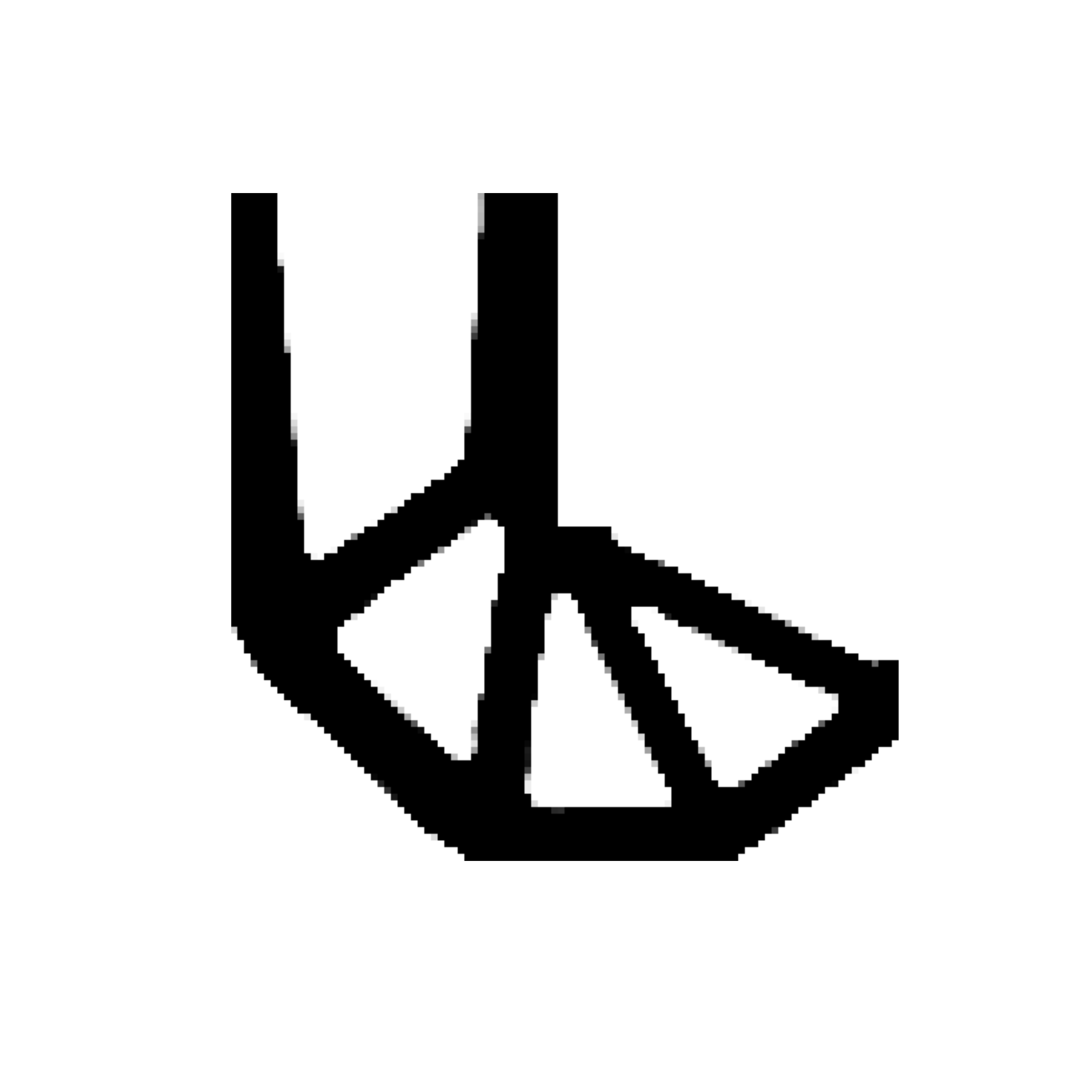} & \includegraphics[width=2.5in]{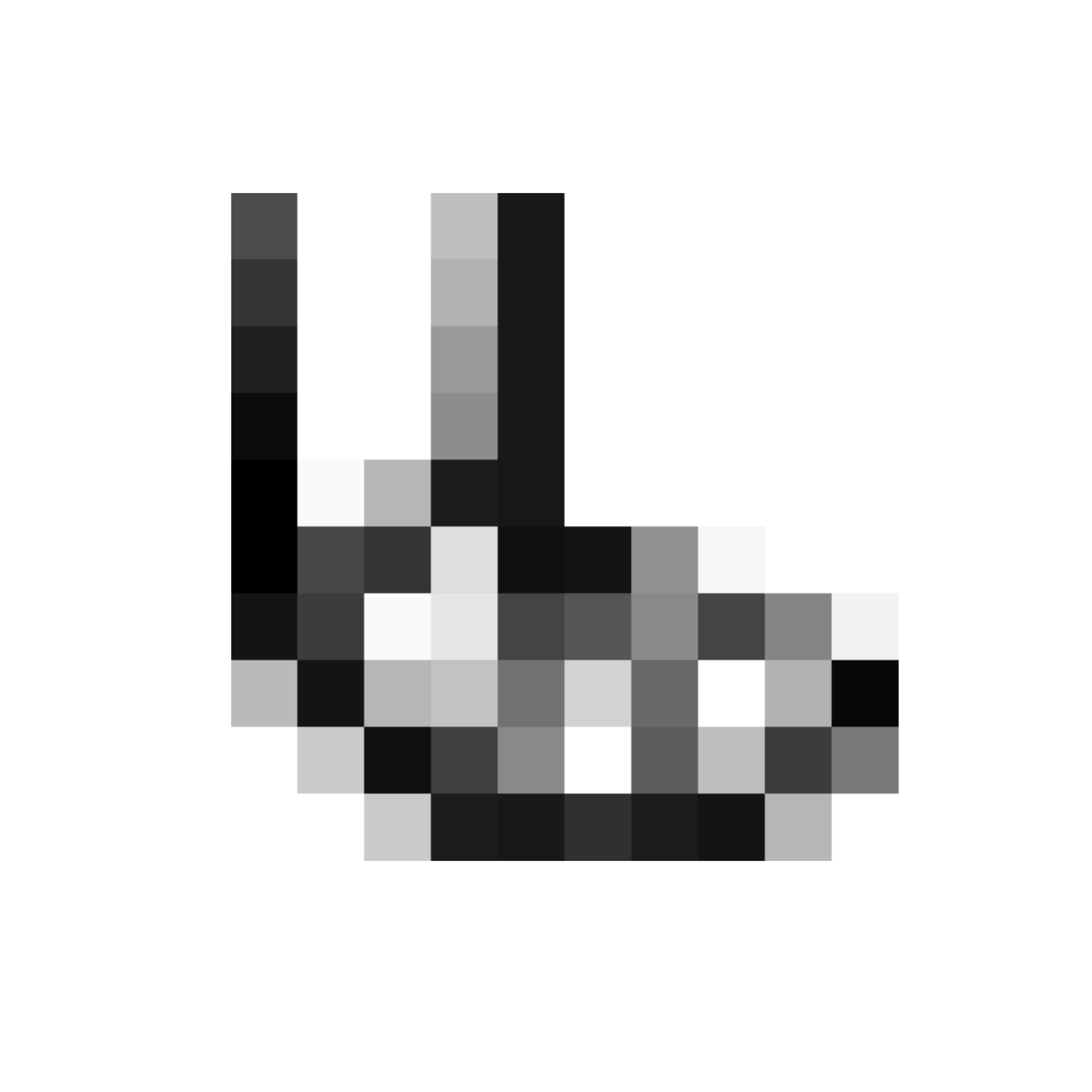} \\
    {$100 \times 100$ mesh} &  {$10 \times 10$ mesh}
  \end{tabular}
  \caption{Transformation of density $\bm \rho$ from high-resolution (left) to low-resolution (right). Each element in the right mesh is obtained by averaging the densities of $10 \times 10$ cell in the left mesh.}\label{figtranslate}
\end{figure}

\subsection{Interpolation Nodes}\label{S3_2}

Having the density and parameters associated with the low-resolution mesh, a large number of simulations $N$ is performed. A crucial step is to determine $n \ll N$ points at which the fine-resolution FEA will be performed. We determine important samples based on the fact that the span produced by the basis vectors $ \big\{\bm u^L(k_1), \bm u^L(k_2),\ldots, \bm u^L(k_n)   \big\}$ gets as close as possible to the span $\bm U^L_{\Gamma_N}$ which includes $N$ samples. To quantify closeness we define the standard distance between a function and a subspace as
\begin{equation}
d^L (x,Y) = \mathop{\inf}_{y \in Y} ||x-y||^L =  || (I-P_Y)x||^L,
\end{equation}
where $P_Y$ is the orthogonal projection operator onto a subspace $Y$, and $I$ is the identity operator.

Selection of important samples from a sample pool to minimize the above distance is a complex combinatorial problem.  Typically greedy procedures which are computationally tractable, are adopted for such problems, e.g., in the reduced basis methods (RBM)~\cite{Hesthaven15}. In particular, one may initialize the procedure with a trivial subspace $\Gamma_0=\{ \}$ and iteratively add samples to the set to maximize the distance between the newly added sample and the existing span, i.e.
\begin{equation}\label{greedy}
p_n= \mathop{ {\rm{arg~max}} }_{p \in \Gamma_N} d^L (\bm u^L(p), \bm U^L(\Gamma_{n-1}) ), \quad \Gamma_n= \Gamma_{n-1} \cup \{ p_n\}.
\end{equation}
We note that the distance is maximized such that the newly added samples cover more linearly independent directions in the span. We also note that the number of subsamples $n \ll N$ is determined based on our computational budget. Naturally using more samples results in more  accurate reconstruction of unknown samples in parameter space, as long as the number of samples does not exceed the numerical rank of $\bm u^L(\Gamma_N)$. 

This greedy procedure can be performed via a standard numerical linear algebra operations as discussed in~\cite{Narayan14}. In this reference three different linear algebraic choices namely i) column-pivoting QR decomposition, ii) full-pivoting LU decomposition, and iii) pivoted Cholesky decomposition are discussed. We adopt the column-pivoting QR decomposition in this work and use the pivot information to select the $n$ important samples of $\bm U^L$. The pivot contains the index of columns such that the first $n$ columns are linearly independent~\cite{golub1996matrix}.

We only need the integer-valued pivot in this paper, however given a matrix $\bm A \in \mathbb{R}^{m \times n}$ with $m>n$ one can use the upper-triangular matrix $\bm R$, from the output of the QR decomposition, see Equation~5.1.5 in~\cite{golub1996matrix} to compute $\bm Q$. It can then be verified that $\bm A \bm \Pi=\bm Q \bm R$ where $\bm \Pi$ is the $n \times n$ permutation matrix that contains the pivot indices. Equivalently the pivoted QR factorization can be performed via a built-in function in MATLAB denoted by \textit{qr} which we use in our numerical examples.

\subsection{Bi-Fidelity Construction}\label{S3_3}

We now have $N$ low-resolution FE solutions and have identified $n$ parameters at which we perform high resolution simulations $\{\bm u^H(p_i)\}_{i=1}^n$. We aim to find an approximation to $\bm u^H(p_i)$ on unknown samples $p_i \in \Gamma_{N-n}$.

Having $\bm u^L$ on the entire sample space, it is possible to construct the best parametric approximation to the low-resolution solutions $\{\bm u^L(p_i)\}_{i=n+1}^N$ as a function of important samples $\{\bm u^L(p_i)\}_{i=1}^n$ and subsequently use the same approximations to estimate parametric high-resolution solutions  $\{\hat{\bm u}^H(p_i)\}_{i=n+1}^N$. Precisely, the best approximation is defined such that $\|\hat{\bm u}^L(p_i) - u^L(p_i)\|,~\forall p_i \in \{p_{n+1},\ldots,p_{N}\}$ is minimized. This can be expressed as
\begin{equation}\label{eq_bf_cons_1}
  \langle \hat{\bm u}^L(p_i), \gamma(p) \rangle^L= \langle \bm u^L(p_i), \gamma(p) \rangle^L \quad \forall \gamma \in \bm U^L(\Gamma_n), \quad \hat{\bm u}^L(p_i) \in \bm U^L(\Gamma_n),
\end{equation}
which is equivalent to a projection of $\{\bm u^L(p_i)\}_{i=n+1}^N$ onto the space $\bm U^L(\Gamma_n)$ denoted by $\mathrm{P}_{\bm U_{\Gamma_n}^L}$ i.e.,
\begin{equation}\label{eq_bf_cons_2}
\mathbf{I}_{L}^{L}(\bm u^L(\Gamma_n)) := \hat{\bm u}^L(p_i) =\displaystyle \sum_{j=1}^{n} c_j \bm u^L(p_j),
\end{equation}
where we have defined the interpolation operator $\mathbf{I}_{L}^{L}$ using coefficients $c_j$ obtained from low-resolution conditions \eqref{eq_bf_cons_1} to approximate low-resolution solutions on $\{ p_i\}_{i=n+1}^N$.

The linear algebraic version of \eqref{eq_bf_cons_1} is
\begin{equation}\label{coef}
 \bm G^L \bm c^L = \bm f^L
\end{equation}
where the low-resolution Gramian $\bm G^L$ and $\bm f^L$ are expressed as
\begin{equation}\label{gramm_f}
\begin{array}{l}
\bm G^L_{ij} = \langle \bm u^L(p_i),\bm u^L(p_j) \rangle^L, \quad \forall p_i,p_j \in \{p_1,\ldots,p_n \}.\\
\\
\bm f^L = (f^L_i)_{1 \leq i \leq n}, \qquad f^L_i = \langle \bm u^L(p_i),\bm u^L(p_j) \rangle^L, \quad p_j \in \{p_{n+1},\ldots,p_N  \}.
\end{array}
\end{equation}
Note that the Gramian $\bm G^L$ is constructed once for all important samples however the right hand side $\bm f^L$ is computed for every parameter $\{ p_i\}_{i=n+1}^N$ individually. Therefore this analysis yields $N-n$ coefficient vectors $\bm c$. We also note that solving the linear system~\eqref{coef} is not challenging since the size of $\bm G^L$ is small.

Upon finding the coefficients $\bm c$ we estimates the higher resolution solutions $\{\bm u^H({p_i})\}_{i=n+1}^N$ by a lifting procedure i.e.
\begin{equation}\label{eq:MHF_1}
\hat{\bm u}^H({p_i})=\mathbf{I}_{L}^{H}(\bm u^L(\Gamma_n)) := \displaystyle \sum_{j=1}^{n} c_j \bm u^H(p_j).
\end{equation}
We now have the high-resolution response on the entire sample space i.e. we have $\bm U^H(\Gamma_N) = \{\bm u^H(p_1),\ldots,\bm u^H(p_n), \hat{\bm u}^H(p_{n+1}),\\ \ldots,\hat{\bm u}^H(p_{N})\}$. Having the high-resolution response, which is equivalent to the adjoint solution, we compute approximations $\hat{C}$ and $\partial \hat{C}/\partial \bm \rho$ to the compliance and its sensitivity for the sample space $\Gamma_N$: $\{\hat{C}(p_1),\partial \hat{C}(p_1)/\partial \bm \rho,\ldots, \hat{C}(p_N),\partial \hat{C}(p_N)/\partial \bm \rho \}$ as discussed in Section~\ref{S2_3}.

Algorithm~\ref{alg:high} summarizes the steps in Bi-Fidelity construction.

\begin{algorithm}
\caption{Bi-Fidelity Construction}\label{alg:high}

\begin{algorithmic}[1]
\State Given $\bm u^L(\Gamma_N)$ and important samples $\{p_i \}_{i=1}^n$ compute $\{\bm u^H(p_i)\}_{i=1}^n$
\State Form Gramian $ G_{ij}^L = \langle \bm u^L(p_i), \bm u^L(p_j) \rangle^L$ once $\forall p_i,p_j \in \{p_1,\ldots,p_n \}$ using Equation~\eqref{gramm_f}
\State Compute the right hand side in~\eqref{coef} for every parameter $\{p_i \}_{i=n+1}^N$ using Equation~\eqref{gramm_f}
\State Compute coefficients $\bm c^L$ in~\eqref{coef}
\State Form the interpolants for $\bm u^H$ i.e. $\hat{\bm u}^H =\sum_{i=1}^n c_i \bm u^H(p_i)$
\State Compute compliance and its sensitivity $\{C( p_i),\partial C( p_i)/\partial \bm \rho \}_{i=1}^N$ associated with the response samples $\bm U^H(\Gamma_N) = \{\bm u^H(p_1),\ldots,\bm u^H(p_n),\hat{\bm u}^H(p_{n+1}),\ldots,\hat{\bm u}^H(p_{N})\}$
\end{algorithmic}
\end{algorithm}

\subsection{Primal and Sensitivity Analyses}\label{S3_4}

To compute statistical moments in the optimization problem~\eqref{to1_0_1} we use either a quadrature rule or Monte Carlo integration. In either case, we compute $N$ higher resolution samples (or approximations to them as described previously) and subsequently compute $\mu$ and $\sigma$
\begin{equation}\label{stc01}
\begin{array}{l}
  \mu = E[C] =\displaystyle \sum_{i=1}^{N} C(p_i) w_i \\
  \\
  \sigma = \sqrt{E[(C-\mu)^2]} = \sqrt{\displaystyle \sum_{i=1}^{N} C^2(p_i) w_i - \mu^2 }
  \end{array}
\end{equation}
where $w_i=1/N$ in the case of Monte Carlo integration. The sensitivity of statistical moments are computed similarly, i.e.
\begin{equation}\label{stc_d01}
\begin{array}{l}
 \displaystyle \frac{\partial \mu} {\partial \bm \rho} = \displaystyle \sum_{i=1}^{N} \frac{\partial C(p_i)}{ \partial \bm \rho} w_i \\
  \\
  \displaystyle \frac{\partial \sigma}{\partial \bm \rho}=\frac{1}{\sigma} \Big(\displaystyle \sum_{i=1}^{N} C(p_i) \frac{\partial C(p_i)}{\partial \bm \rho} w_i - \mu \frac{\partial \mu}{ \partial \bm \rho}\Big)
  \end{array}
\end{equation}
where $\partial C(p_i)/\partial \bm \rho$ is readily available from the high-resolution displacement $u^H(p_i)$. It is noted that in this work we only focus on statistical moments in the optimization problems. Similarly one can compute probability of failure and its sensitivity either via a Bi-Fidelity Monte Carlo analysis or polynomial chaos expansion as done in~\cite{KESHAVARZZADEH2017120}.

Algorithm~\ref{alg:bif_topopt} summarizes the steps and  Figure~\ref{figflow} depicts the flowchart for the bi-fidelity topology optimization for compliance minimization.

\begin{algorithm}
\caption{Bi-Fidelity Topology Optimization}\label{alg:bif_topopt}
\begin{algorithmic}[1]
\State Given design iterate $\bm \rho$ in the high-resolution mesh, transform it to a low-resolution mesh via Equation~\eqref{transform}
\State Perform low-resolution FEA on $N$ samples where $N$ is determined from a Monte Carlo analysis or a quadrature rule and find $N$ displacement vectors $\{\bm u^L(p_i)\}_{i=1}^N$
\State Perform pivoted QR decomposition on $\bm U^L$ to select $n \ll N$ important samples $\{p_i\}_{i=1}^n$ cf. Section~\eqref{S3_2}
\State Compute low resolution Gramian and use it to find interpolative coefficients $\{c^L_i\}_{i=1}^{N-n}$ associated with the remaining samples cf. Equations~\eqref{coef},\eqref{gramm_f}
\State Perform high-resolution FEA on $n$ samples to find displacement vectors $\{\bm u^H(p_i)\}_{i=1}^n$
\State Use interpolative coefficients $\{c_i\}_{i=1}^{N-n}$ of Step $4$ to estimate the displacement on the remaining $N-n$ high-resolution samples $\{\hat{\bm u}^H(p_i)\}_{i=1}^{N-n}$
\State Compute compliance and its sensitivity $\{C(p_i),\partial C(p_i)/\partial \bm \rho \}_{i=1}^N$ associated with the displacements $\Big\{\{\bm u^H(p_i)\}_{i=1}^{n},\{\hat{\bm u}^H(p_i)\}_{i=1}^{N-n}\Big\}$
\State Compute $\mu$ and $\sigma$ and their sensitivities cf. Equations~\eqref{stc01},\eqref{stc_d01}
\State Feed the information from the previous step to gradient based optimizer and get the new high resolution design iterate, go back to Step $1$
\end{algorithmic}
\end{algorithm}

\begin{figure}[!h]
\centering
\includegraphics[width=4.5in]{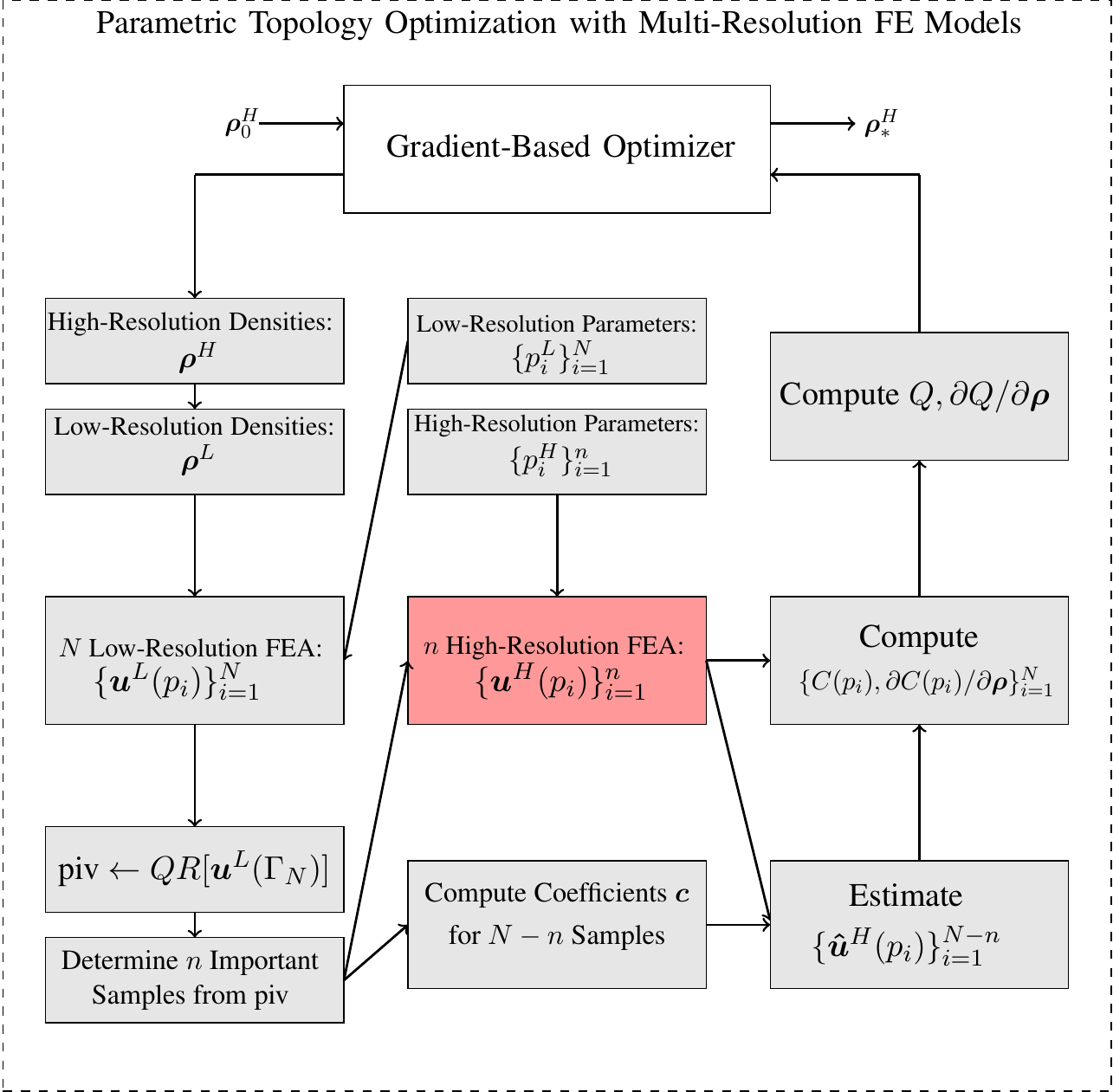}
\caption{{Flowchart of the parametric topology optimization with multi-resolution FE models. Note that the low- and high-resolution parameters are obtained once before the inception of optimization.}}\label{figflow}
\end{figure}

\subsection{Error Estimate for Bi-Fidelity Compliance and Its Sensitivity}\label{S3_5}

In this section, we show that our bi-fidelity estimate for each compliance sample and its derivative is close enough to the high fidelity estimate. In other words, the differences $|C(p_i)-\hat{C}(p_i)|,~ ||\partial C(p_i)/\partial \bm \rho-\partial \hat{C}(p_i)/\partial \bm \rho||_2,~ \forall p_i \in \Gamma$ is bounded for an arbitrary iteration throughout the optimization. In most of our analysis we do not show the dependence of these quantities on parameters and optimization iteration.

To derive an estimate for compliance we first need to analyze the difference in the displacement approximations i.e. $||\bm u^H - \hat{\bm u}^H||_2$. This difference is bounded via the triangle inequality
\begin{equation}\label{eq_estimate_1}
||\bm u^H - \hat{\bm u}^H|| \leq ||\bm u^H - \mathrm{P}_{\bm U_{\Gamma_n}^H} \bm u^H  || + ||\mathrm{P}_{\bm U_{\Gamma_n}^H} \bm u^H - \hat{\bm u}^H ||,
\end{equation}
where $P_{\bm U_{\Gamma_n}}^H$ is the $\bm U^H$-orthogonal projector onto $\bm U_{\Gamma_n}^H$. The first term on the right hand side is rather abstract and is typically bounded via a Kolmogrov $n$-width argument for the samples obtained from the greedy procedure~\eqref{greedy} i.e.
\begin{equation}
\mathop{\sup}_{p_i \in \Gamma}  ||\bm u^H - \mathrm{P}_{\bm U_{\Gamma_n}^H} \bm u^H(p_i)  || \leq \mathrm{C_1} \sqrt{d_{n/2}(u(\Gamma))}
\end{equation}
where $d_n(\bm u^H(\Gamma))$ is the $N$-width of the manifold $\bm u^H(\Gamma)$~\cite{Narayan14} . Detailed analysis of this term is beyond the scope of this work; however, we assume that the contribution of this error is negligible with respect to the second term on the right hand side. This assumption is frequently true in practice and therefore we assume it as a fraction of the second error term. This assumption will be verified in the numerical examples.

Before stating a result on the bound for $||\bm u^H - \hat{\bm u}^H||$ we first introduce some notations and assumptions. Similarly to Equation~\eqref{gramm_f}, let
\begin{equation}
\begin{array}{l l }
 G^H_{ij} = \langle \bm u^H(p_i),\bm u^H(p_j) \rangle^H, \quad \forall p_i,p_j \in \{p_1,\ldots,p_n \}\\
\\
 f^H_i = \langle \bm u^H(p_i),\bm u^H(p_j) \rangle^H, \quad p_j \in \{p_{n+1},\ldots,p_N  \}
\\
\mathrm{P}_{\bm U_{\Gamma_n}^H} \bm u^H(p_i) = \displaystyle \sum_{j=1}^{n} c^H_j \bm u^H(p_j)
\end{array}
\end{equation}
be the high-resolution Gramian, forcing term and interpolator. The high-fidelity coefficients $\bm c^H$ can then be obtained from $\bm G^H \bm c^H= \bm f^H$. Also let $\sqrt{\bm G^H}=\bm V \sqrt{\bm S} \bm V^T$ be the square root of the positive (semi)definite Gramian  $\bm G^H=\bm V \sqrt{\bm S} \bm V^T$ where $\bm S, \bm V$ are eigenvalues and eigenvectors of $\bm G^H$.



\begin{assumption}
The low- and high-fidelity coefficients are close i.e.
\begin{equation}\label{assumption1}
 \|\bm c^H - \bm c^L \|  \leq \epsilon
\end{equation}
and the ratio
\begin{equation}\label{assumption2}
\delta =\frac{||\bm u^H - \mathrm{P}_{\bm U_{\Gamma_n}^H} \bm u^H  ||}{||\mathrm{P}_{\bm U_{\Gamma_n}^H} \bm u^H - \hat{\bm u}^H ||}
\end{equation}
is small i.e. $\delta \ll 1$.
\end{assumption}

The following lemma bounds the error in the displacement approximation~\eqref{eq_estimate_1}:

\begin{lemma}\label{lemma1}
Let $n$ important samples be given via~\eqref{greedy} and the assumptions~\eqref{assumption1} and~\eqref{assumption2} hold, then
\begin{equation}\label{lemma1_1}
|| \bm u^H - \hat{\bm u}^H|| \leq {\epsilon}(1+\delta) \sqrt{\sigma_{max}(\bm G^H)}
\end{equation}
where $\sigma_{max}(\bm G^H)$ is the largest singular value of $\bm G^H$.

%
\end{lemma}

\begin{proof}
The second error in the right hand side of~\eqref{eq_estimate_1} is expressed as
\begin{equation}\label{proof_lemma}
\begin{array}{l l}
\|\mathrm{P}_{\bm U_{\Gamma_n}^H} \bm u^H - \hat{\bm u}^H\|^2 &=\left \| \displaystyle \sum_{j=1}^n (c_j^H-c_j^L) \bm u^H(p_j) \right \|^2\\
\\
&=(\bm c^H-\bm c^L) \bm G^H (\bm c^H-\bm c^L) = \| \sqrt{\bm G^H}(\bm c^H-\bm c^L)||^2\\
\\
& \leq \| \sqrt{\bm G^H}\|^2 \|\bm c^H-\bm c^L\|^2 \leq \sigma_{max}(\bm G^H) \epsilon^2
\end{array}
\end{equation}
which yields $\|\mathrm{P}_{\bm U_{\Gamma_n}^H} \bm u^H - \hat{\bm u}^H\| \leq \epsilon \sqrt{\sigma_{max}(\bm G^H)}$. Using this inequality and the definition~\eqref{assumption2} in inequality~\eqref{eq_estimate_1} yields the estimate~\eqref{lemma1_1}.
\end{proof}

\begin{remark}\label{rem1}
As mentioned the ratio $\delta$ is not known \textit{a priori} via analytical estimates. In practice we directly compute it in our numerical experiments by using first few unimportant high-resolution samples in the pivot vector and selecting the worst ratio i.e. the largest one. We also compute $\bm f^H$ and subsequently $\bm c^H$ associated with that particular unimportant sample and find an approximate value for $\epsilon$.
\end{remark}

The following proposition bounds the error for the compliance and its sensitivity:

\begin{proposition}\label{prop:C-error}
Let the bi-fidelity approximation error in displacement be given by~\eqref{lemma1_1} then the error in compliance approximation and its sensitivity is bounded by

\begin{equation}\label{prop1}
\begin{array}{c l}
|C-\hat{C}| & \leq \mathrm{A}~ \sigma_{\max}(\bm K) \\
\\

\left \|\displaystyle \frac{\partial C}{\partial \bm \rho}-\frac{\partial \hat{C}}{\partial \bm \rho}\right \| & \leq \mathrm{A}~ \sigma_{\max}\left(\displaystyle \frac{\partial \bm K}{\partial \bm \rho}\right)
\end{array}
\end{equation}
where
\begin{equation}\label{prop2}
  \mathrm{A} = {\epsilon}(1+\delta) \sigma_{\max} (\bm G^H) \left[ 2\|\bm c^l \|  + \epsilon (1 + \delta)\right],
\end{equation}
$\sigma_{\max}$ denotes the largest singular value and $\bm K,\partial \bm K/\partial \bm \rho$ denote the global stiffness matrix cf. Equation~\eqref{to5} and its derivative.
\end{proposition}

\begin{proof}
We use the result of Lemma~\ref{lemma1} and the definition of compliance and its derivative i.e. $C = \bm u^T \bm K \bm u$,~$\partial C/\partial \bm \rho=-\bm u^T (\partial \bm K/\partial \bm \rho) \bm u$:
\begin{equation}\label{prop_proof1}
\begin{array}{c l}
|C-\hat{C}| &= \| \bm u^T \bm K \bm u - \hat{\bm u}^T \bm K \hat{\bm u} \| = \| (\bm u^T-\hat{\bm u}^T) \bm K \bm u + \hat{\bm u}^T \bm K (\bm u -\hat{\bm u})  \|\\
\\
& \leq \| \bm u - \hat{\bm u}\| \sigma_{\max}(\bm K) \Big[\|\bm u   \| + \|\hat{\bm u}  \|\Big]\\
\\
& \leq \| \bm u - \hat{\bm u}\| \sigma_{\max}(\bm K) \Big[\| \bm u - \hat{\bm u}\| +2\|\hat{\bm u} \|    \Big]\\
\end{array}
\end{equation}
According to~\eqref{lemma1_1} and~\eqref{proof_lemma}
\begin{equation}\label{prop_proof2}
\begin{array}{c l}
\|\bm u-\hat{\bm u}\| \leq {\epsilon}(1+\delta) \sqrt{\sigma_{max}(\bm G^H)} \\
\\
\| \hat{\bm u}\| = \left \| \sqrt{\bm G^H} \bm c^L  \right \| \leq \sqrt{\sigma_{max}(\bm G^H)} \| \bm c^L \|.
\end{array}
\end{equation}
Using the above estimates in~\eqref{prop_proof1} yields
\begin{equation}\label{prop_proof3}
\begin{array}{l l}
|C-\hat{C}| &\leq {\epsilon}(1+\delta)  \sigma_{\max}(\bm K) \sigma_{\max} (\bm G^H) \Big[{\epsilon}(1+\delta) +2 \|\bm c^l \| \Big] \\
 &  = A \sigma_{\max}(\bm K)
\end{array}
\end{equation}
Similarly the bound for compliance sensitivity is obtained by replacing $\sigma_{\max}(\bm K)$ with $\sigma_{\max}(\partial \bm K/\partial \bm \rho)$ in the above estimate.
\end{proof}

\section{Numerical Illustration}\label{S4}

\subsection{Loading Variability}

In this example we consider variations in loading on a square carrier plate shown in Figure~\ref{figcarrier}. The domain is discretized using standard square finite elements with different number of elements from coarse to fines meshes i.e. $4 \times 4$, $10 \times 10$ , $20 \times 20 $ , $50 \times 50$ and $100 \times 100$ elements. We fix top two layers of elements as solid elements to ensure the connectivity between load and structure. The load consists of a deterministic distributed vertical load $f_2=2$ and random horizontal load $f_1(x,p)$ which is modeled as a random field cf. Equation~\eqref{pcen216} with zero mean and square exponential covariance function similarly to Equation~\eqref{pcen212} with $l_c=0.2$. In Equation~\eqref{pcen216} we assume $\phi(\bm p)$ as uniform random variables $U[-1,1]$ and consider $n_M=10$ modes which yields the ratio $\sum_{i=1}^{10} \sqrt{\lambda_i}/\sum_{i=1}^{100} \sqrt{\lambda_i}=0.9$, reasonably close to $1$. First five modes of loading is shown in Figure~\ref{figKLE_load}.

\begin{figure}[!h]
\centering
\includegraphics[width=2.5in]{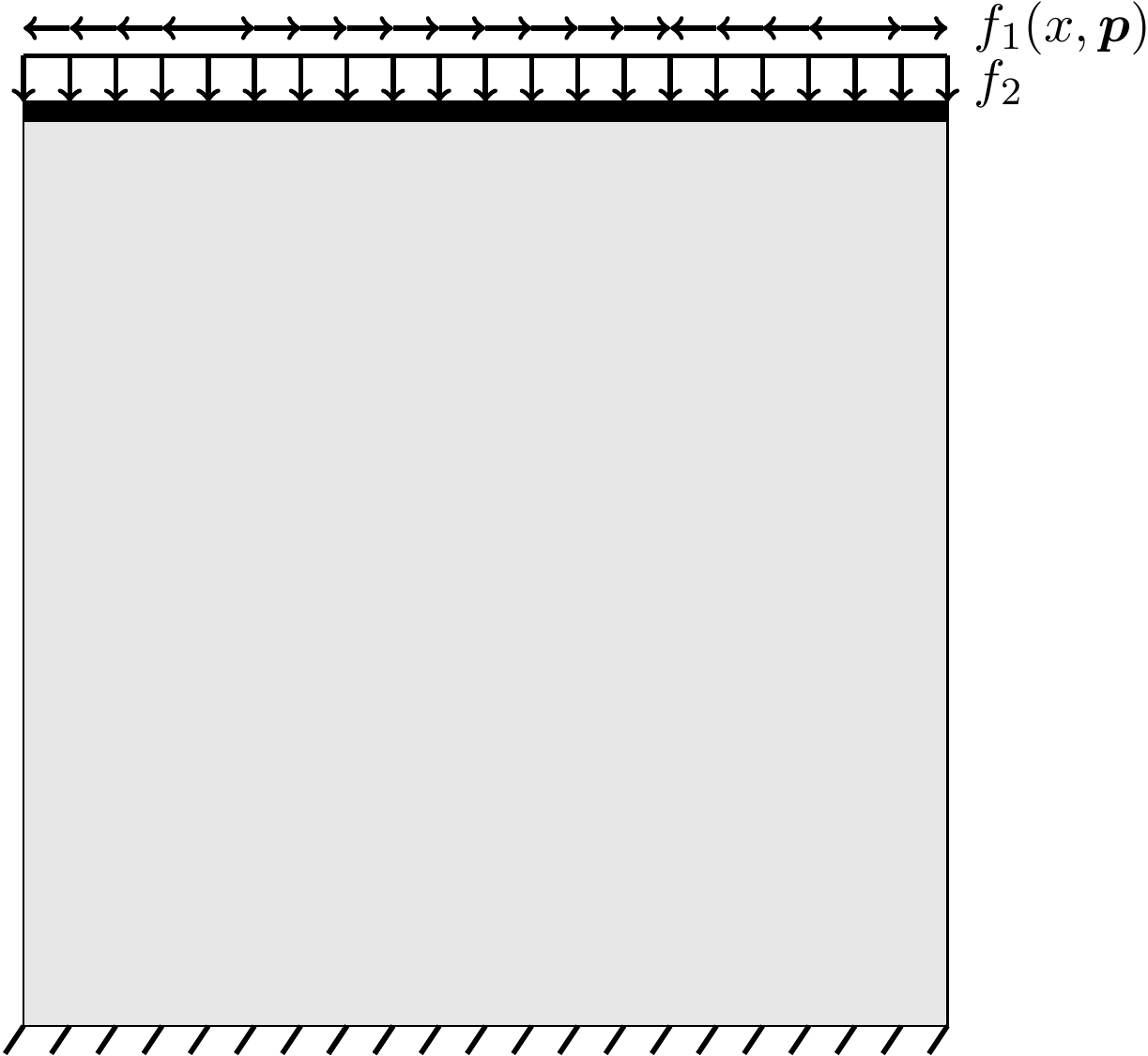}
\hspace{-1.2cm}
\caption{{Carrier Plate}}\label{figcarrier}
\end{figure}

\begin{figure}[!h]
\centering
\includegraphics[width=5.5in]{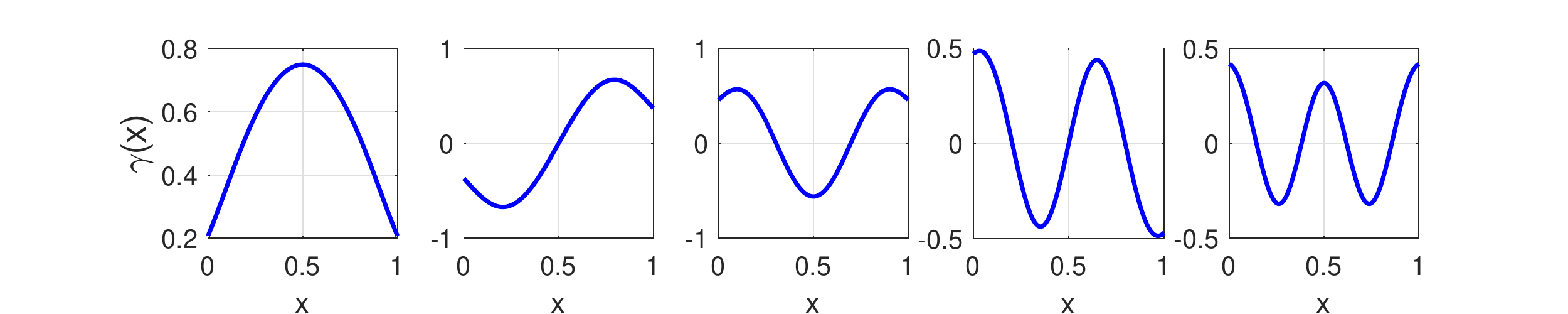}
\caption{{First five modes of Karhunen-Loeve Expansion for loading}}\label{figKLE_load}
\end{figure}

To apply these modal loads on the structure with lower resolution we simply interpolate the high resolution modes e.g. the ones with $100$ elements (shown in Figure~\ref{figKLE_load}). To compute the statistical moments we use a quadrature rule named \textit{designed quadrature} which was developed previously by the authors~\cite{Keshavarzzadeh2018}. This quadrature rule is specially designed for integration in multiple dimensions where the positivity of weights are ensured and in all cases tested the number of nodes is smaller than those in a corresponding sparse grid rule. We use $N=148$ points that integrates a function with $10$ variables associated with $n_M$ in~\eqref{pcen216} and the total order $\bm \alpha=5$ i.e. this set of points can integrate $\int x^{\alpha_1}_1 x^{\alpha_2}_2 \ldots x^{\alpha_{10}}_{10} dx_1 dx_2 \ldots dx_{10}$ for $\sum_{i=1}^{10} \alpha_i \leq 5$ accurately. We deem total order $\bm \alpha=5$ sufficient for our problem.

The next two figures show some intermediate results associated with an iteration in the middle of optimization. Figure~\ref{figsvd} shows the decay of singular values in the low-resolution models $\bm u^L(\Gamma_N)$ for $4 \times 4$ and $10 \times 10$ meshes. As seen the coarser mesh has only numerical rank $r=6$ while the finer mesh has $r=11$. This suggests that all $10$ horizontal modes and the vertical load can be captured by the finer mesh while the coarser one does not have enough degrees of freedom (only 6 DOFs) to capture all modal loads.

\begin{figure}[!h]
\centering
\includegraphics[width=5in]{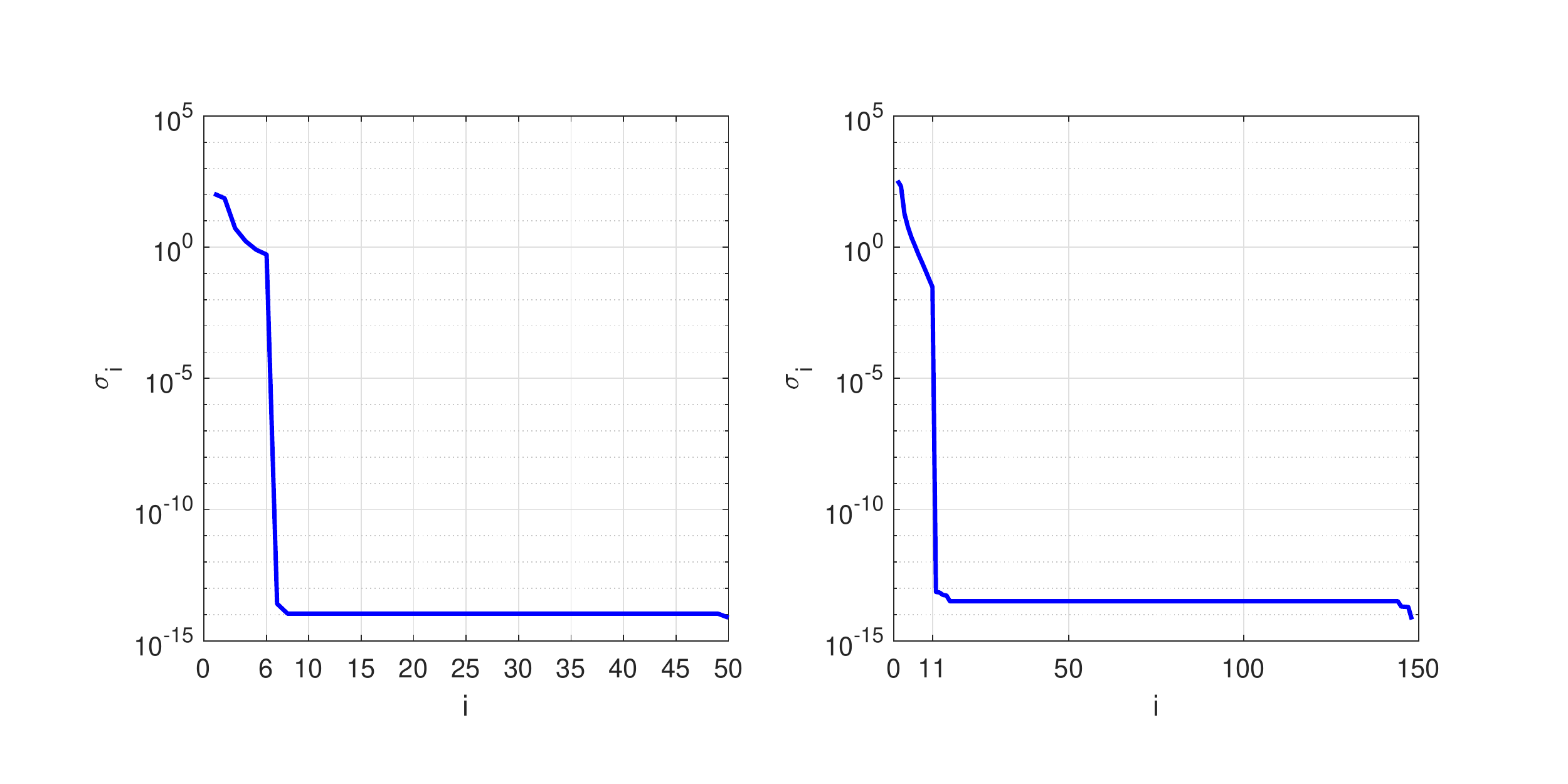}
\caption{{Decay of singular values in $\bm u^L(\Gamma_N)$ for $4 \times 4$ and $10 \times 10$ meshes}}\label{figsvd}
\end{figure}

Figure~\ref{fign_vary_F} shows the difference between bi-fidelity approximation and high-fidelity solutions for the displacement, compliance and compliance sensitivity with respect to different values on $n$, number of high-resolution simulations.

\begin{figure}[!h]
\centering
\includegraphics[width=5.5in]{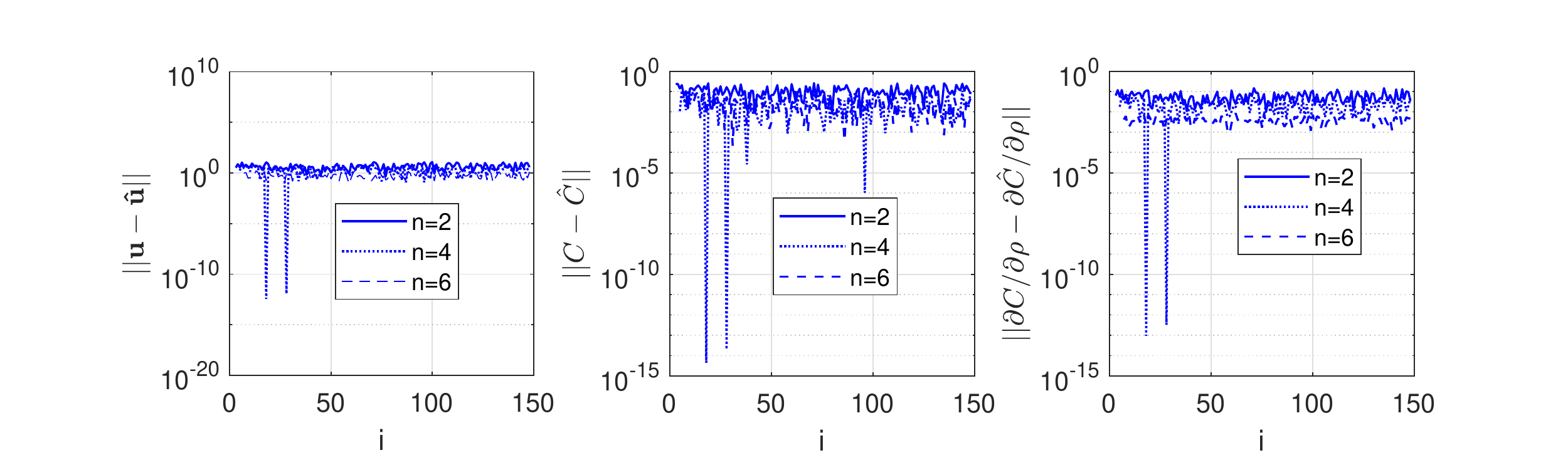}
\vspace{-0.2cm}
\includegraphics[width=5.5in]{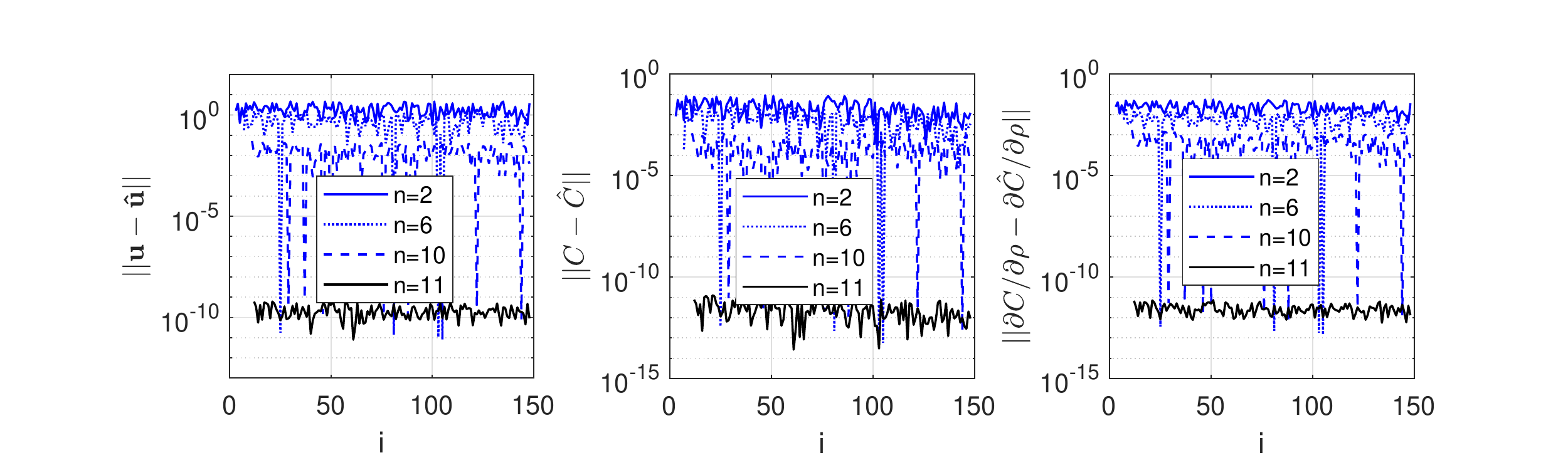}
\caption{{Bi-fidelity actual approximation error for displacement, compliance and compliance sensitivity with respect to different number of high-resolution simulations $n$ for $4 \times 4$ (top) and $10 \times 10$ (bottom) meshes.}}\label{fign_vary_F}
\end{figure}

Topology optimization results with different meshes are plotted in Figure~\ref{fign_vary_mesh}. We use filter radius $r_{min}=6,3,2,1.5,1.05$ for different meshes; in the case of two coarsest meshes $4 \times 4$ and $10 \times 10$ we only fix the top layer instead of top two layers. We use the same optimality criteria algorithm~\cite{Bendsoe03,Andreassen2011} to update the design parameters until the optimization is converged. From these plots it is obvious that the single resolution optimizations (top plots) with coarse meshes yield uninformative topology designs but using these coarse meshes in our bi-resolution framework results in designs that are almost identical to high-resolution optimization.

\begin{figure}[!h]
\begin{tabular}{l l l l l}
\vspace{-0.5cm}
\hspace{-0.4cm}\includegraphics[width=1.5in]{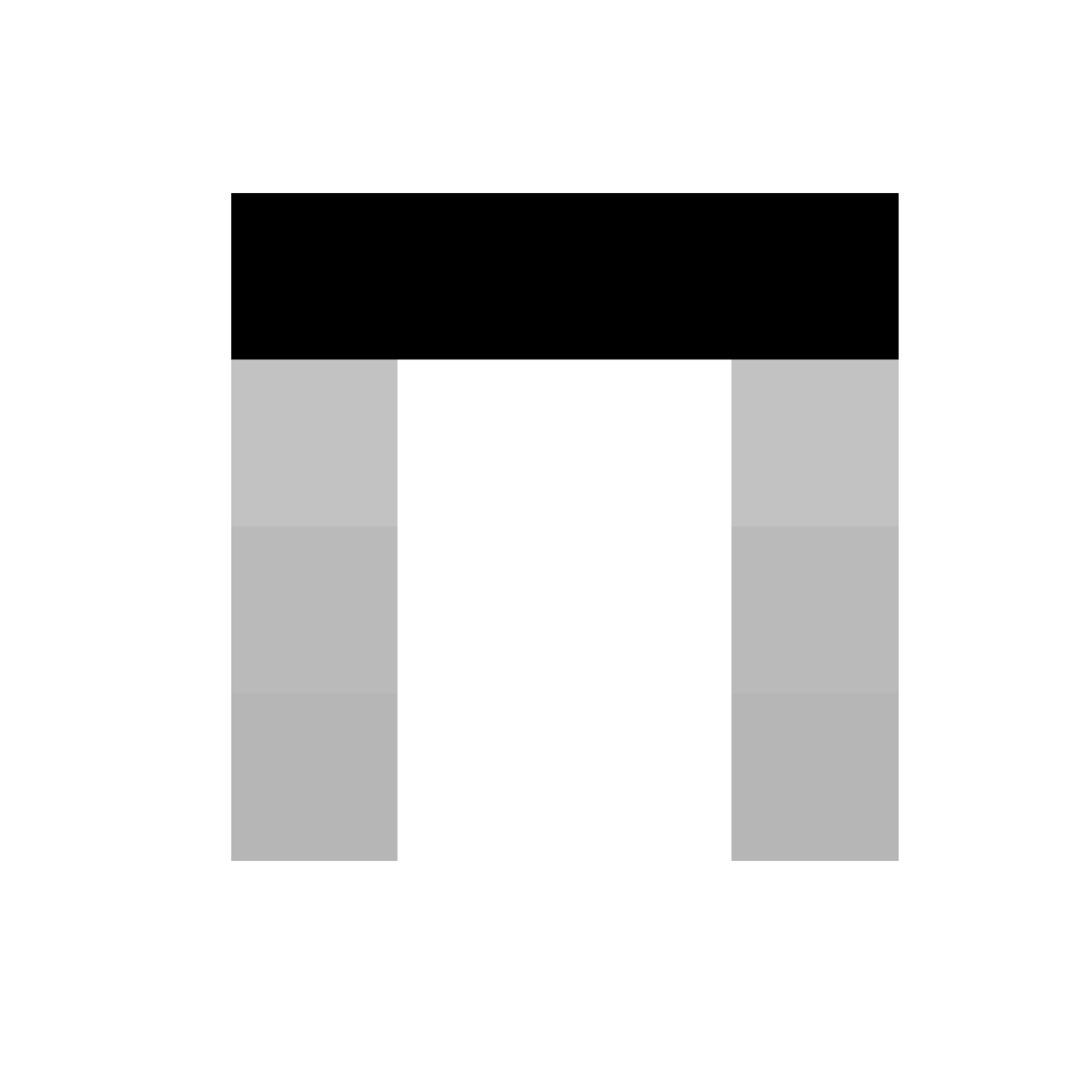}&\hspace{-0.8cm}\includegraphics[width=1.5in]{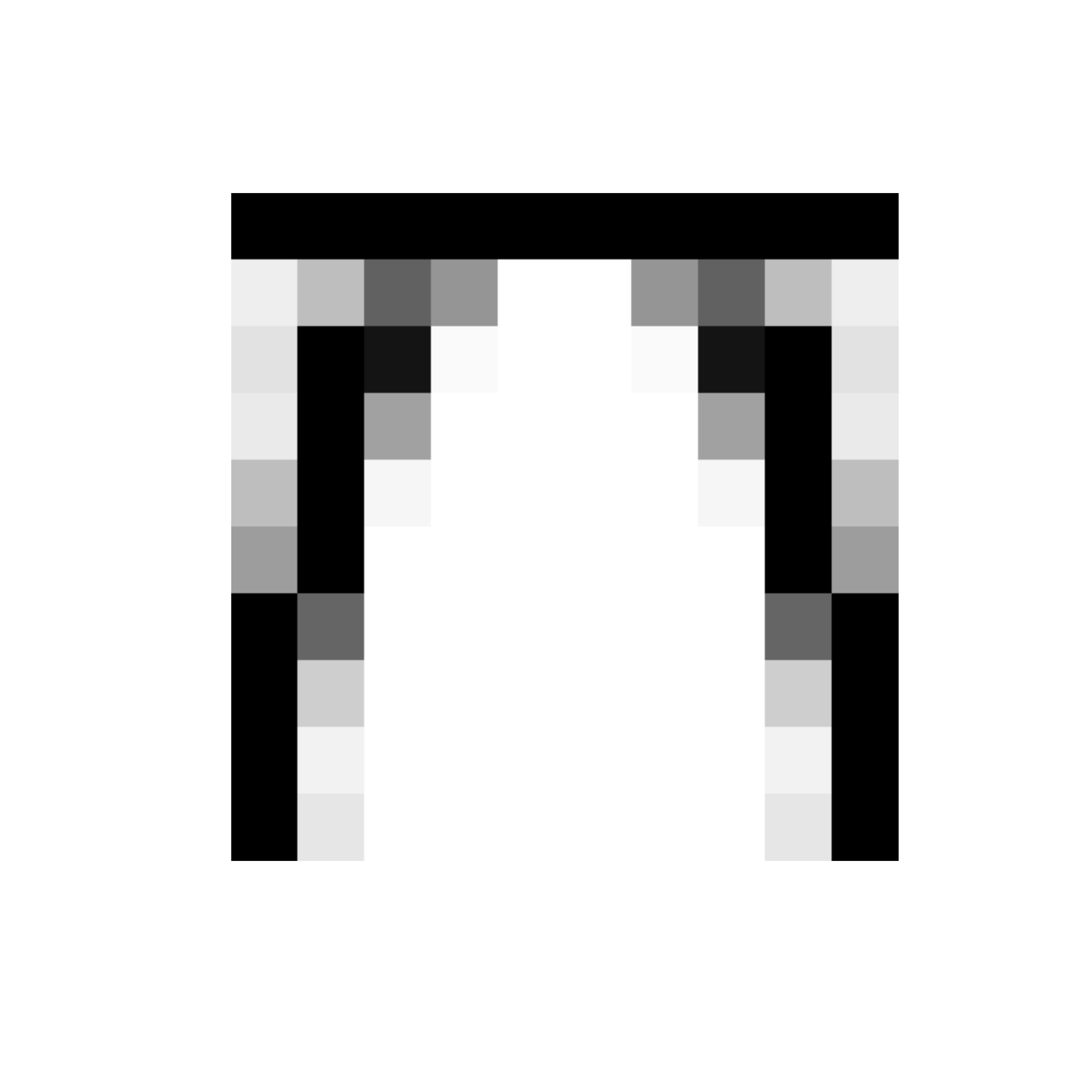}&\hspace{-1.0cm}\includegraphics[width=1.5in]{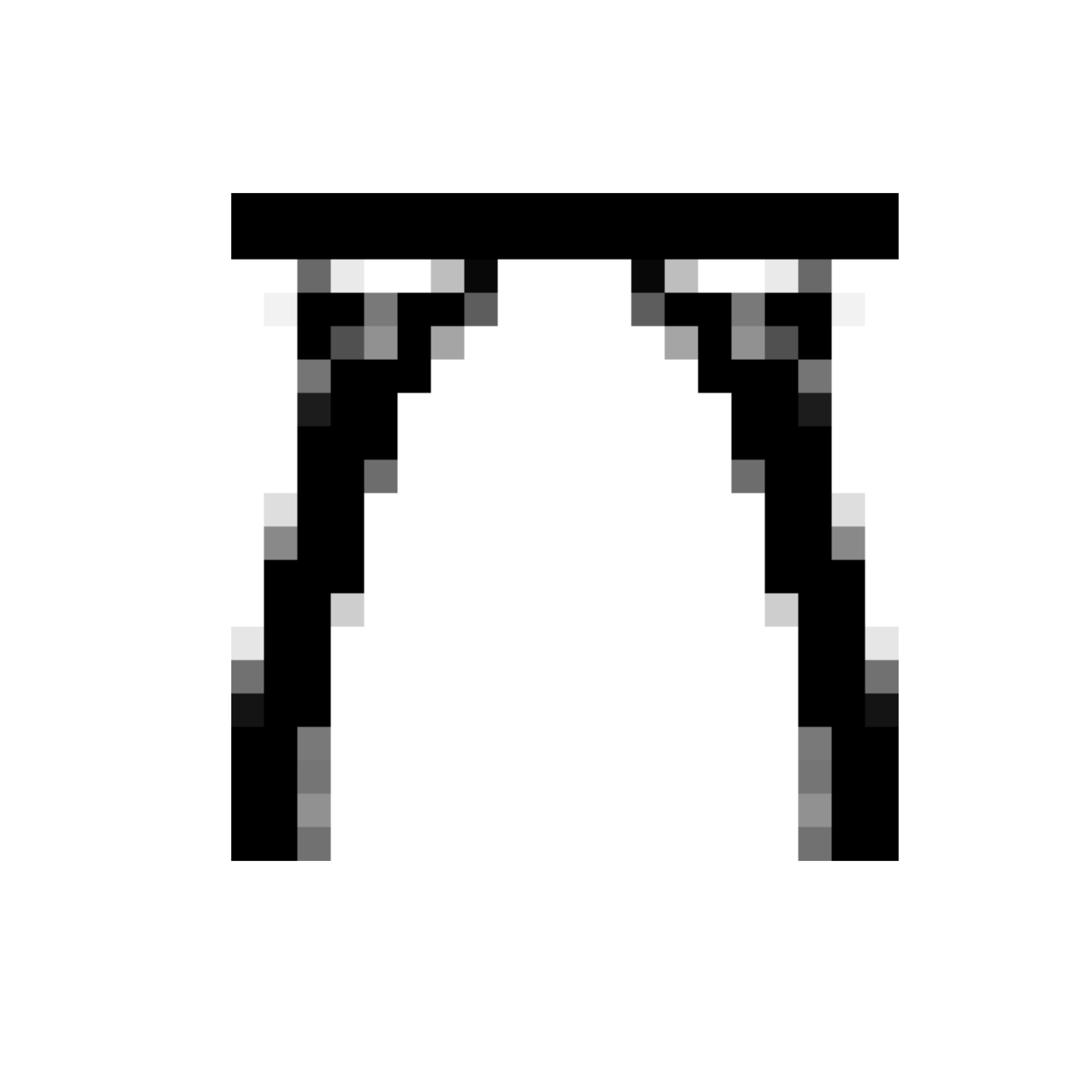}&\hspace{-1.0cm}\includegraphics[width=1.5in]{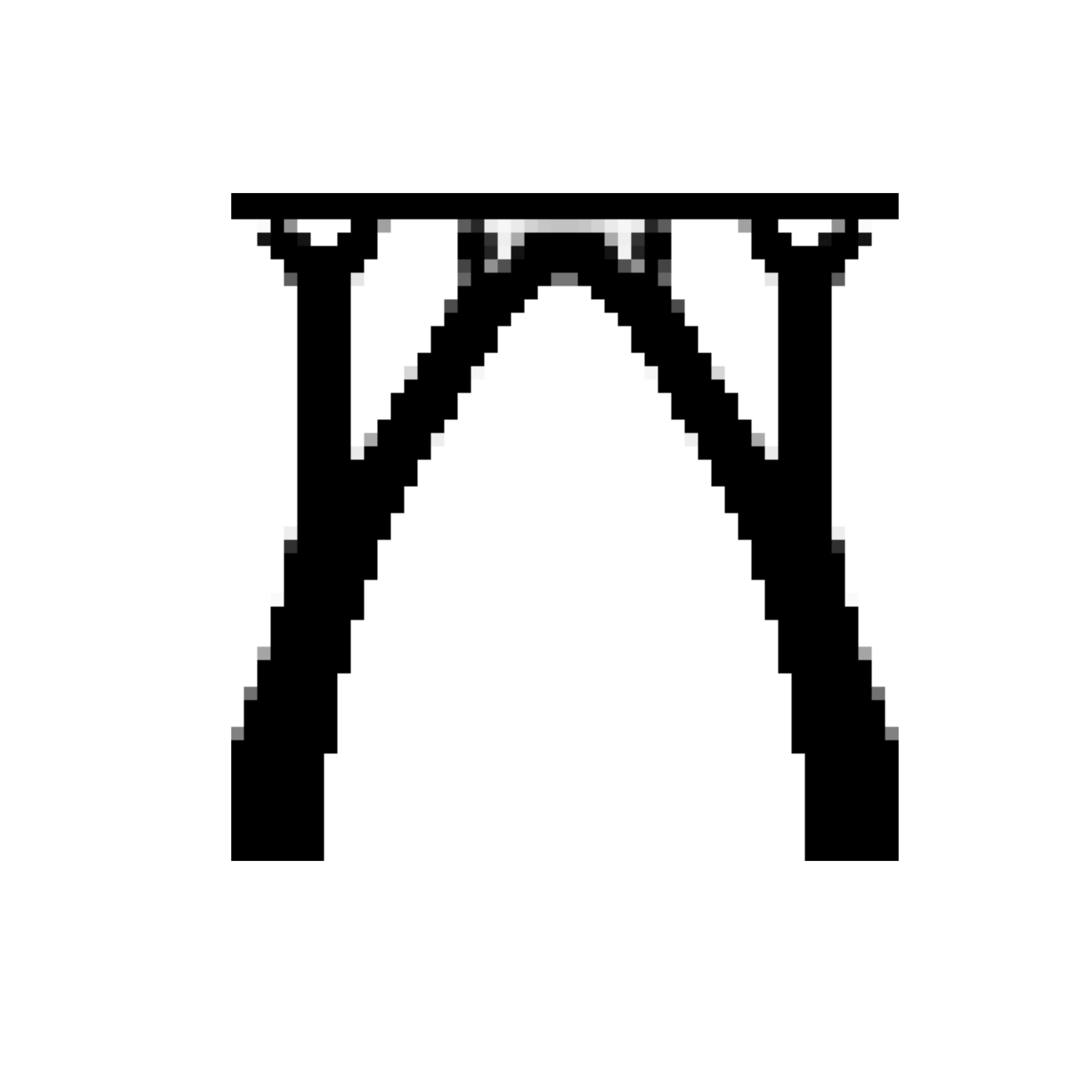}
&\hspace{-1.1cm}\includegraphics[width=1.5in]{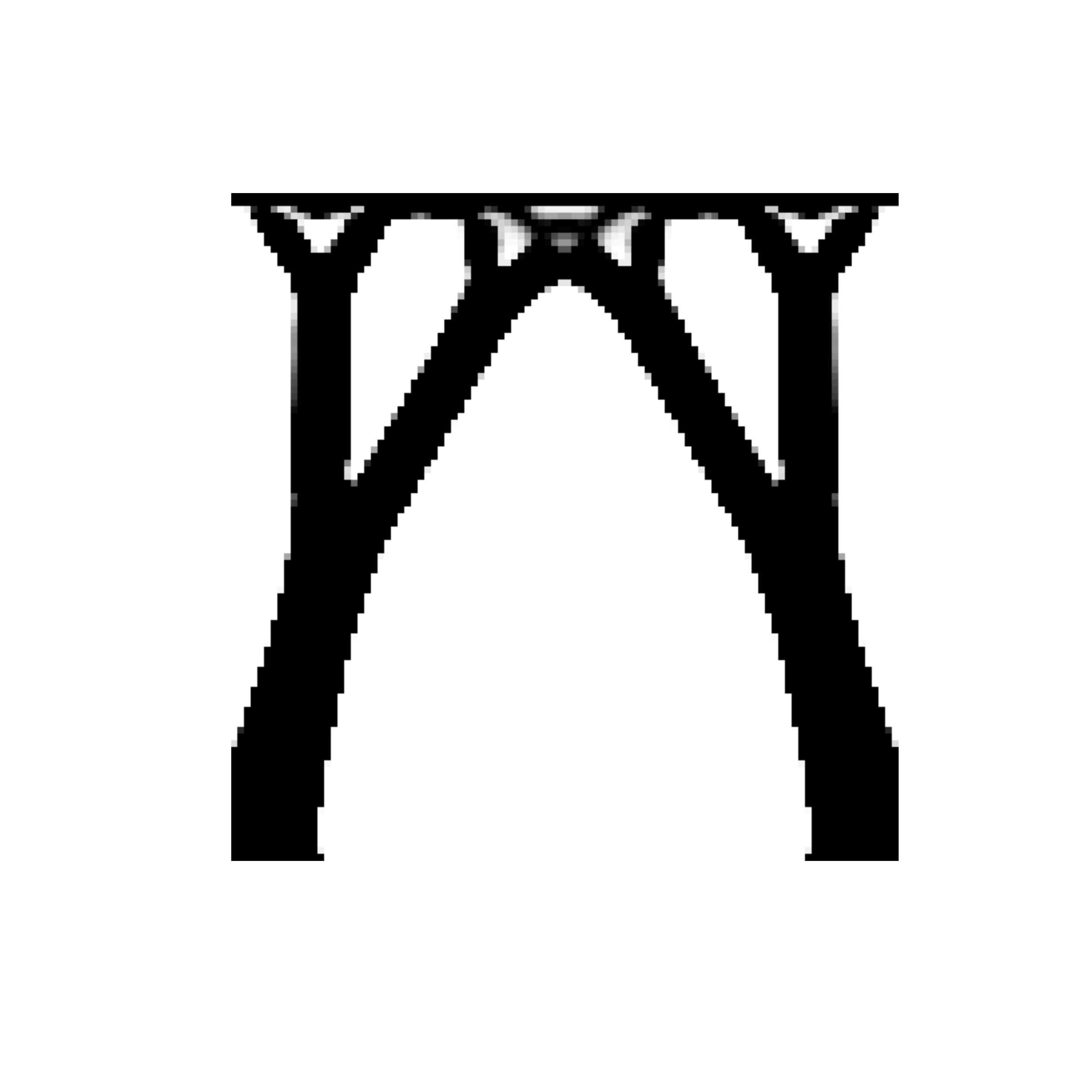}\\
\hspace{-0.4cm}\includegraphics[width=1.5in]{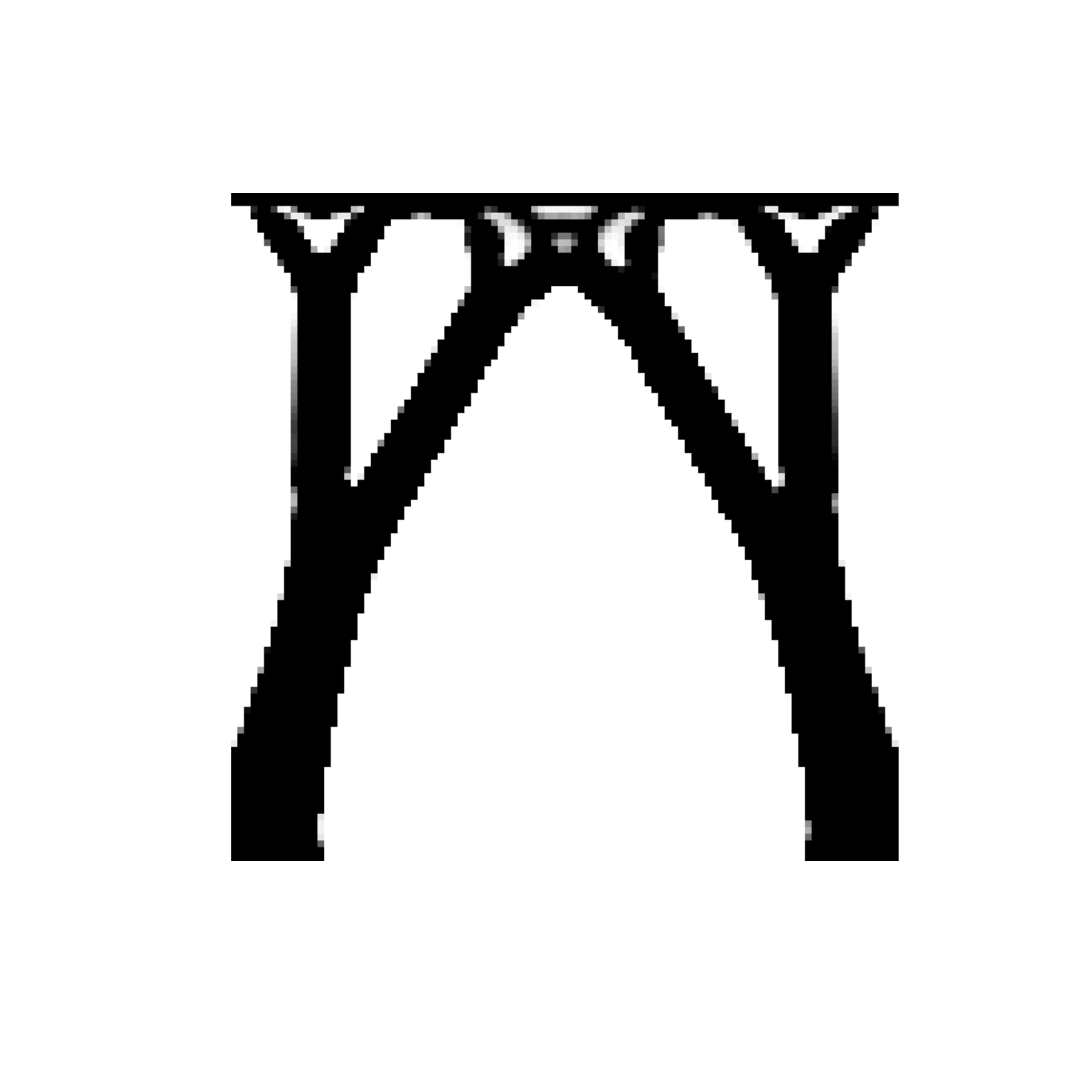}&\hspace{-0.8cm}\includegraphics[width=1.5in]{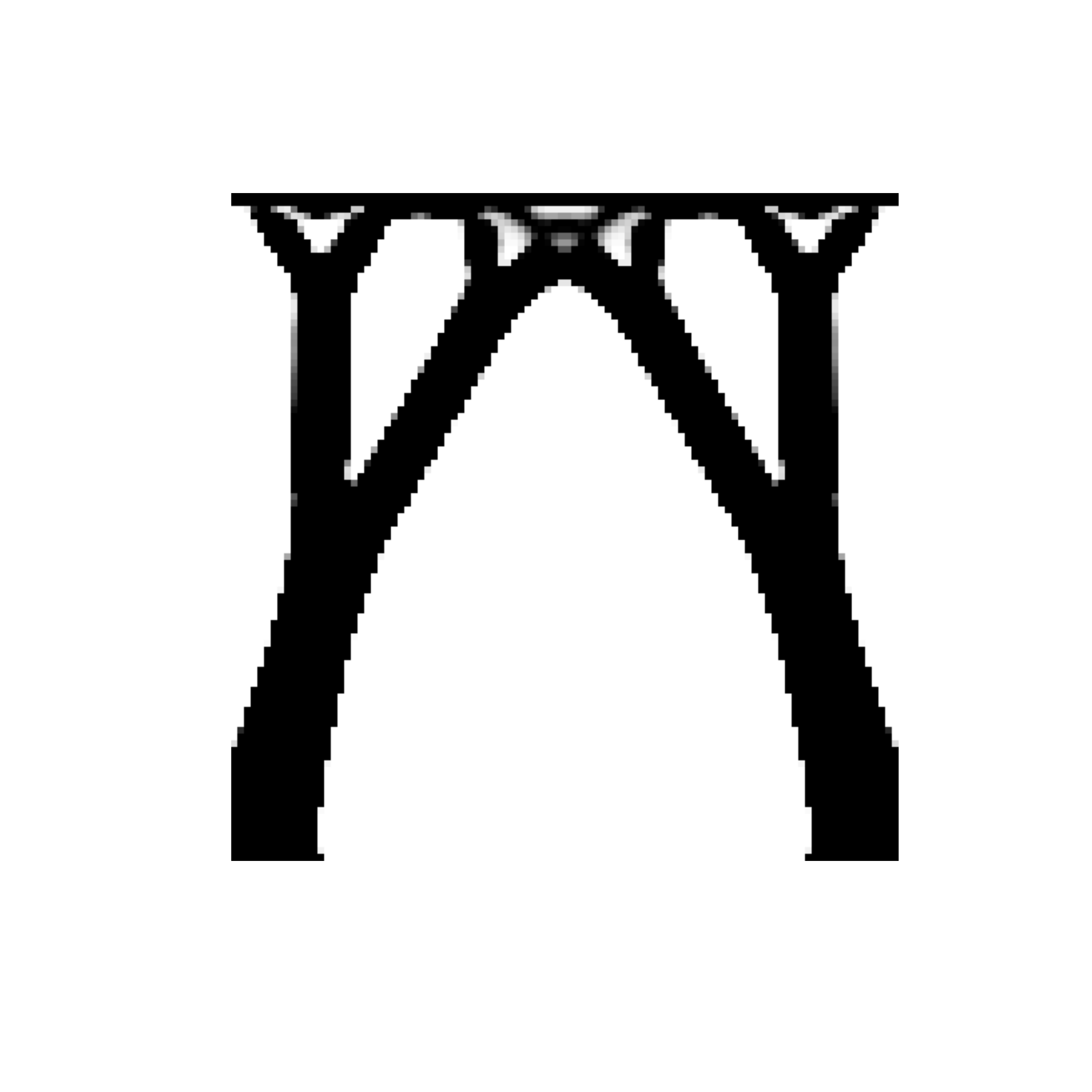}&\hspace{-1.0cm}\includegraphics[width=1.5in]{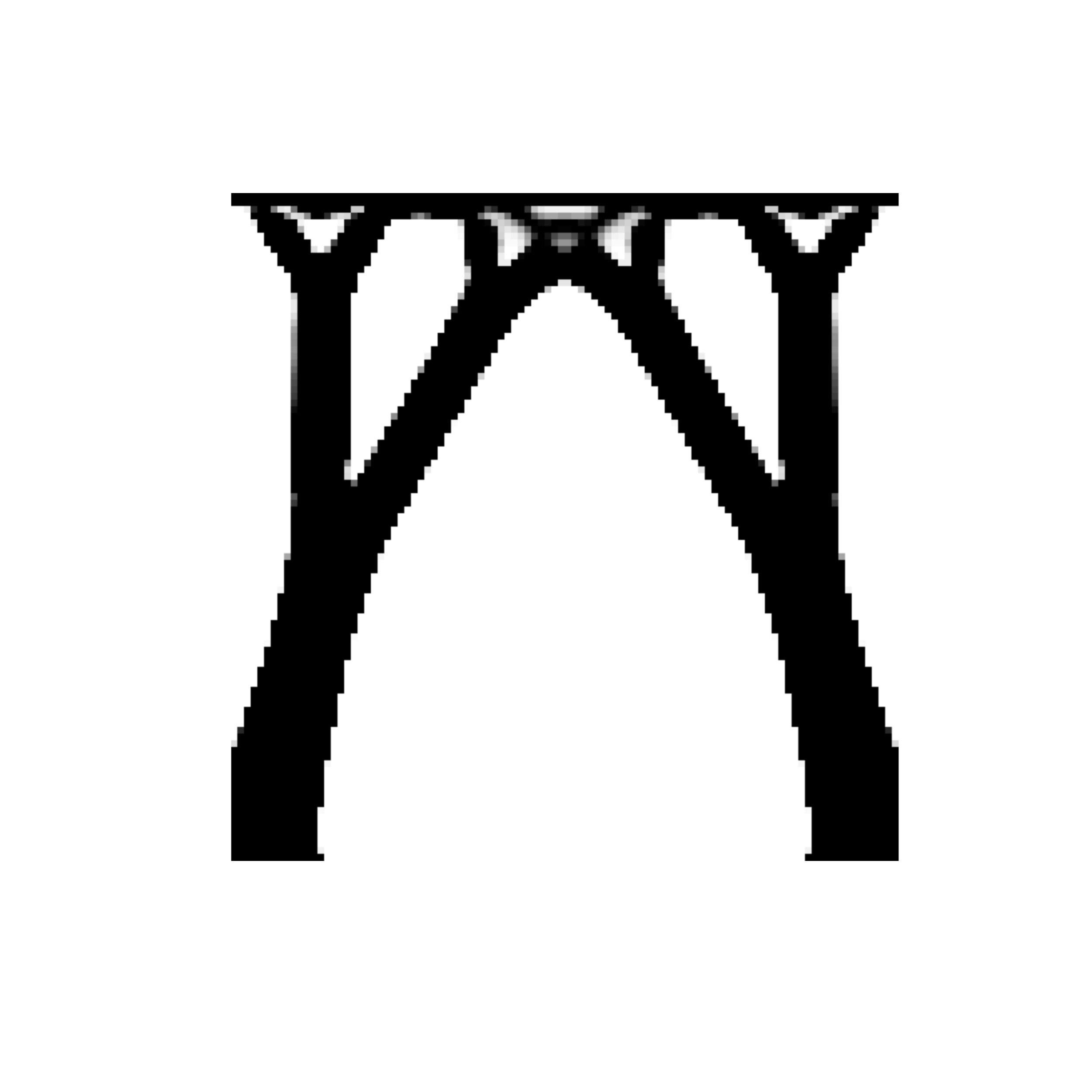}&\hspace{-1.0cm}\includegraphics[width=1.5in]{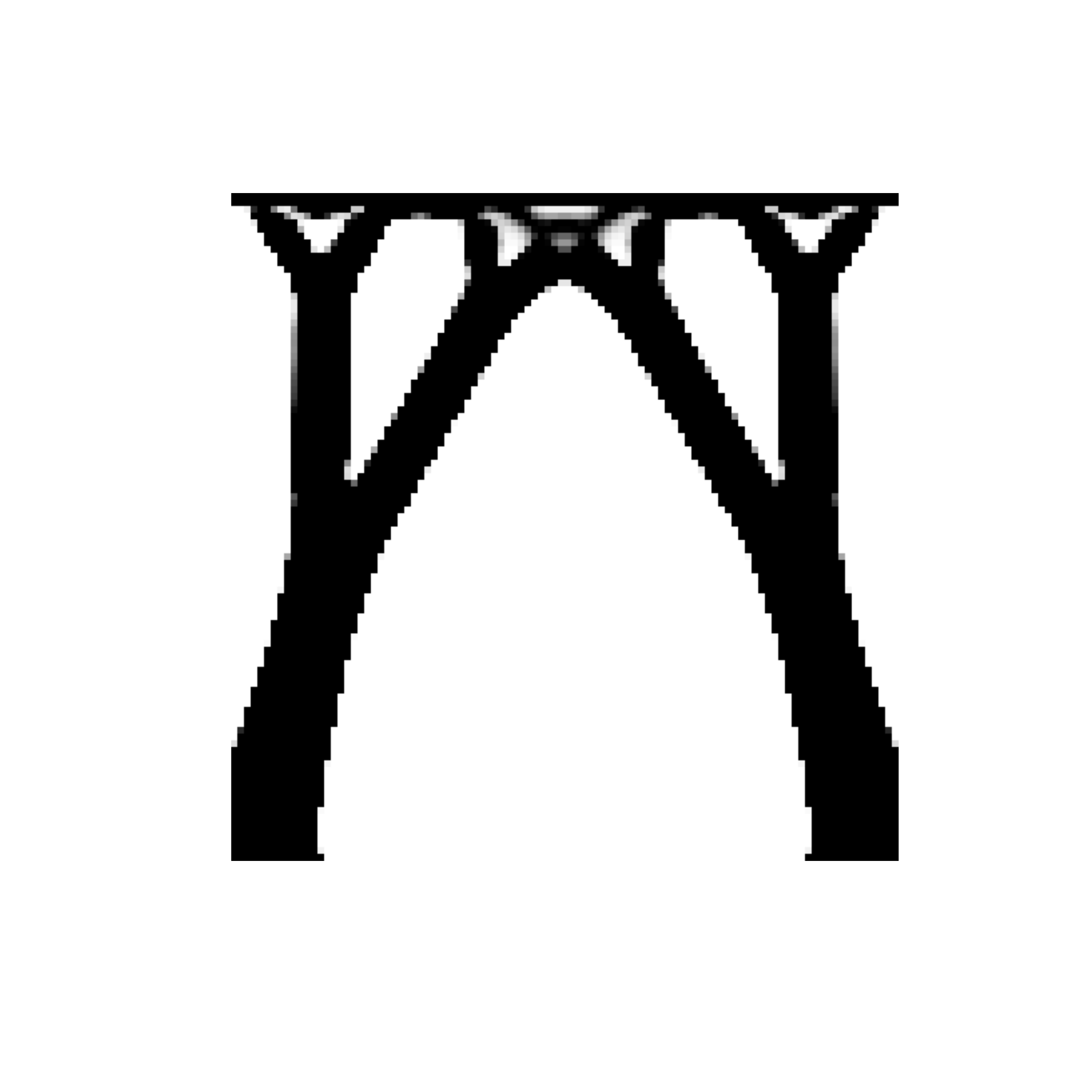}
\end{tabular}
\caption{{Topology optimization results for different meshes: Single resolution optimization with $4 \times 4$, $10 \times 10$, $20 \times 20$, $50 \times 50$ and $100 \times 100$ meshes(top row); Bi-fidelity optimization with $4 \times 4$, $10 \times 10$, $20 \times 20$, $50 \times 50$ meshes (bottom row). The top right figure is obtained with $148$ high-resolution simulations (on $100\times100$ mesh) at each design iterations whereas the bottom plots associated with $10\times 10, 20\times20,50\times50$ are obtained with only $11$ high-resolution simulations. The bottom left plot associated with $4\times4$ mesh is obtained with $6$ high-resolution simulations.}}\label{fign_vary_mesh}
\end{figure}

We compute the difference between optimal designs obtained from the bi-resolution approach and high-resolution design as $e_{\bm \rho} = || \bm \rho^{B} - \bm \rho^H||/\sqrt{n^H_{elem}}$ where $n^H_{elem}=10^4$ in this case. Similarly we define the error in the objective function $Q$ cf. Equation~\eqref{to1_0_1Q} as  $e_{Q} = | Q^{B} - Q^H|/{Q^H}$. Table~\ref{tabNEload_1} shows the number of iterations, number of high resolution simulation which is $6$ for the coarsest mesh and $11$ for the rest of meshes at each iteration, as well as $e_{\bm \rho}$ and $e_{Q}$. It is apparent that the bi-resolution topology optimization with $10 \times 10$ mesh yields almost the same design with much smaller cost.


\begin{table}[!h]
\caption{Loading uncertainty: Error vs cost for single and bi-resolution optimization.}
\centering
\begin{tabular}{l c c c c}
  \hline\hline
Resolution & No. Iter. & No. Hi. Res. Sim. & $e_{\bm \rho}$  & $e_Q$  \\
\hline
Hi. Res. $100 \times 100$    & 385 & 56980 & - & - \\
Bi-Res. $4 \times 4$   & 587 & 3522 &   1.09e-01 & 2.32e-02 \\
{\color{blue} Bi-Res. $10 \times 10$ }  & 385 & 4235 & 5.74e-05 & 2.01e-05\\
Bi-Res. $20 \times 20$   & 385 & 4235 & 5.74e-05 & 2.01e-05\\
Bi-Res. $50 \times 50$   & 385 & 4235 & 5.74e-05 & 2.01e-05\\
\hline\hline
  \end{tabular}
  \label{tabNEload_1}
\end{table}

To investigate the effect of standard deviation weight on the optimal design, we consider three values for $\lambda=0.001,~0.1,~1$ cf. Equation~\eqref{to1_0_1Q}. Figure~\ref{fign_lambda_Q} shows the optimization iteration for both single and bi-resolution which are almost identical for different values of $\lambda$. We show the corresponding designs in Figure~\ref{fign_lambda_design} where again similar topologies are obtained.

\begin{figure}[!h]
\centering
\includegraphics[width=5.5in]{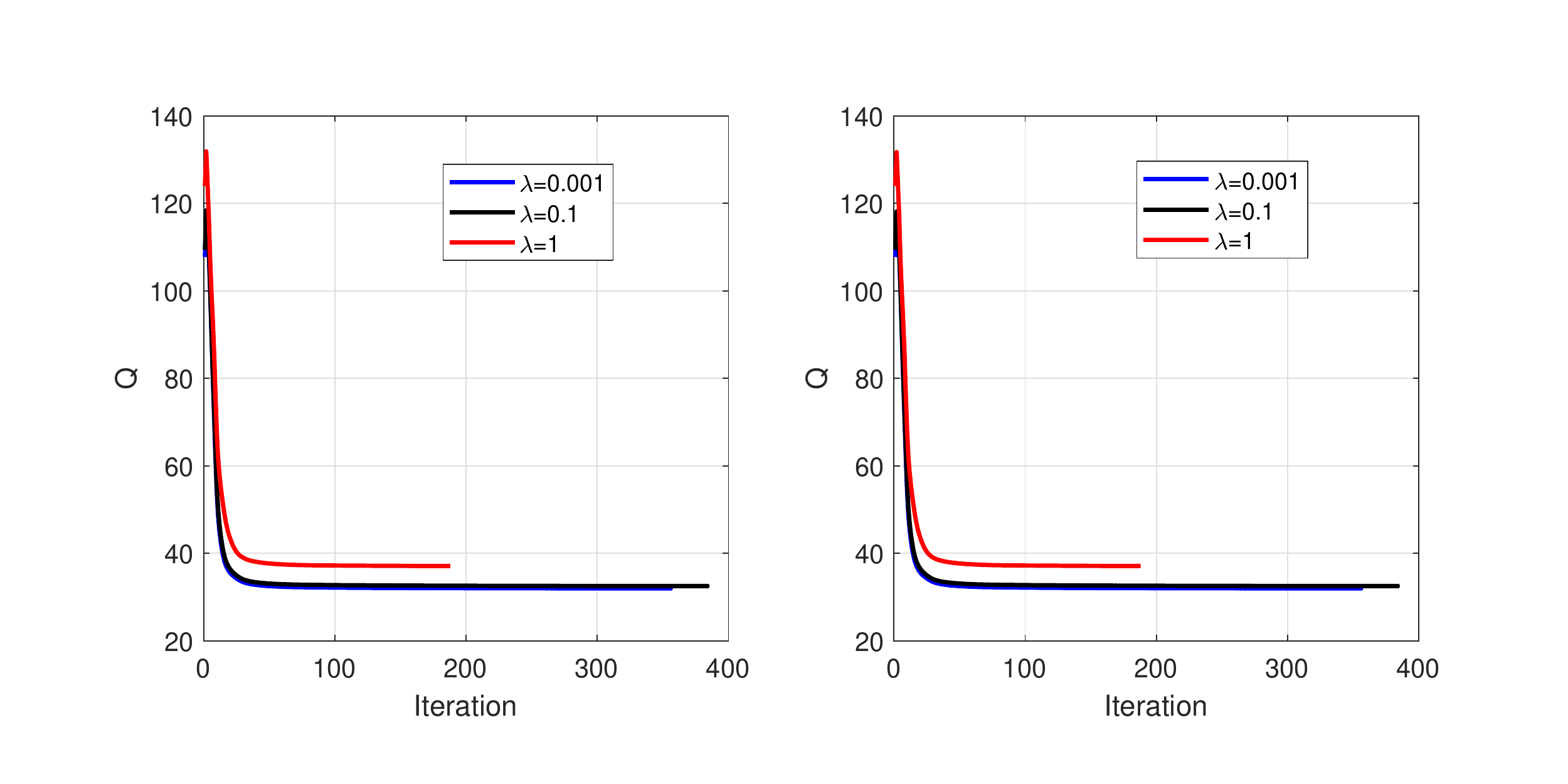}
\caption{{Effect of $\lambda$ on optimized objective function $Q$. Left and right figures are obtained from single high resolution $100 \times 100$ optimization and bi-fidelity optimization with $10 \times 10$ mesh respectively.}}\label{fign_lambda_Q}
\end{figure}

\begin{figure}[!h]
\begin{tabular}{c c c }
\vspace{-0.5cm}
\hspace{1.5cm}\includegraphics[width=1.5in]{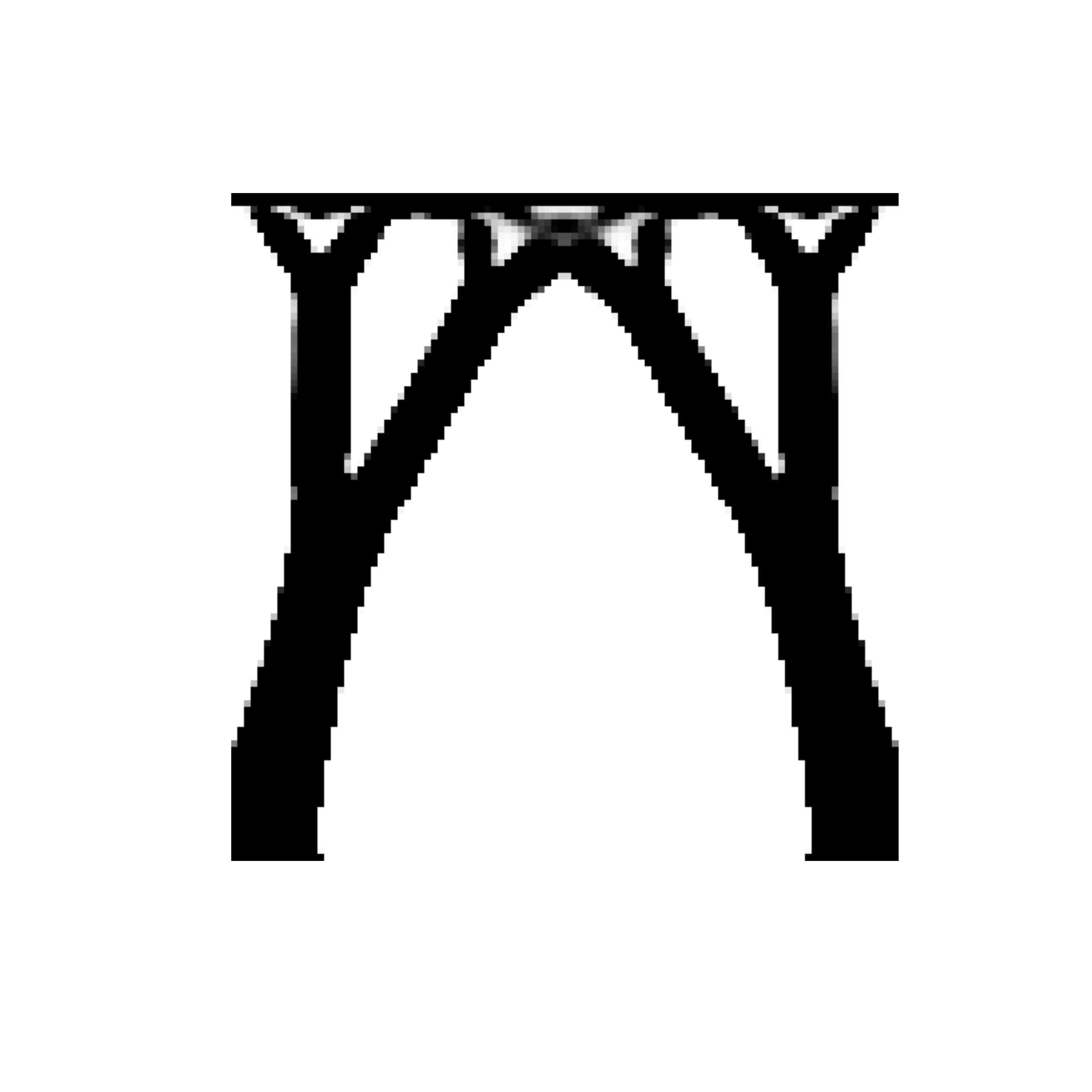}&\includegraphics[width=1.5in]{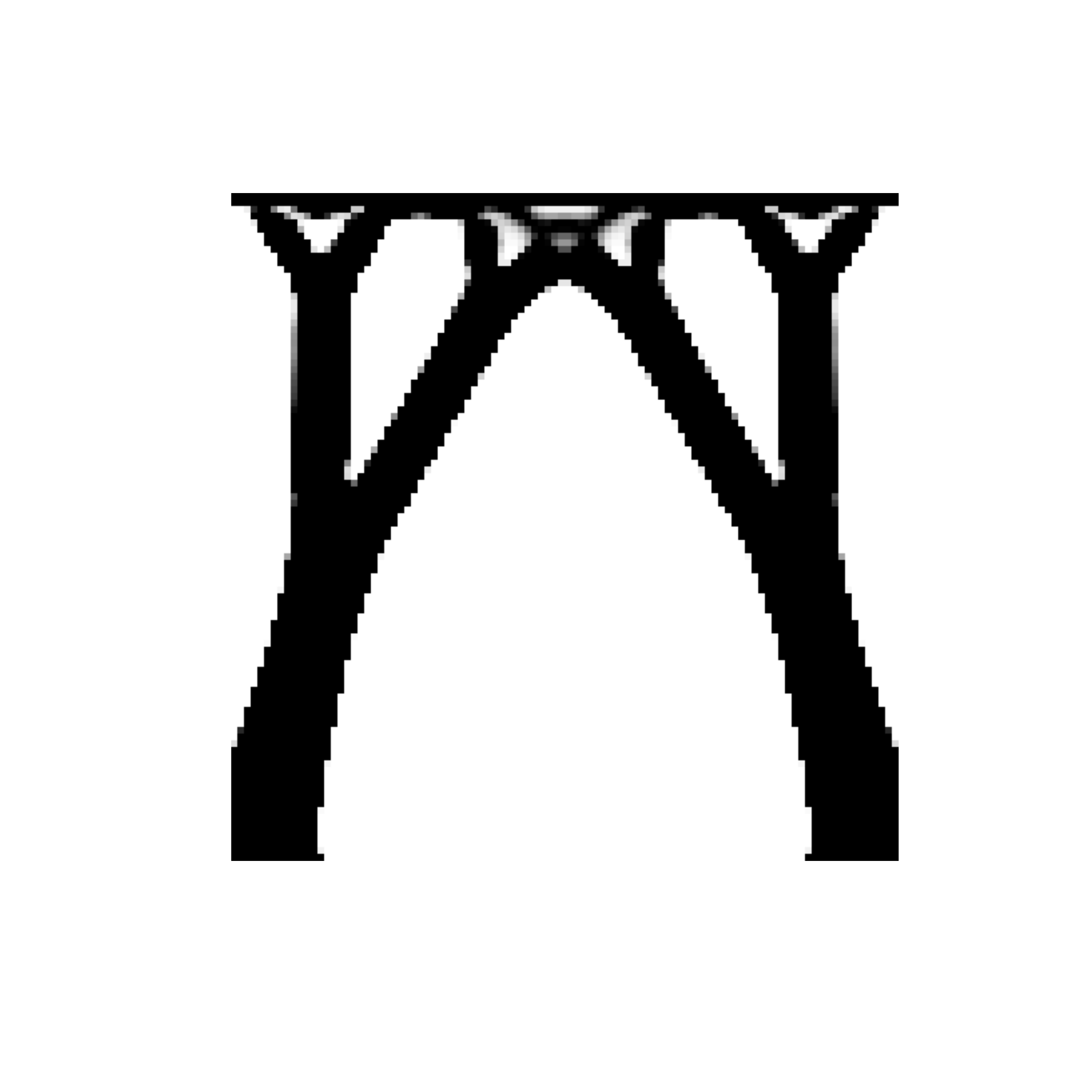}&\includegraphics[width=1.5in]{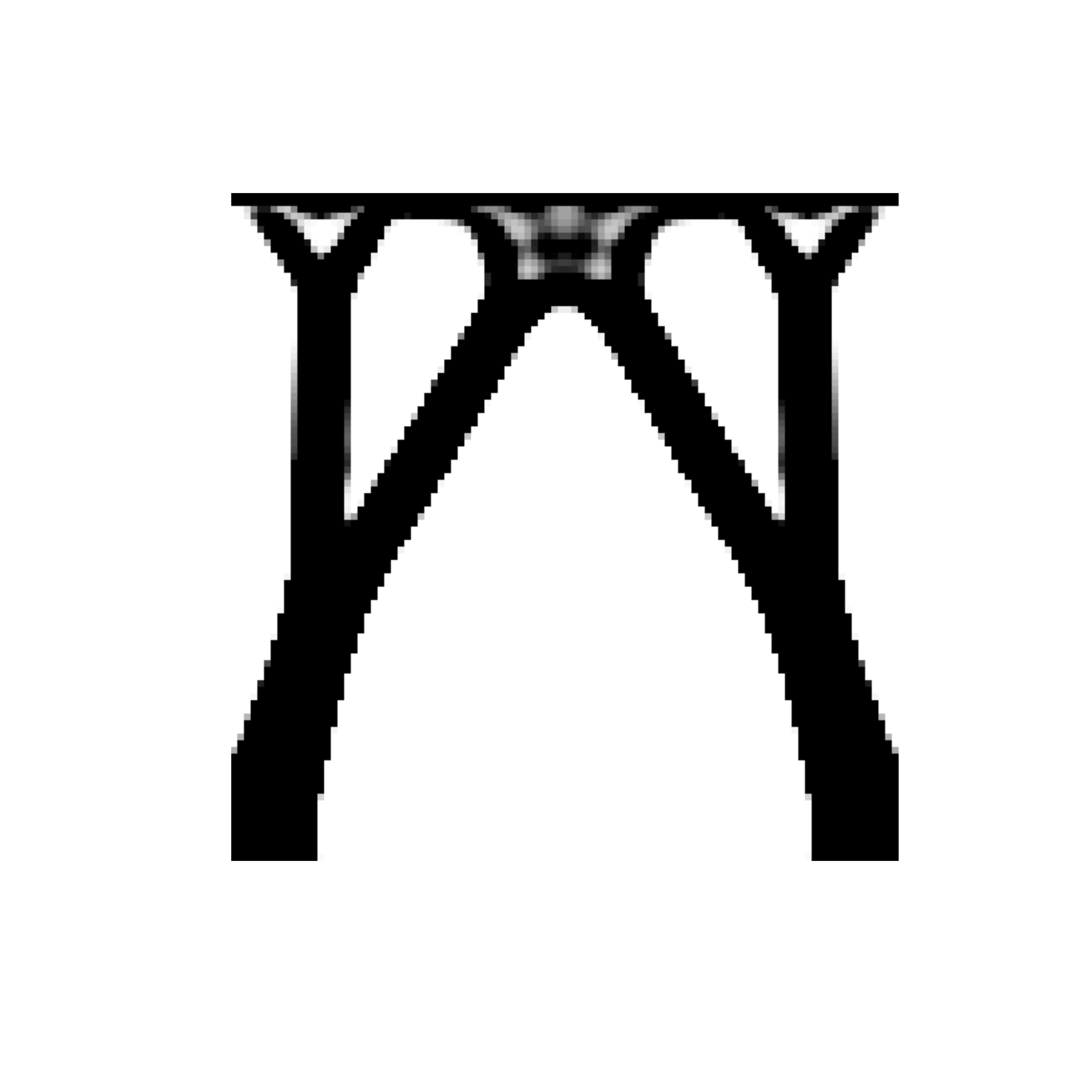}\\
\hspace{1.5cm}\includegraphics[width=1.5in]{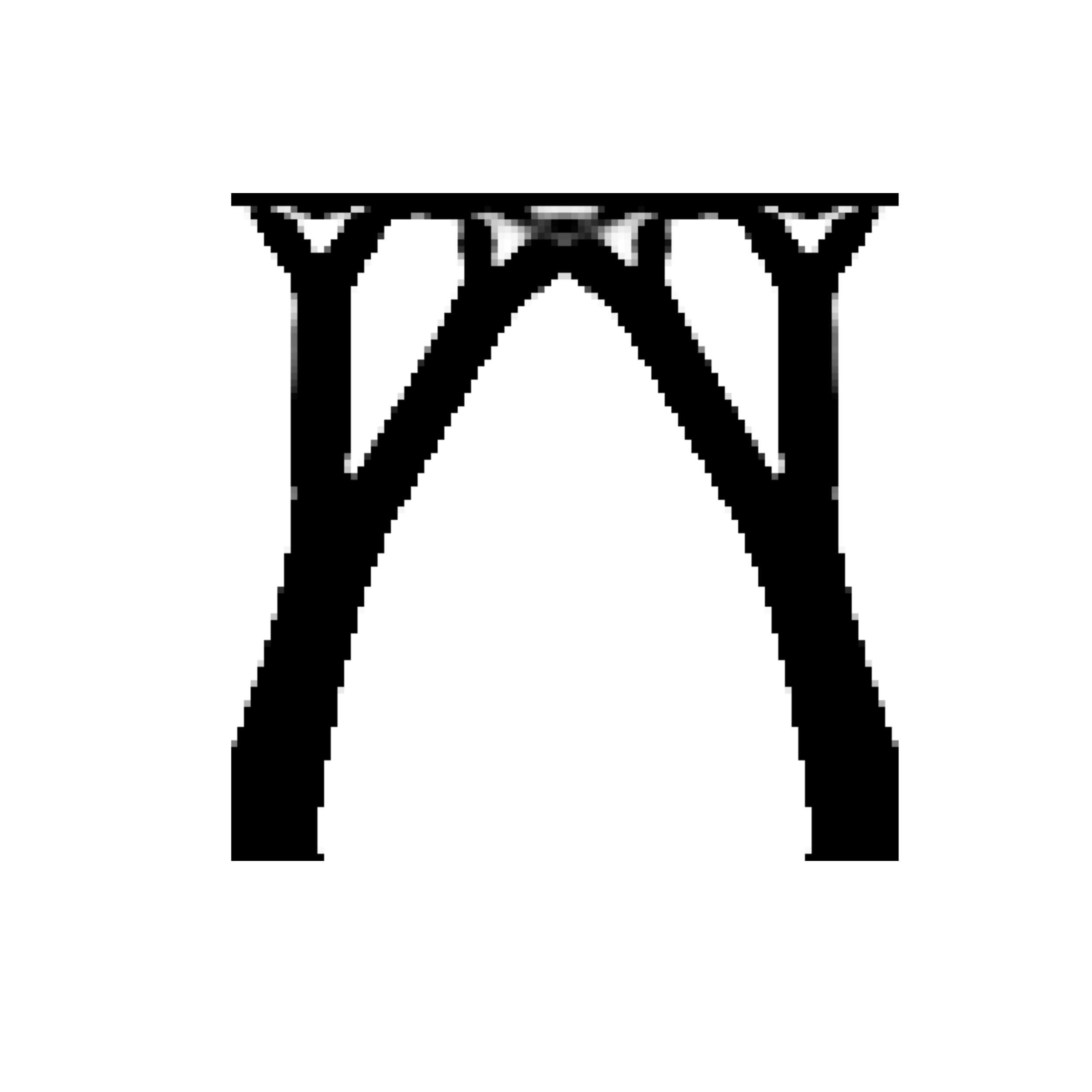}&\includegraphics[width=1.5in]{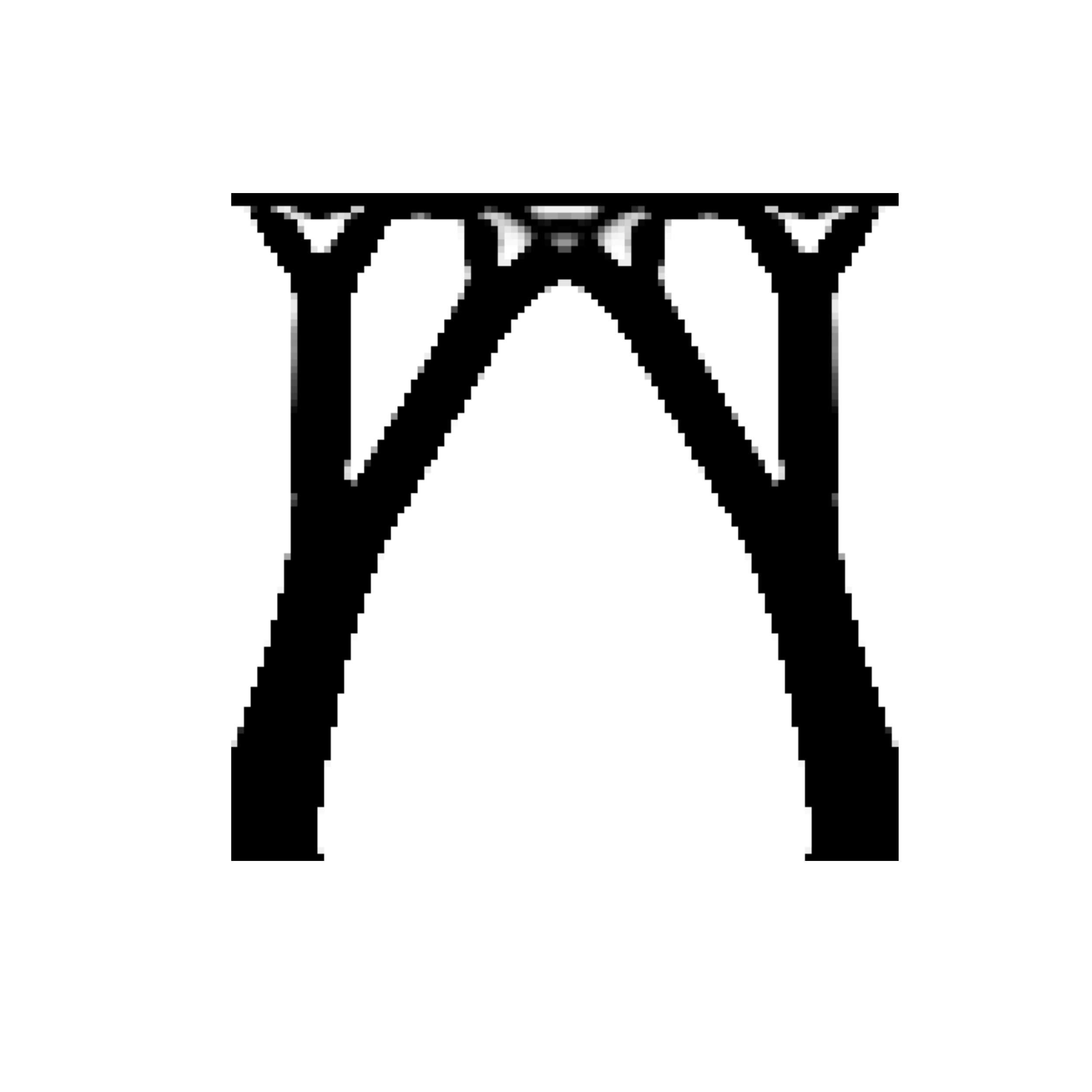}&\includegraphics[width=1.5in]{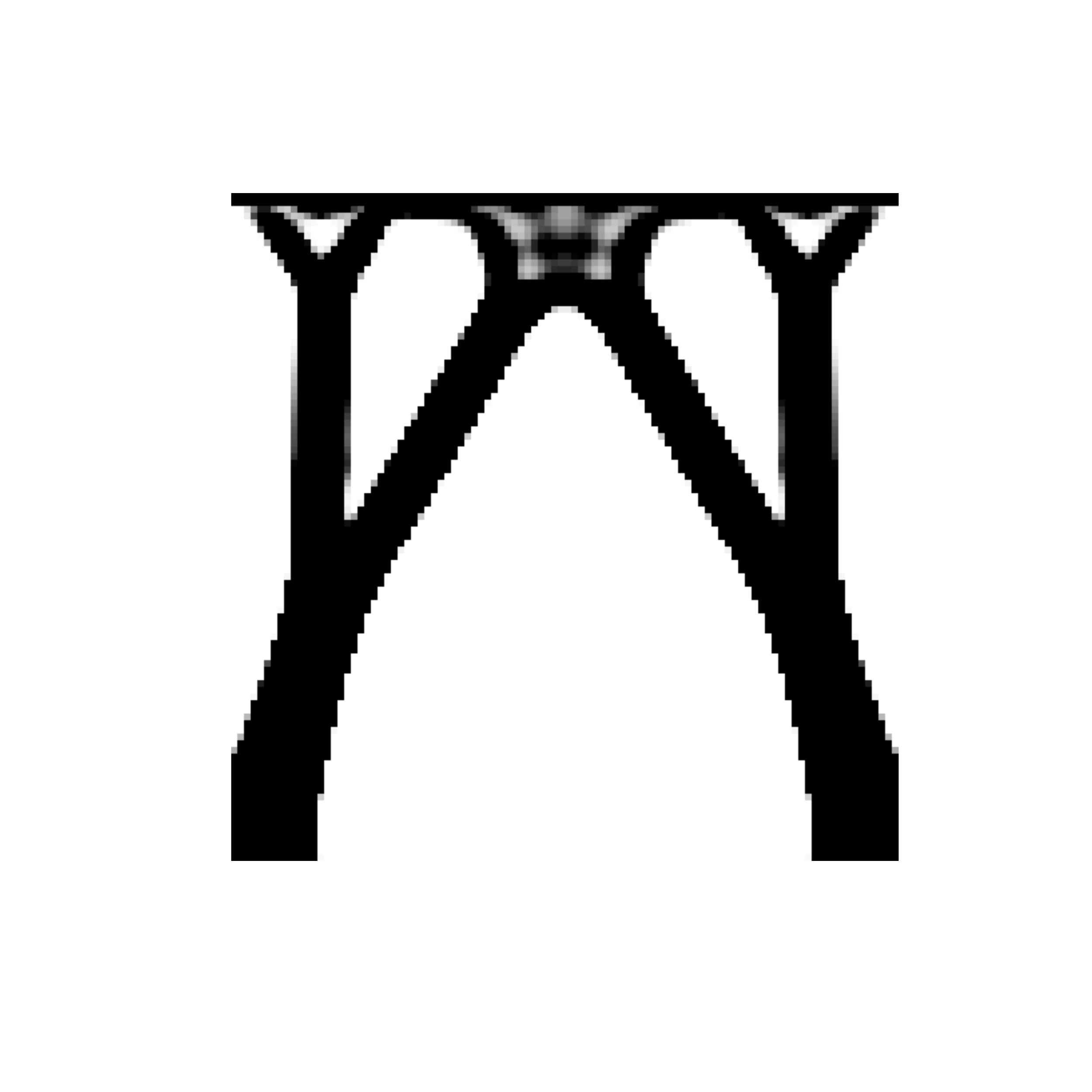}\\
\end{tabular}
\caption{{Effect of $\lambda$ on optimal design: $\lambda=0.001$ (left),  $\lambda=0.1$ (middle),  $\lambda=1$ (right). Top and bottom figures are obtained from single high resolution $100 \times 100$ optimization and bi-fidelity optimization with $10 \times 10$ mesh respectively.}}\label{fign_lambda_design}
\end{figure}

Finally we compute the error bound in approximation of displacement, compliance and compliance sensitivity. To that end, we consider the first iteration where the densities are considered uniformly $\bm \rho=0.35$. We also consider $n=11$ with $10 \times 10$ mesh as the full rank of the low fidelity model. As mentioned earlier to obtain $\delta$ we directly compute the two norms in~\eqref{eq_estimate_1} for the first few unimportant samples. The maximum ratio is computed to be $\delta=0.916$ for the third sample after $n=11$ samples. The maximum norm for the stiffness matrix and its derivative with respect to the first design variable (the element on the bottom left corner) are $\sigma_{max}(\bm K)=1.0476,~ \sigma_{max}(\partial \bm K/\partial \rho_1)=1.0714$. We have also computed $\sigma_{max}(\bm G^H)=113.384$ and $\epsilon=2.825e-06$ directly from the high-fidelity data. The actual and estimated errors for the aforementioned sample are listed in Table~\ref{tabNEload_estimate}.

\begin{table}[!h]
 \caption{Actual error vs upper bound estimate. Estimated values are computed as the bounding certificates in Lemma \ref{lemma1} and Proposition \ref{prop:C-error}.}
\centering
\begin{tabular}{l c c c}
  \hline\hline
 & $\| \bm u - \hat{\bm u}\|$ & $|C-\hat{C}|$   & $\| \partial C/\partial \rho_1 -  \partial \hat{C}/\partial \rho_1 \| $ \\
 \hline
Actual     & 3.64e-07 &  1.57e-08 & 3.18e-09\\
Estimate & 5.76e-05  & 3.26e-04 &3.33e-04\\

\hline\hline
  \end{tabular}
  \label{tabNEload_estimate}
\end{table}
From this single point it is evident that the upper bound is relatively small. The actual error for bi-fidelity surrogate is even smaller which promises almost identical designs for parametric topology optimization as evidenced by Figure~\ref{fign_vary_mesh}.

\subsection{Manufacturing Tolerances}

In this example, we consider uncertainty in the thresholding parameter $\tau$ cf. Equation~\eqref{to3} to mimic the geometric variations in the thickness of resulting truss bars in the L-shape domain shown in Figure~\ref{figLshape}.
\begin{figure}[!h]
\centering
\includegraphics[width=2.5in]{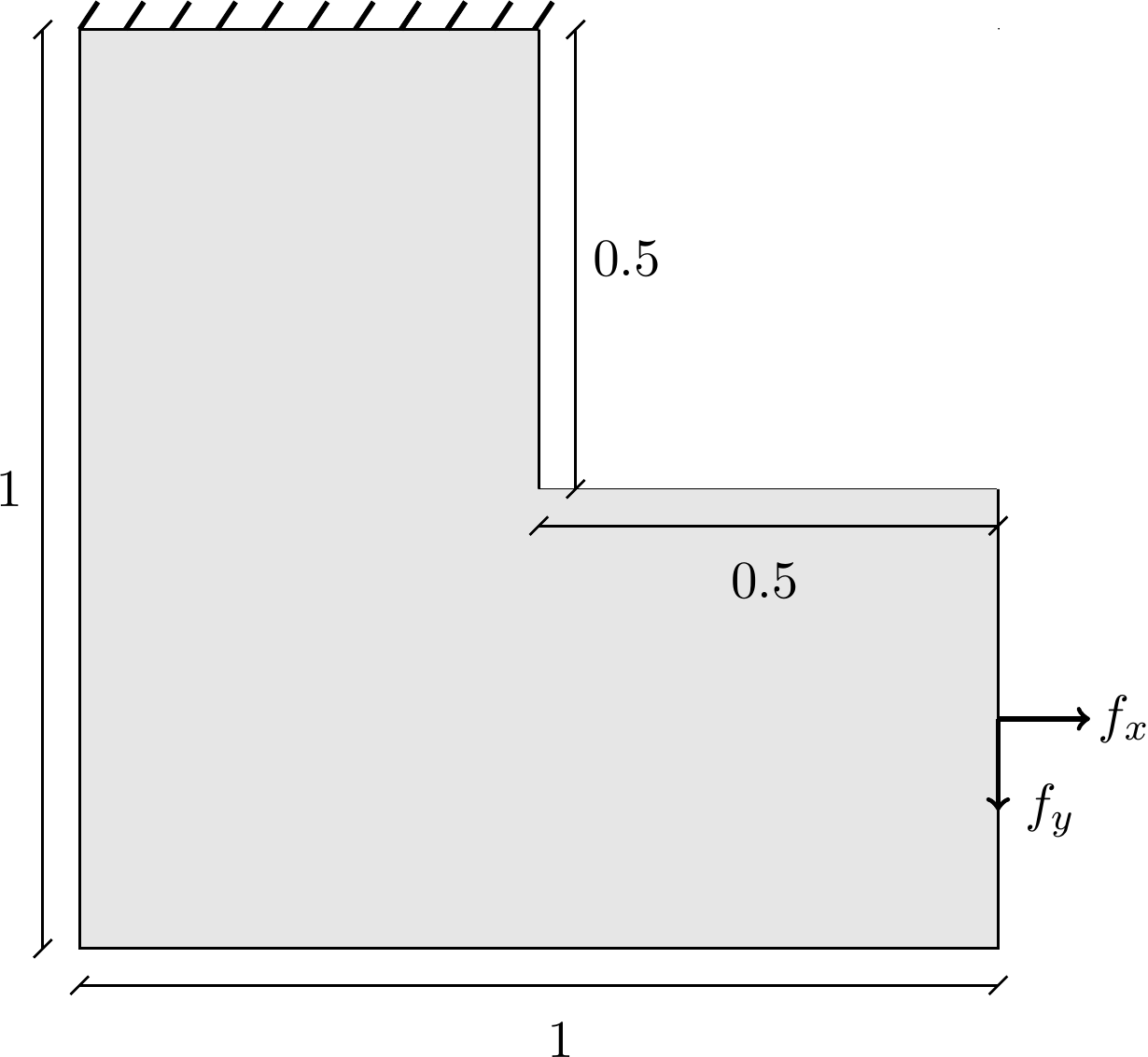}
\caption{{L-Bracket}}\label{figLshape}
\end{figure}
We use $l_c=0.85$ in Equation~\eqref{pcen212} which results in $\sum_{i=1}^{4} \sqrt{\lambda_i}/\sum_{i=1}^{100} \sqrt{\lambda_i}=0.88$. We also use $N=43$ \textit{designed quadrature} points which integrate $d=4$ dimensions with order $\alpha=6$ accurately~\cite{Keshavarzzadeh2018}. The $n_M=4$ Karhunen-Loeve modes and $4$ realizations of spatial threshold $\tau=0.1Z+0.45$ corresponding to arbitrary quadrature nodes are shown in Figure~\ref{fign_KLE_K}.

\begin{figure}[!h]
\centering
\includegraphics[width=6.5in]{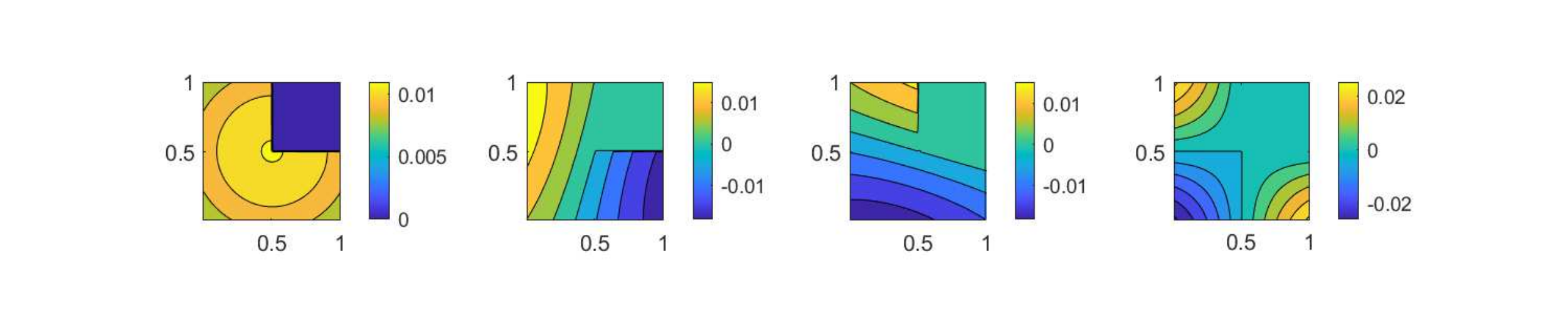}
\vspace{-0.50cm}
\includegraphics[width=6.5in]{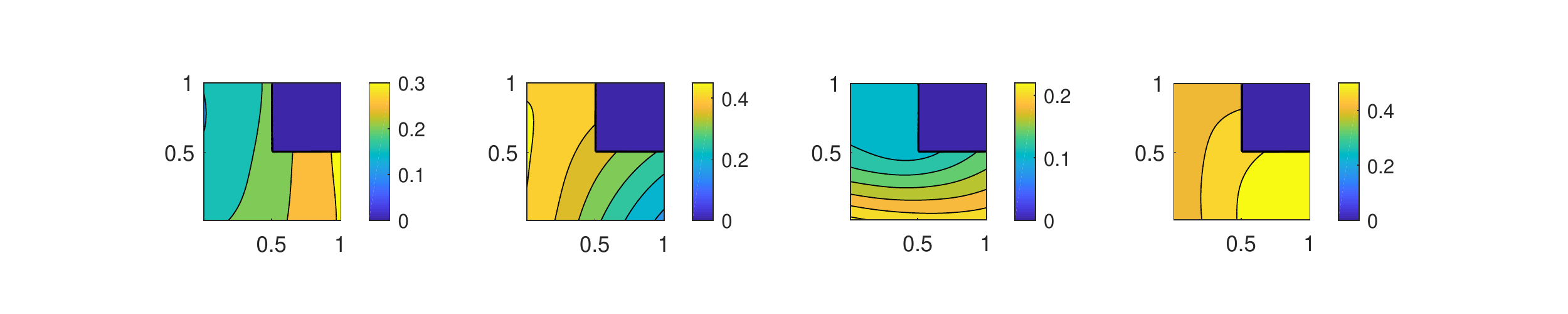}
\caption{{Karhunen-Loeve modes for spatial random field (top), different realizations of spatial threshold on quadrature points (bottom).}}\label{fign_KLE_K}
\end{figure}

Similarly to loading uncertainty, we consider different number of high-resolution simulations $n$ in the bi-fidelity construction and show the difference between bi-fidelity approximation and high-fidelity solutions for the displacement, compliance and compliance sensitivity cf. Figure~\ref{fign_vary_K_n}. Again as expected as the number high resolution simulations increase more accurate bi-fidelity approximations are obtained.

\begin{figure}[!h]
\centering
\includegraphics[width=6.5in]{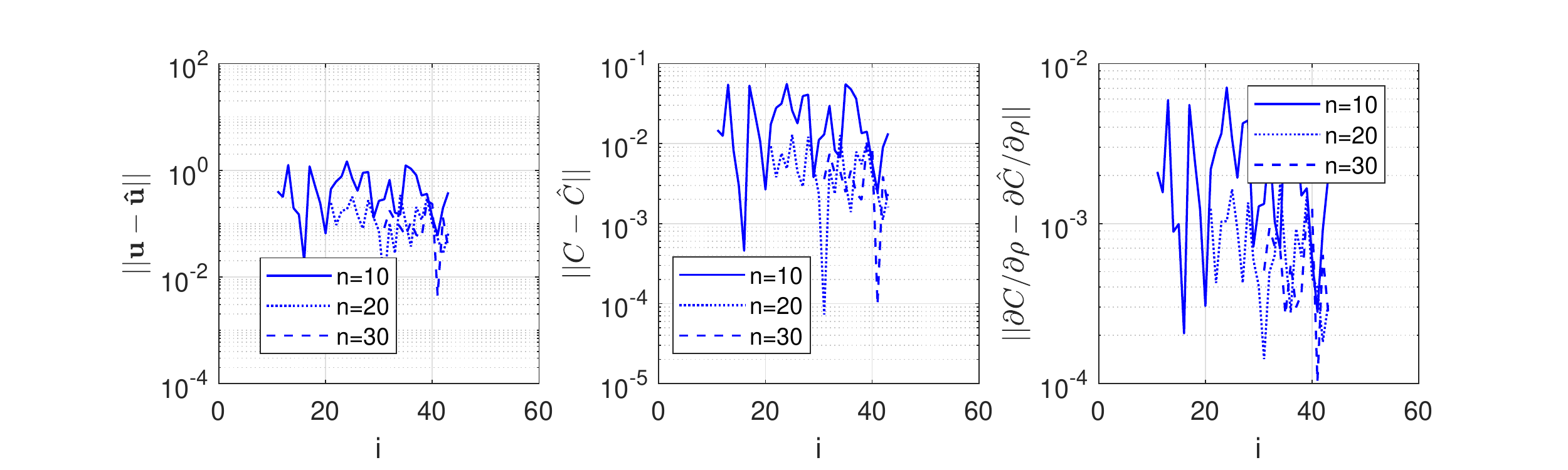}
\vspace{-0cm}
\includegraphics[width=6.5in]{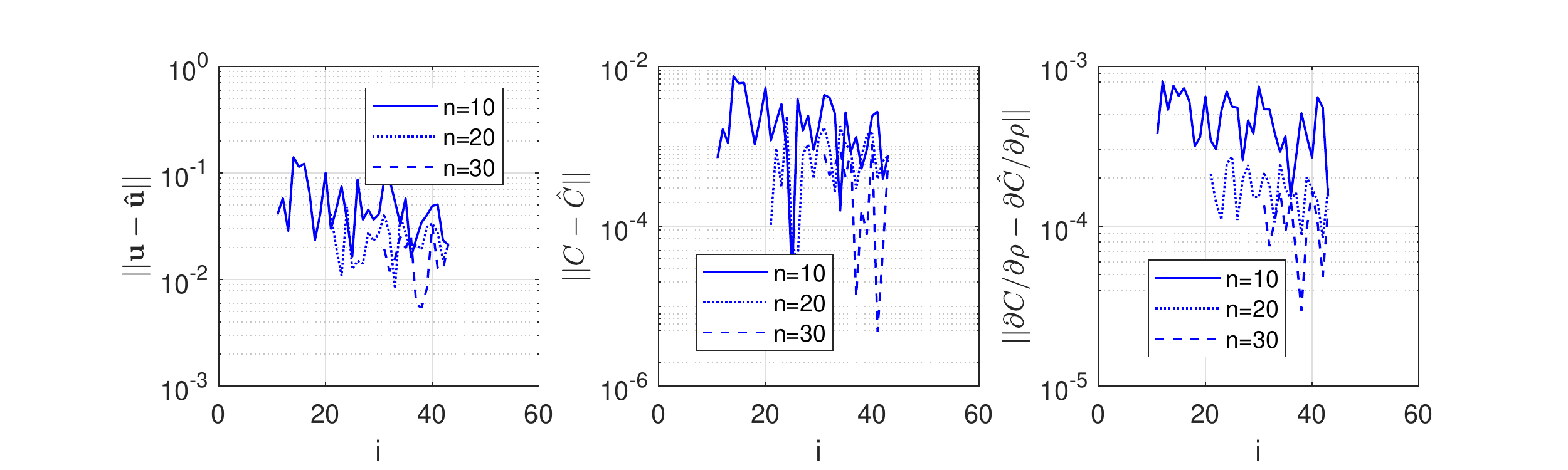}
\caption{{Bi-fidelity actual approximation error for displacement, compliance and compliance sensitivity with respect to different number of high-resolution simulations $n$ for $4 \times 4$ (top) and $10 \times 10$ (bottom) meshes.}}\label{fign_vary_K_n}
\end{figure}

We use filter radius $r_{min}=6,3,2.5,2,1.5$ for different meshes in this case. Topology optimization results for single and bi-resolution models are shown in Figure~\ref{fign_vary_mesh_tol_K}. In the bi-resolution optimizations only $10$ high-resolution simulations
are performed while the single resolution is performed with $43$ simulations.

\begin{figure}[!h]
\begin{tabular}{l l l l l}
\vspace{-0.5cm}
\hspace{-0.4cm}\includegraphics[width=1.5in]{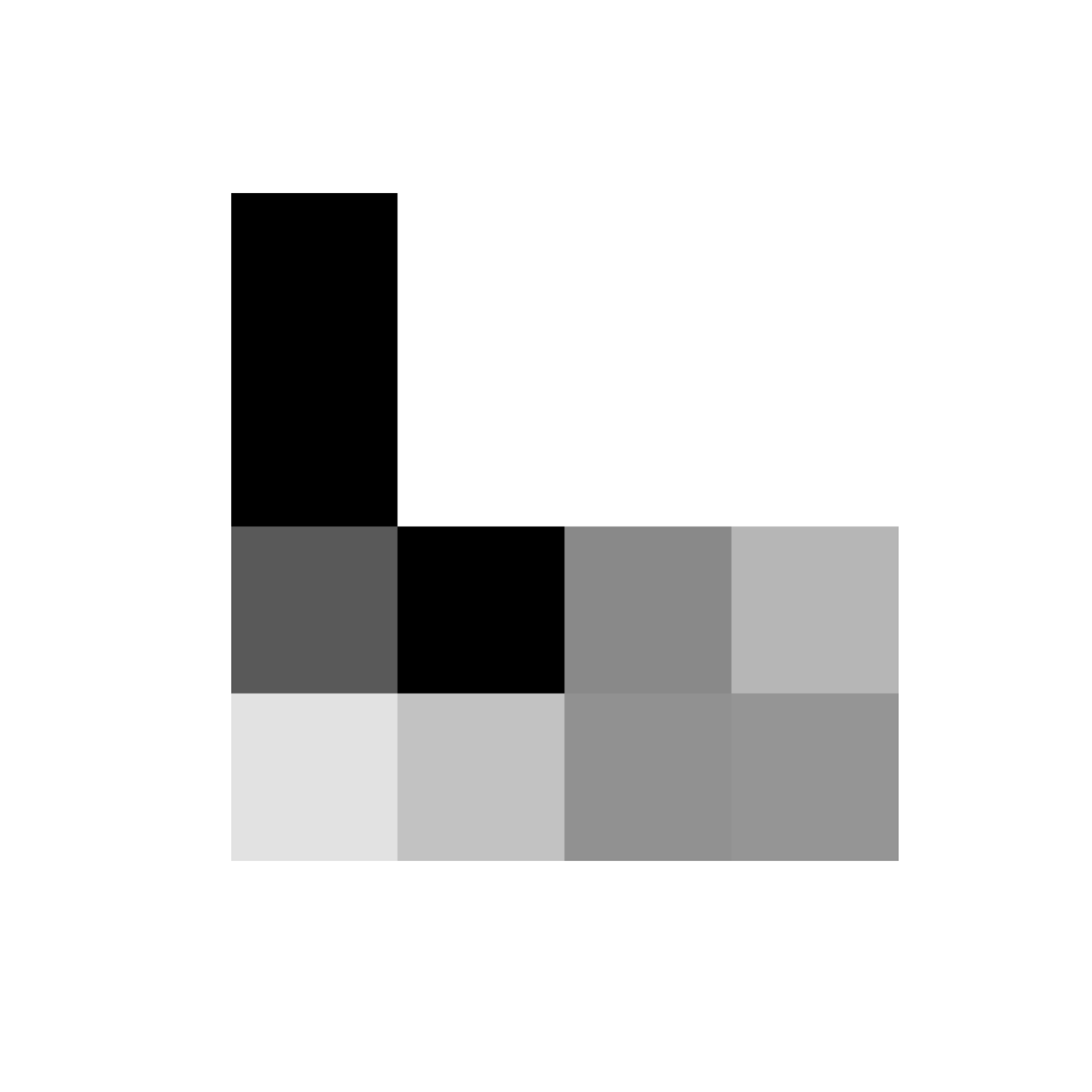}&\hspace{-0.8cm}\includegraphics[width=1.5in]{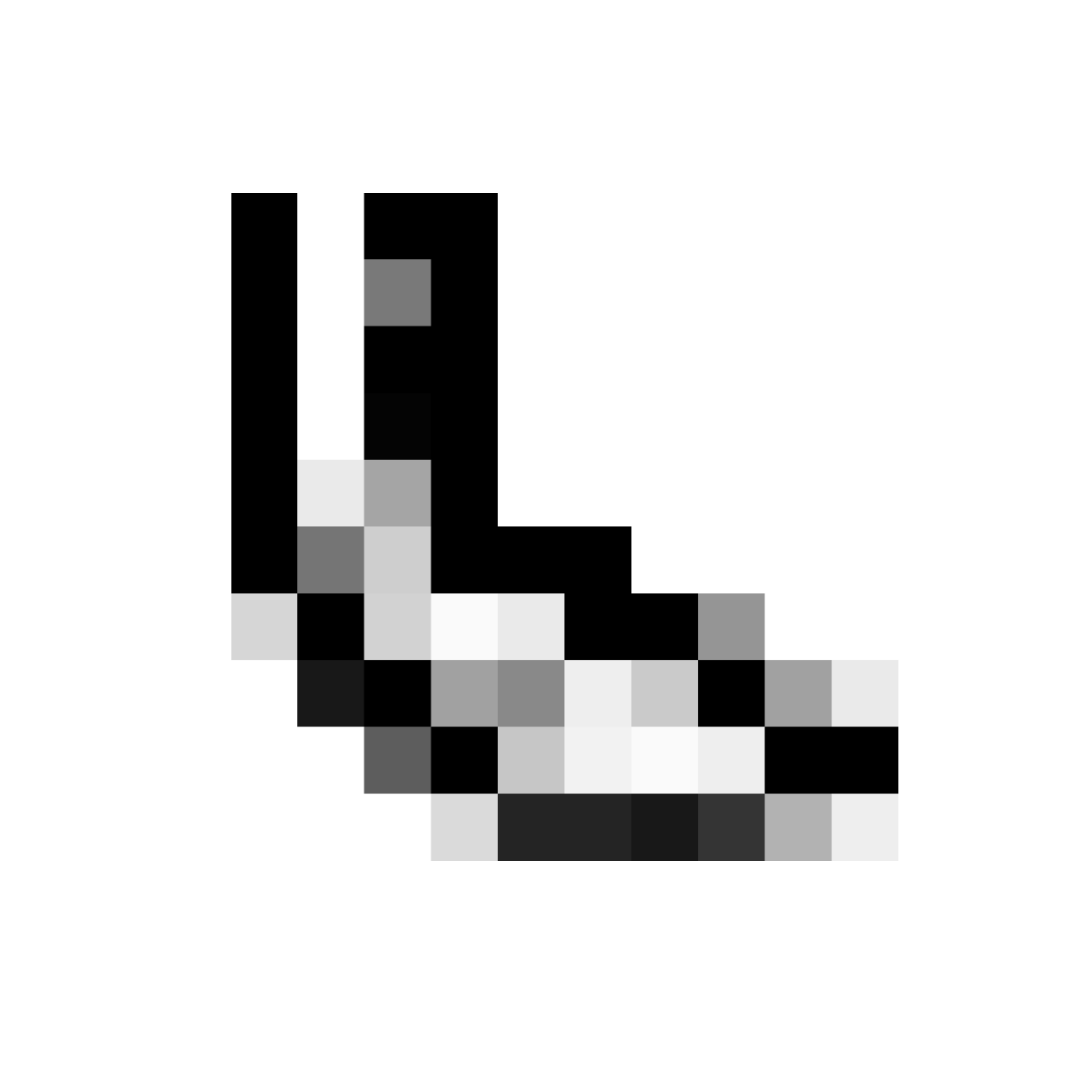}&\hspace{-1.0cm}\includegraphics[width=1.5in]{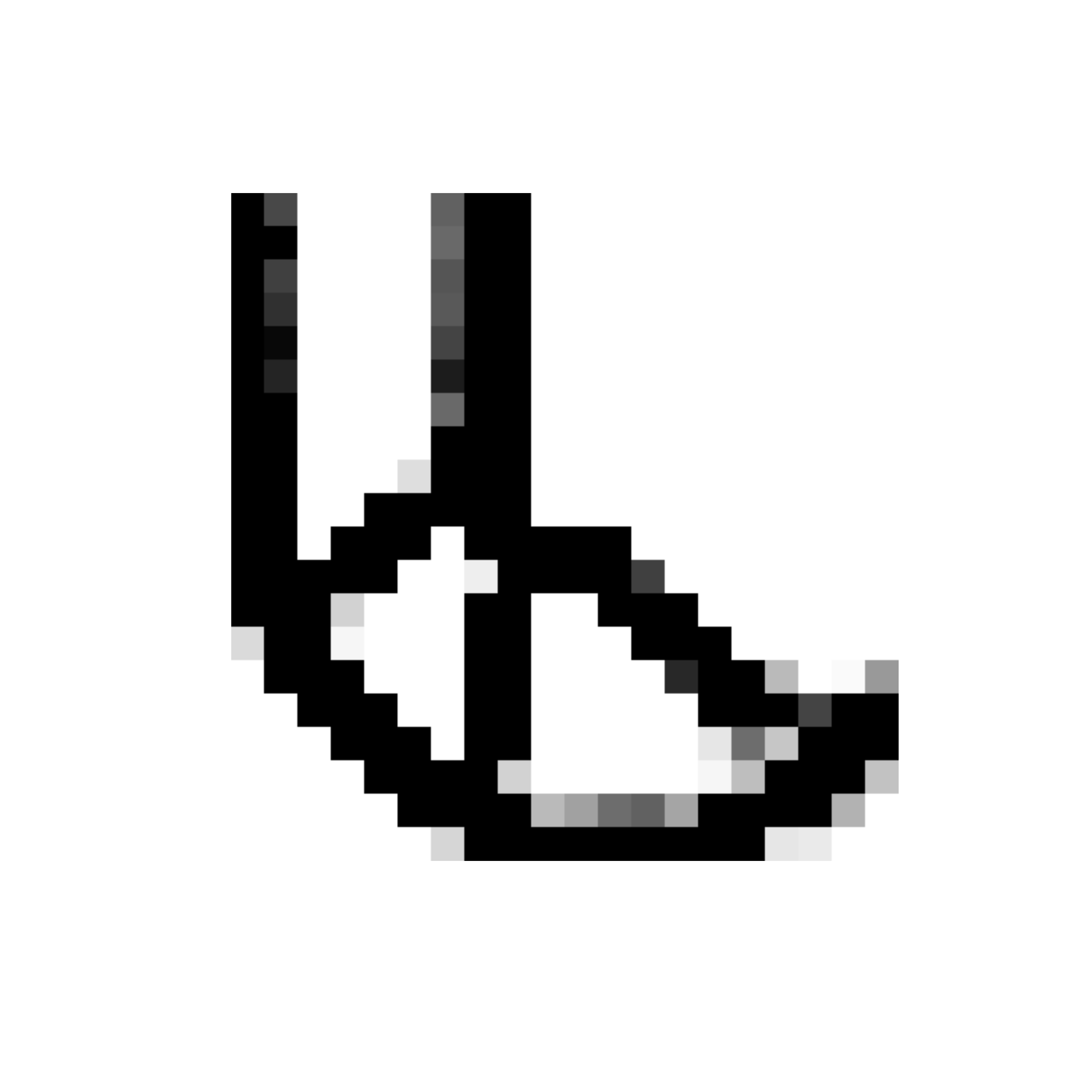}&\hspace{-1.0cm}\includegraphics[width=1.5in]{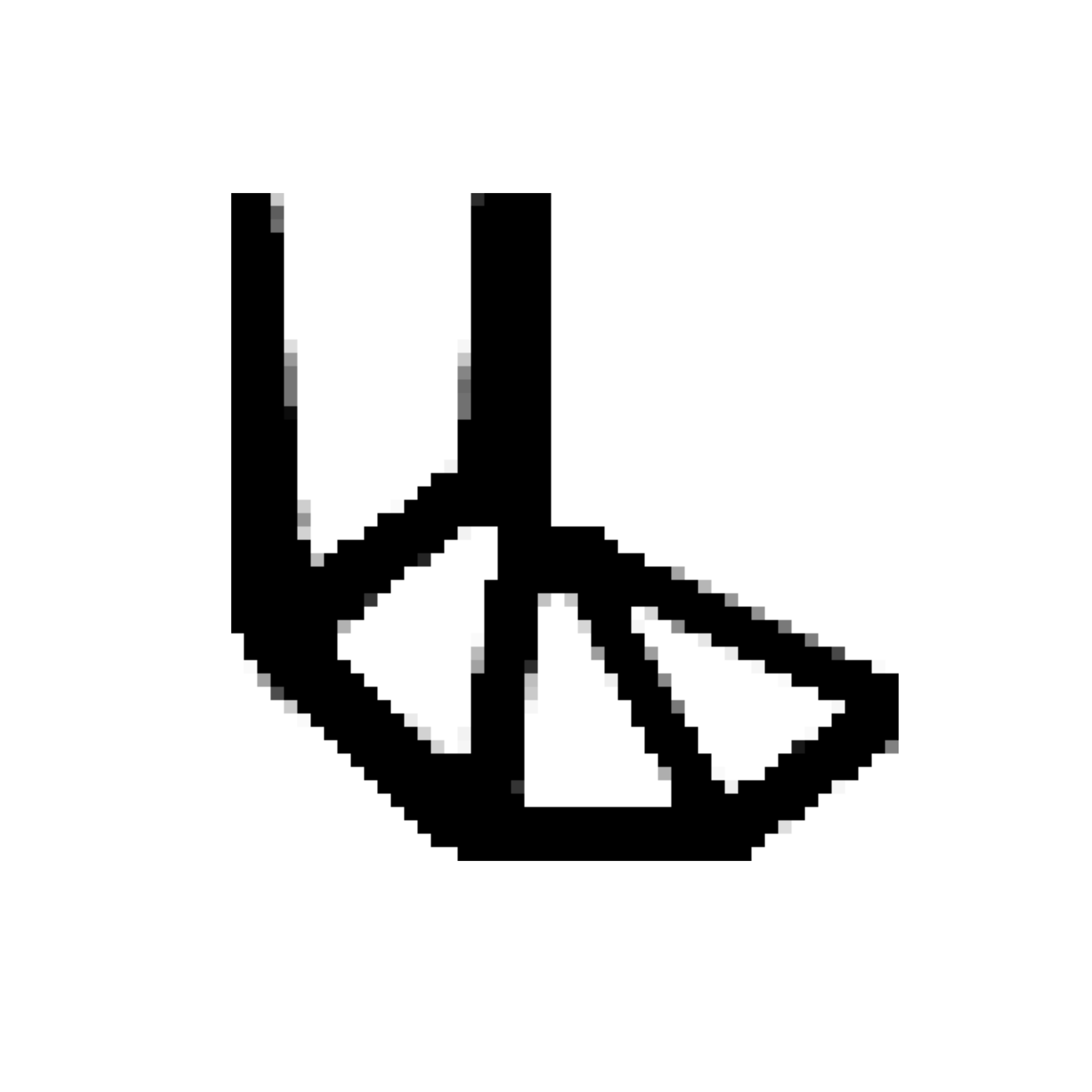}
&\hspace{-1.1cm}\includegraphics[width=1.5in]{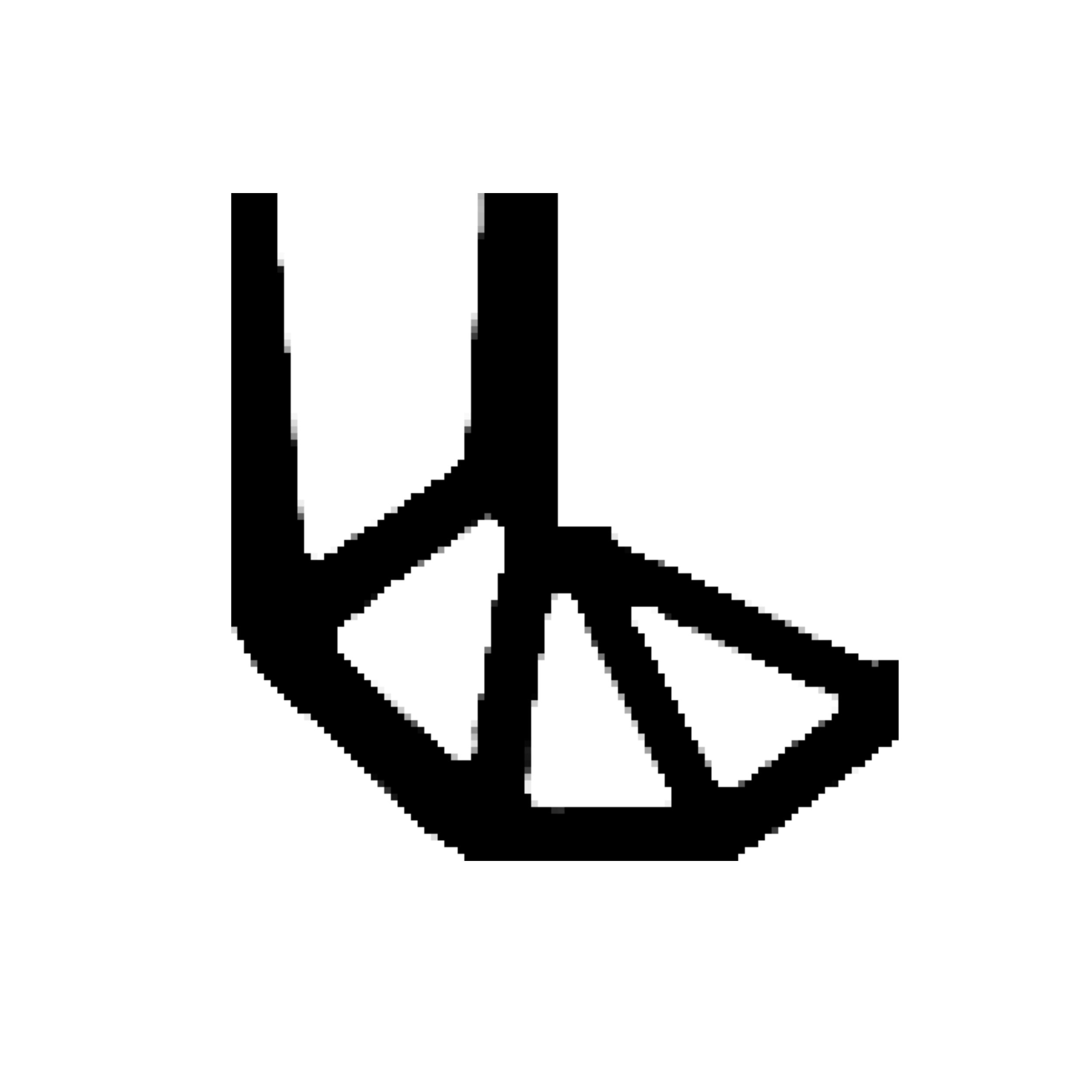}\\
\hspace{-0.4cm}\includegraphics[width=1.5in]{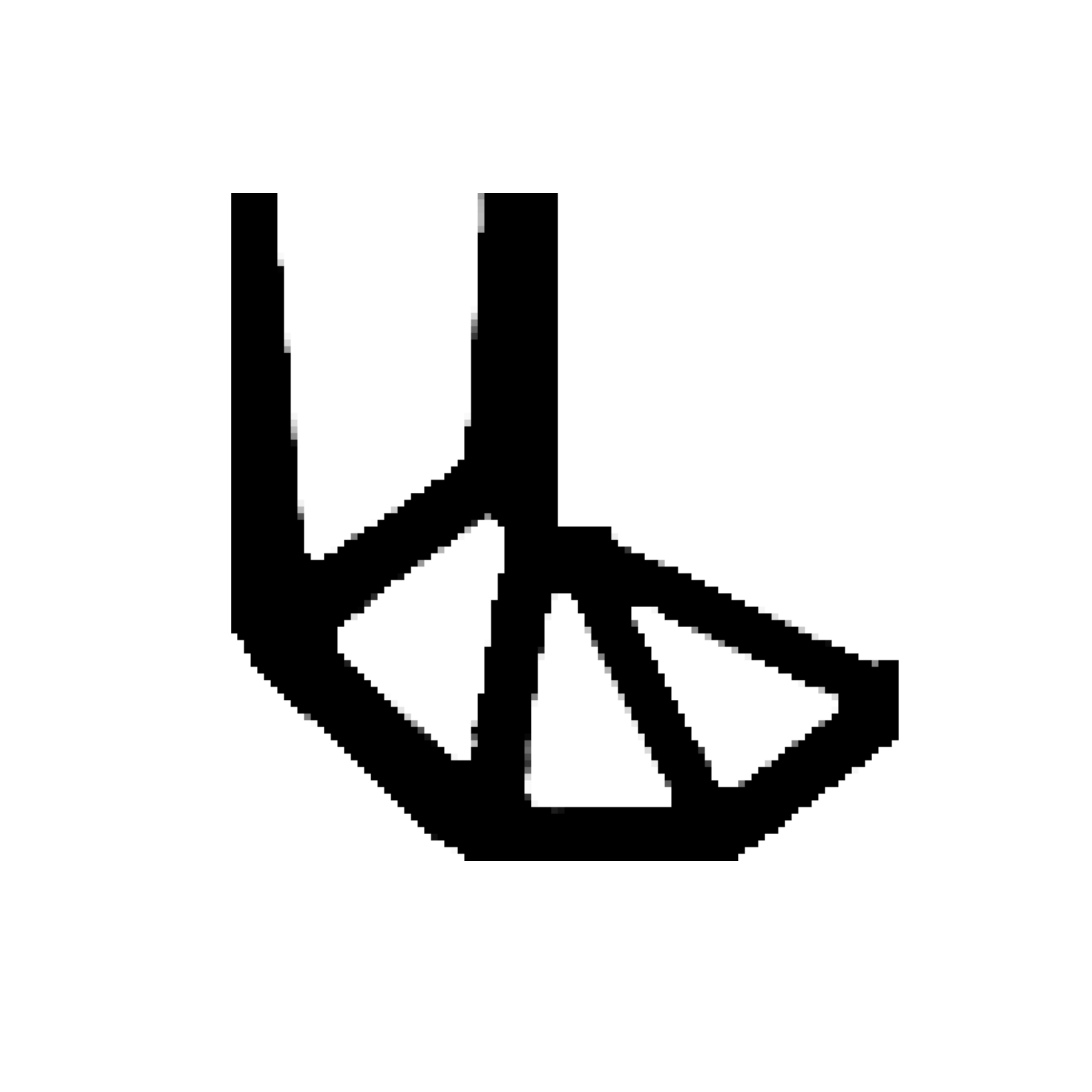}&\hspace{-0.8cm}\includegraphics[width=1.5in]{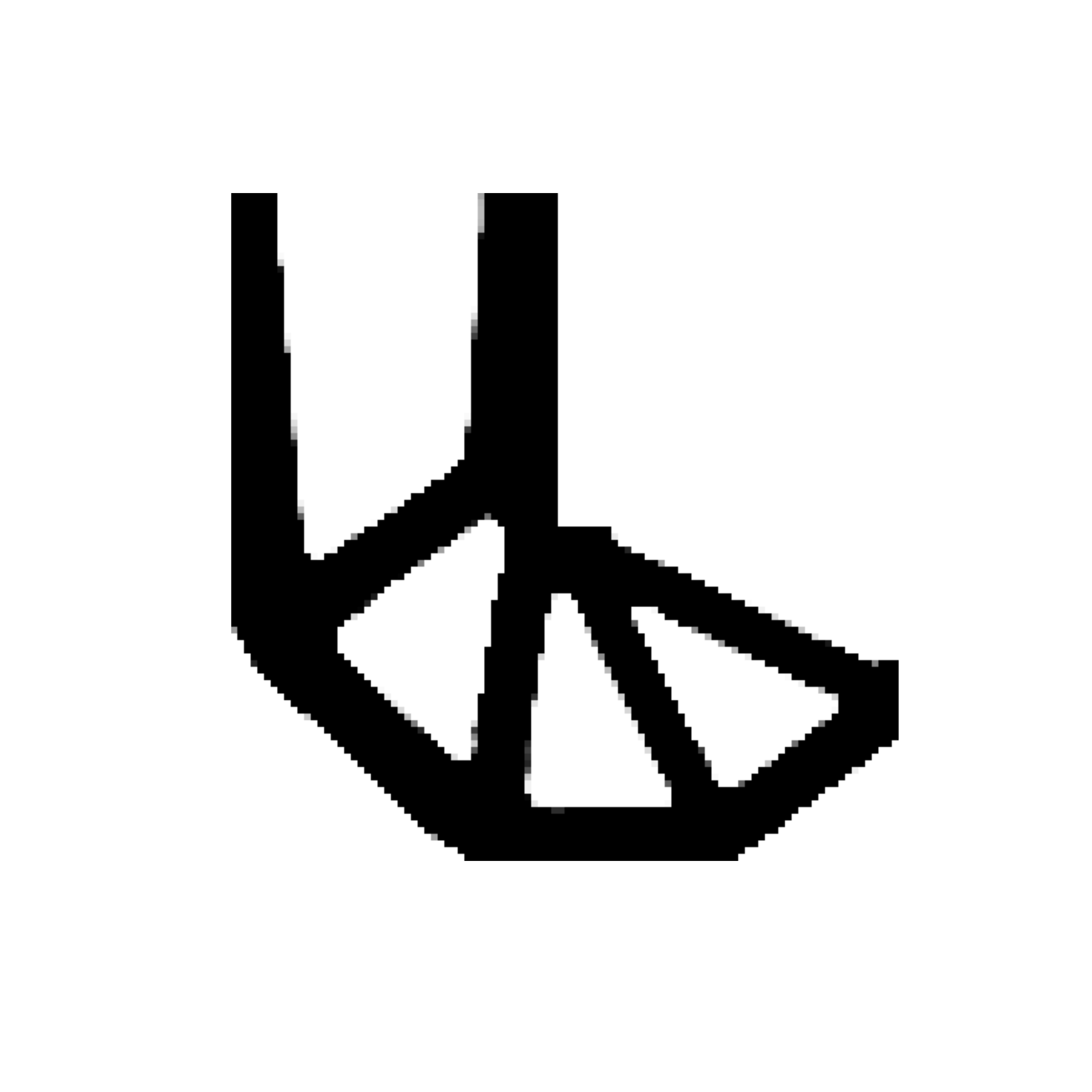}&\hspace{-1.0cm}\includegraphics[width=1.5in]{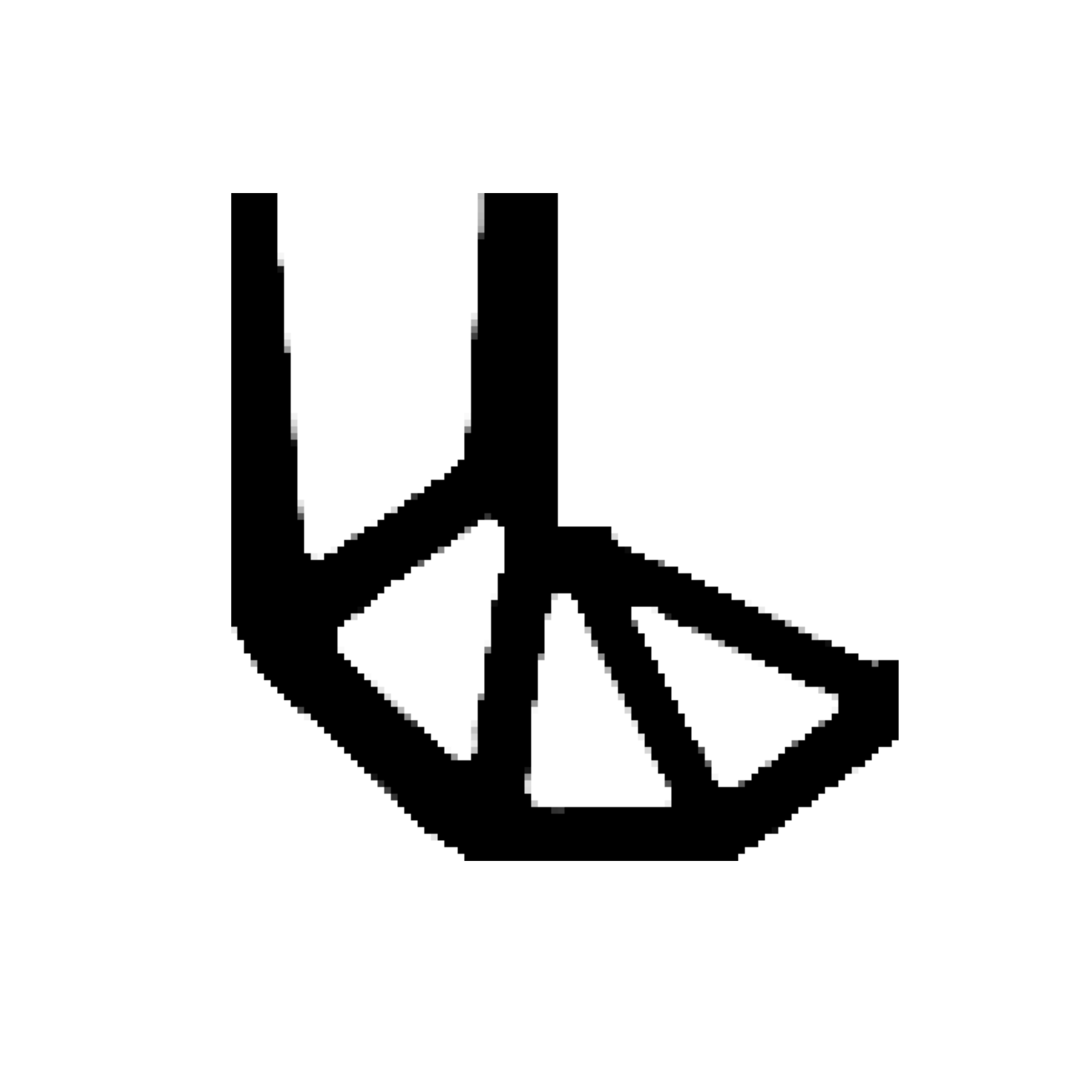}&\hspace{-1.0cm}\includegraphics[width=1.5in]{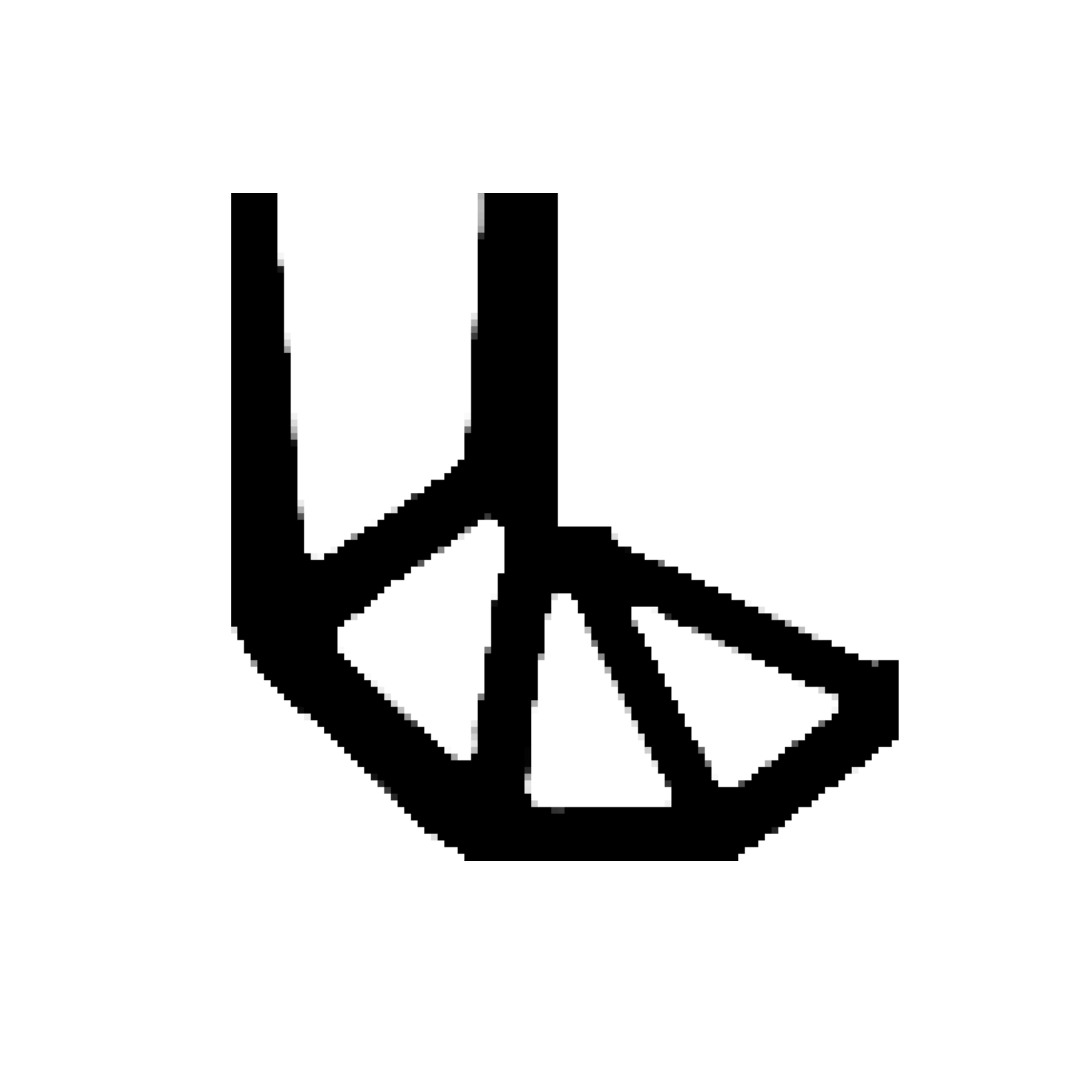}
\end{tabular}
\caption{{Topology optimization results for different meshes: Single resolution optimization with $4 \times 4$, $10 \times 10$, $20 \times 20$, $50 \times 50$ and $100 \times 100$ meshes(top row); Bi-fidelity optimization with $4 \times 4$, $10 \times 10$, $20 \times 20$, $50 \times 50$ meshes (bottom row). All bi-resolution results are obtained with $10$ high-resolution simulations which are almost identical to the top right plot which uses $43$ simulations at each design iterate. }}\label{fign_vary_mesh_tol_K}
\end{figure}

To quantify differences between single and bi-resolution optimizations we perform the same study as done in previous example. Table~\ref{tabNEF_1} shows the error versus the cost for single and bi-resolution optimizations. We again observe that $10\times10$ mesh is the most economical choice as it yields the small error while the most of computation is performed on its relatively coarse mesh.

\begin{table}[!h]
\caption{Geometric uncertainty: Error vs cost for single and bi-resolution optimization.}
\centering
\begin{tabular}{l c c c c}
  \hline\hline
Resolution & No. Iter. & No. Hi. Res. Sim. & $e_{\bm \rho}$   & $e_Q$  \\
\hline
Hi. Res. $100 \times 100$    & 240 & 10320 & - & -\\
Bi-Res. $4 \times 4$   & 283 & 2830 &   0.0236 & 4.24e-03 \\
{\color{blue} Bi-Res. $10 \times 10$ }  & 276 & 2760 & 0.0155 & 1.51e-04 \\
Bi-Res. $20 \times 20$   & 268 & 2680 & 0.0130 & 1.29e-04\\
Bi-Res. $50 \times 50$   & 255 & 2550 & 0.0093 & 9.10e-05\\
\hline\hline
  \end{tabular}
  \label{tabNEF_1}
\end{table}

As mentioned the processed design variables $\bar {\bm \rho}$ are random due to the randomness in $\tau$. We define the error in the mean and standard deviation of processed design variables between single and bi-resolution models as $e_{\mu(\bar{\bm \rho})}=\| \mu(\bar{\bm \rho}^B) - \mu(\bar{\bm \rho}^H)\|/\sqrt{n^H_{elem}}$ and $e_{\sigma(\bar{\bm \rho})}=\|\sigma(\bar{\bm \rho}^B) - \sigma(\bar{\bm \rho}^H)\|/\sqrt{n^H_{elem}}$ which are computed as $e_{\mu(\bar{\bm \rho})}=2.75e-03$ and $e_{\sigma(\bar{\bm \rho})}=1.86e-03$. Figure~\ref{fign_mean_std} shows the mean and standard variation for the processed design variables obtained from single and bi-resolution optimization.

\begin{figure}[!h]
\centering
\begin{tabular}{c c}
\vspace{-0.5cm}
\hspace{-0.4cm}\includegraphics[width=2.0in]{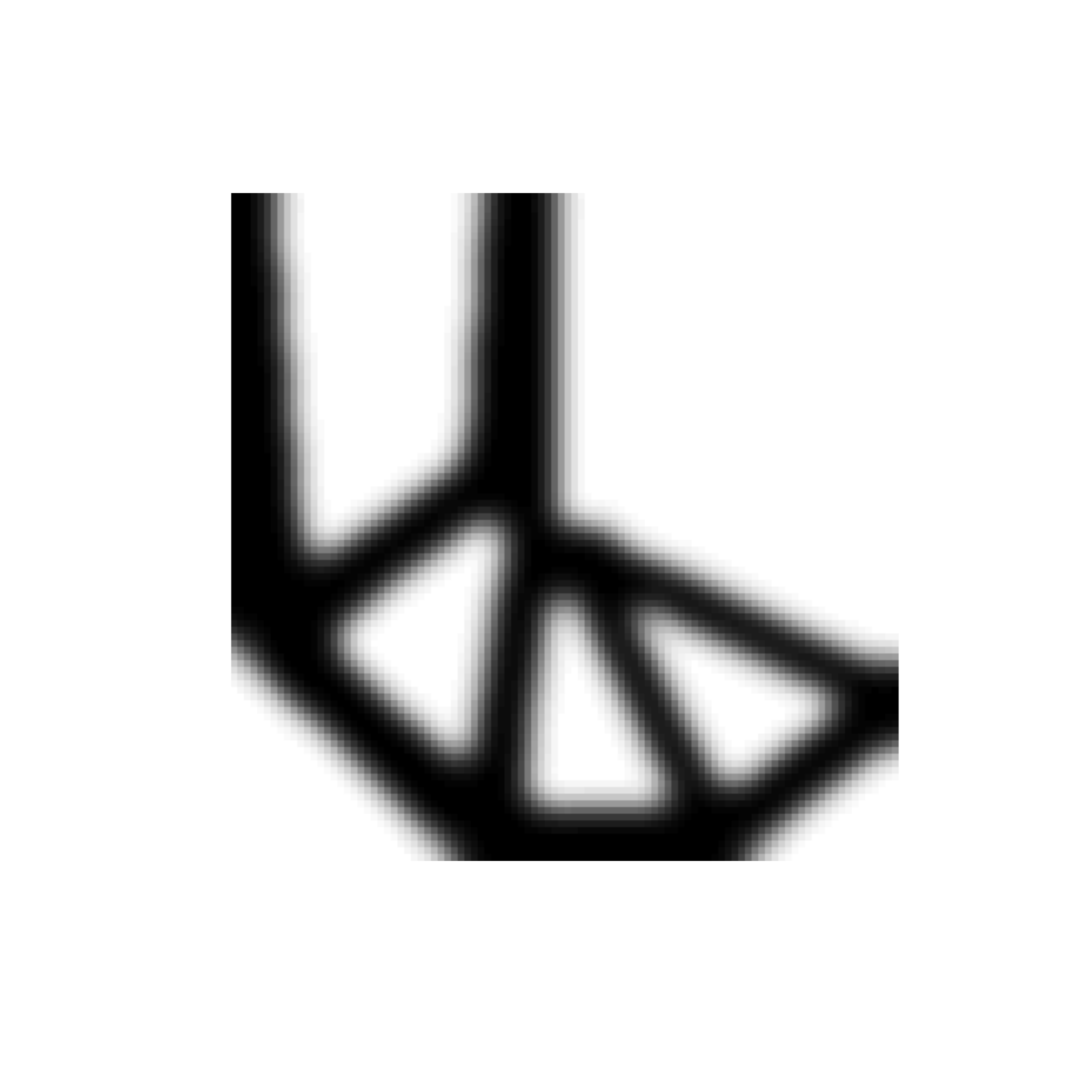}&\hspace{-0.8cm}\includegraphics[width=2.0in]{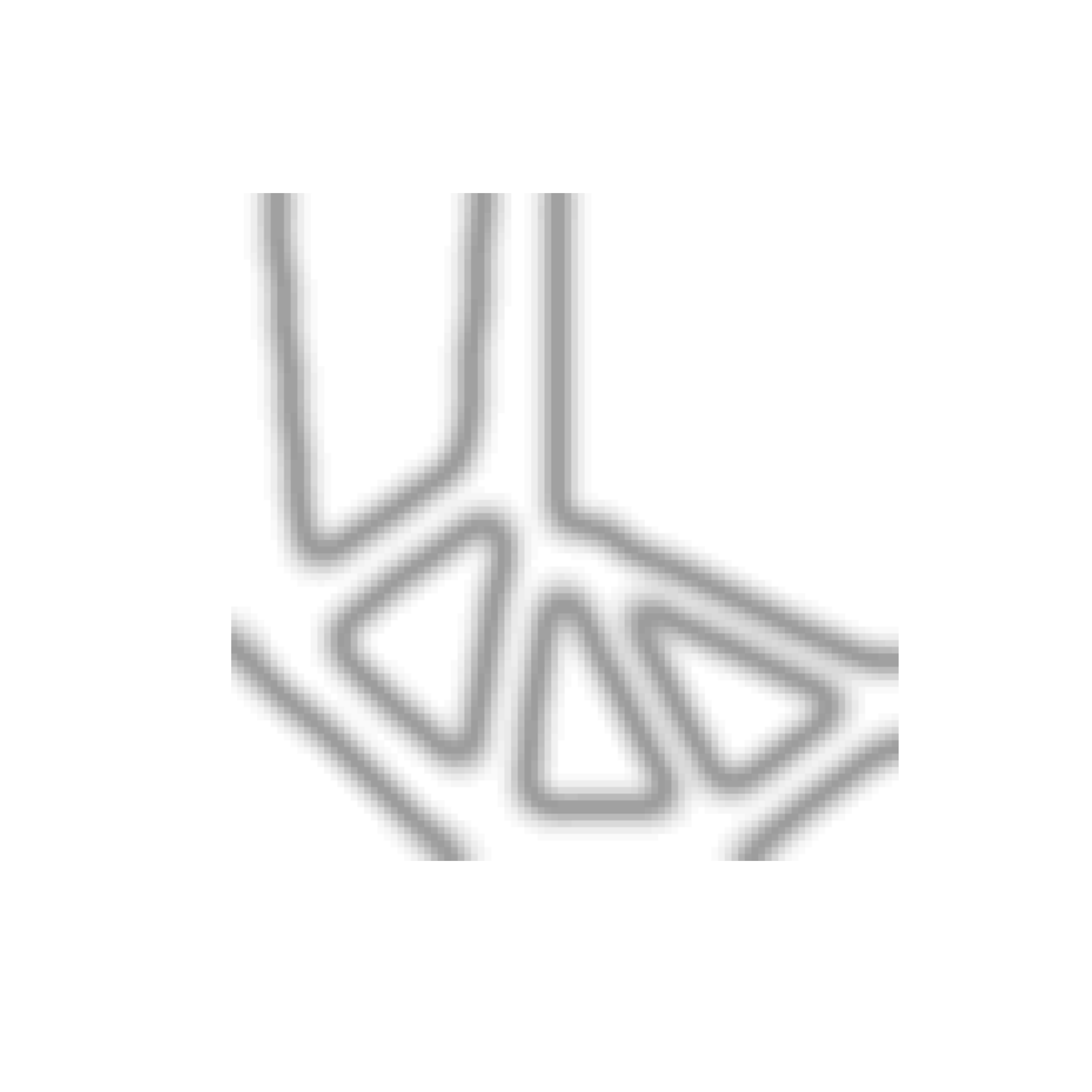}\\
\hspace{-0.4cm}\includegraphics[width=2.0in]{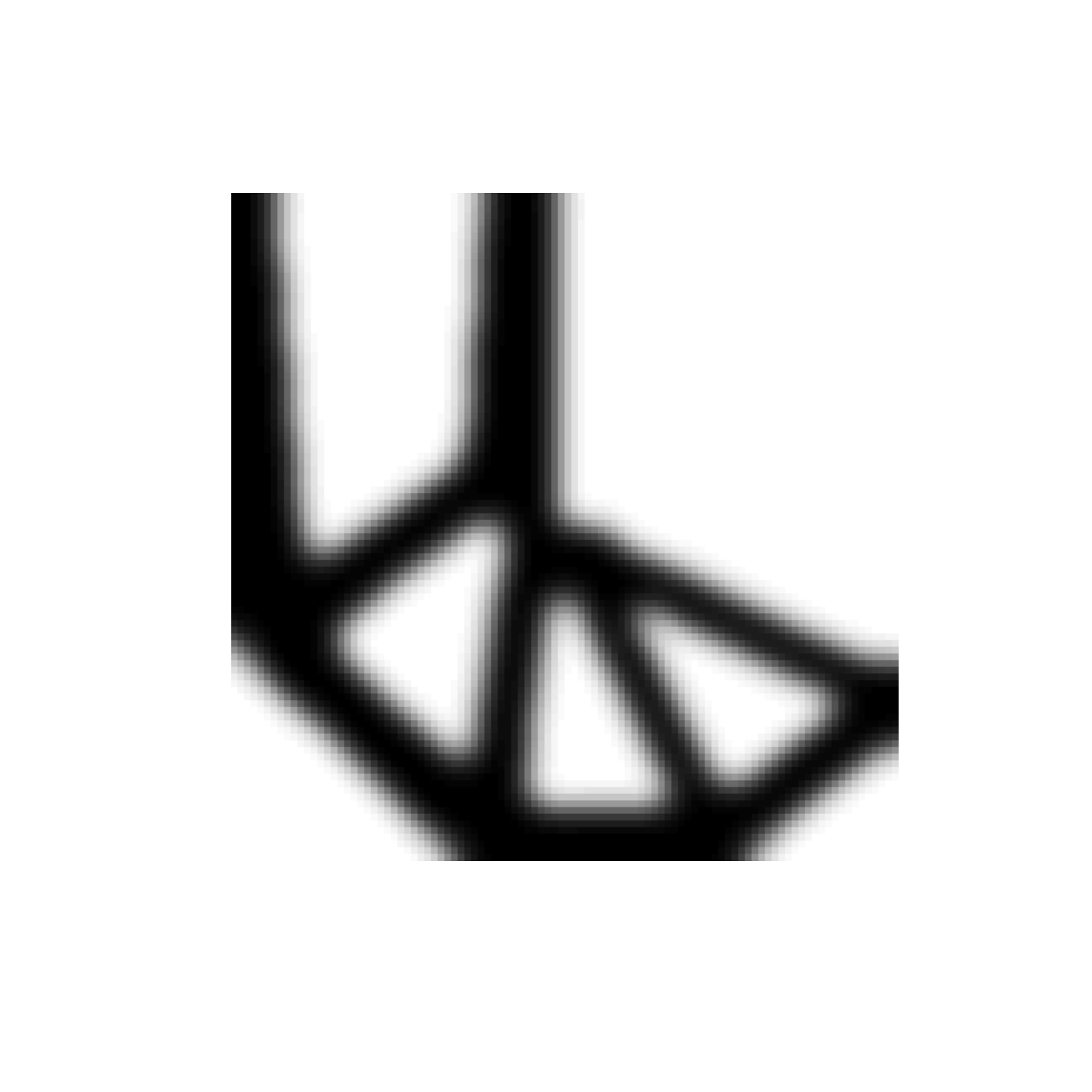}&\hspace{-0.8cm}\includegraphics[width=2.0in]{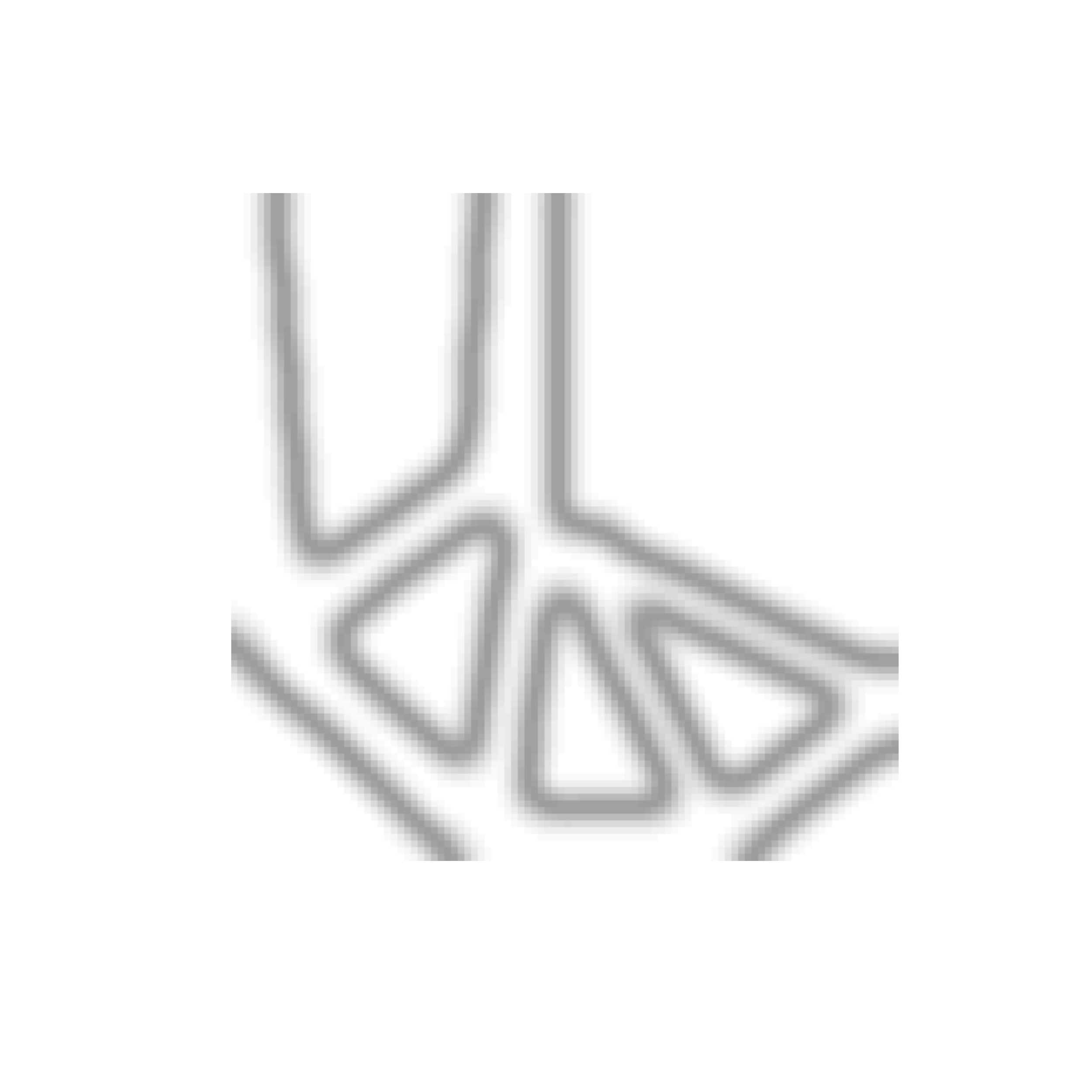}\\
\end{tabular}
\caption{{Mean (left) and standard deviation (right) of processed design variables $\bar{\bm \rho}$ obtained from bi-resolution optimization with $10 \times 10$ mesh (top) and single resolution of $100 \times 100$ mesh (bottom).}}\label{fign_mean_std}
\end{figure}

Finally to show the effectiveness of our approach in approximating challenging quantities of interest we compute the parametric Von-Mises stress for the optimal design using the bi-resolution approach and compare it with Monte Carlo simulations. Figure~\ref{fign_stress} shows a realization of high- and low-resolution Von-Mises stresses associated with $100 \times 100$  and $10 \times 10$ meshes on one of $43$ quadrature points. It is again observed that the low-resolution mesh provides no insight on the stress distribution however using it in the bi-resolution framework in conjunction with $10$ high-resolution stress distribution we can approximate the rest of $33$ high resolution stresses.

\begin{figure}[!h]
\centering
\begin{tabular}{c c}
\vspace{-0.5cm}
\hspace{-0.4cm}\includegraphics[width=3.0in]{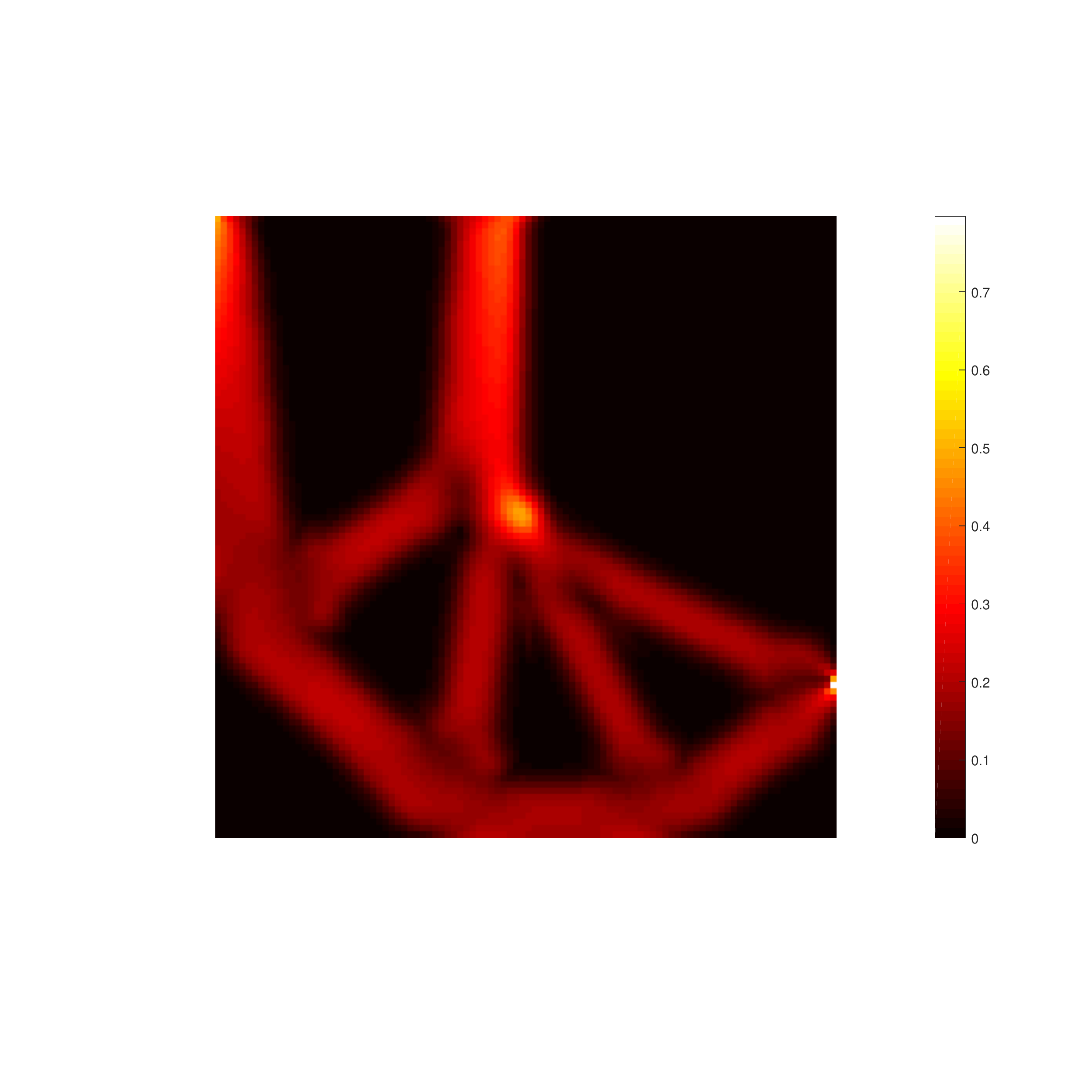}&\hspace{-0.8cm}\includegraphics[width=3.0in]{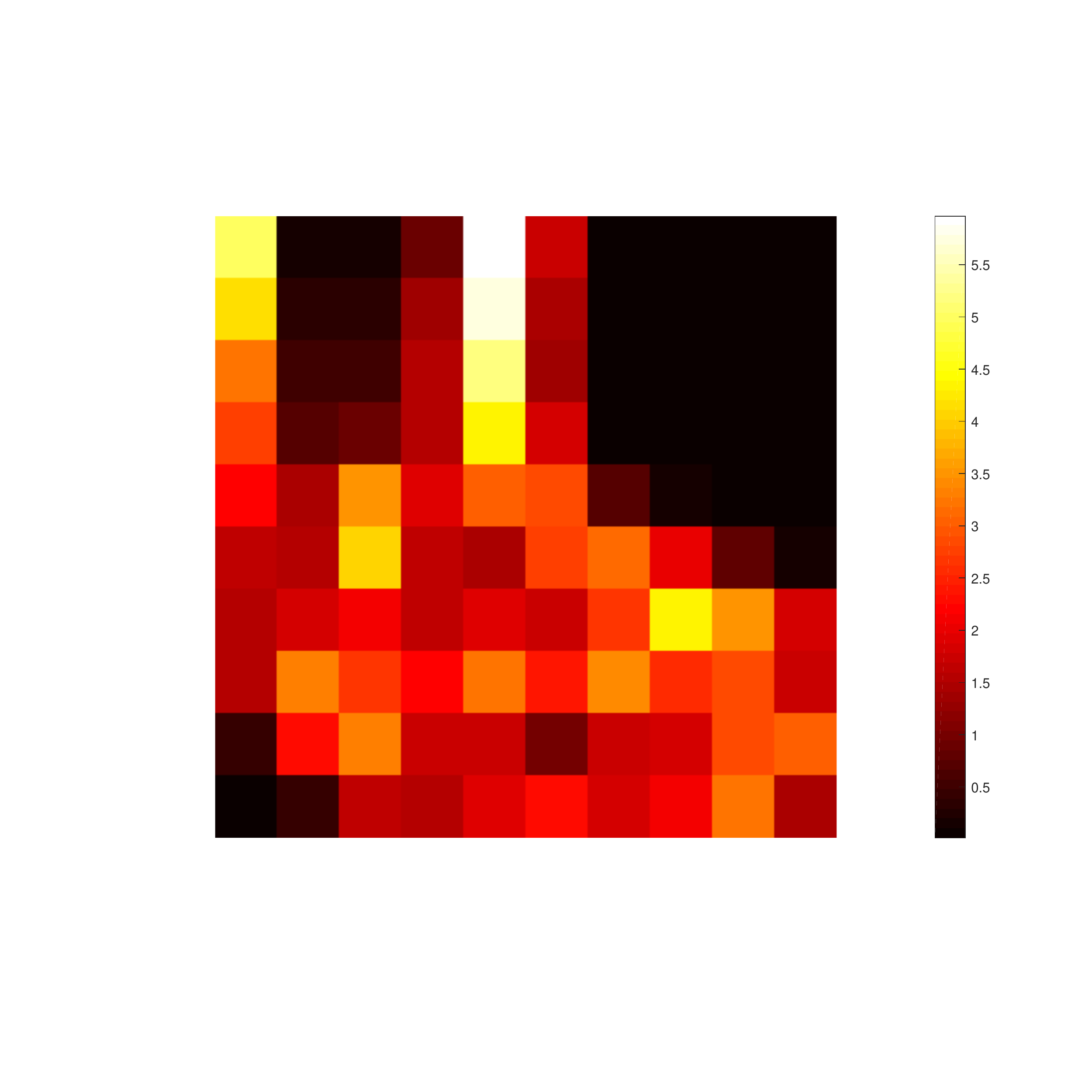}\\
\end{tabular}
\caption{{A realization of high- and low-resolution Von-Mises stresses associated with $100 \times 100$  and $10 \times 10$ meshes.}}\label{fign_stress}
\end{figure}

We use the bi-resolution quadrature samples to compute the mean and standard deviation of spatially averaged stresses. In addition we perform $1000$ high-resolution Monte Carlo simulations (associated with $1000$ samples of threshold random field cf. Equation~\eqref{pcen216}) to find the mean and standard deviation of the same quantity of interest. To investigate the error we use high-level sparse grid with $n=1217$ nodes as the true solution~\cite{HeissQ}. Figure~\ref{fign_stress_err} shows the convergence of the MC simulations and the error between the bi-resolution approximation and the true solution. It is again seen that the bi-resolution approximation of stress with only $10$ high-resolution simulations outperform MC simulations with much larger number of high-resolution simulations.

\begin{figure}[!h]
\centering
\includegraphics[width=6.5in]{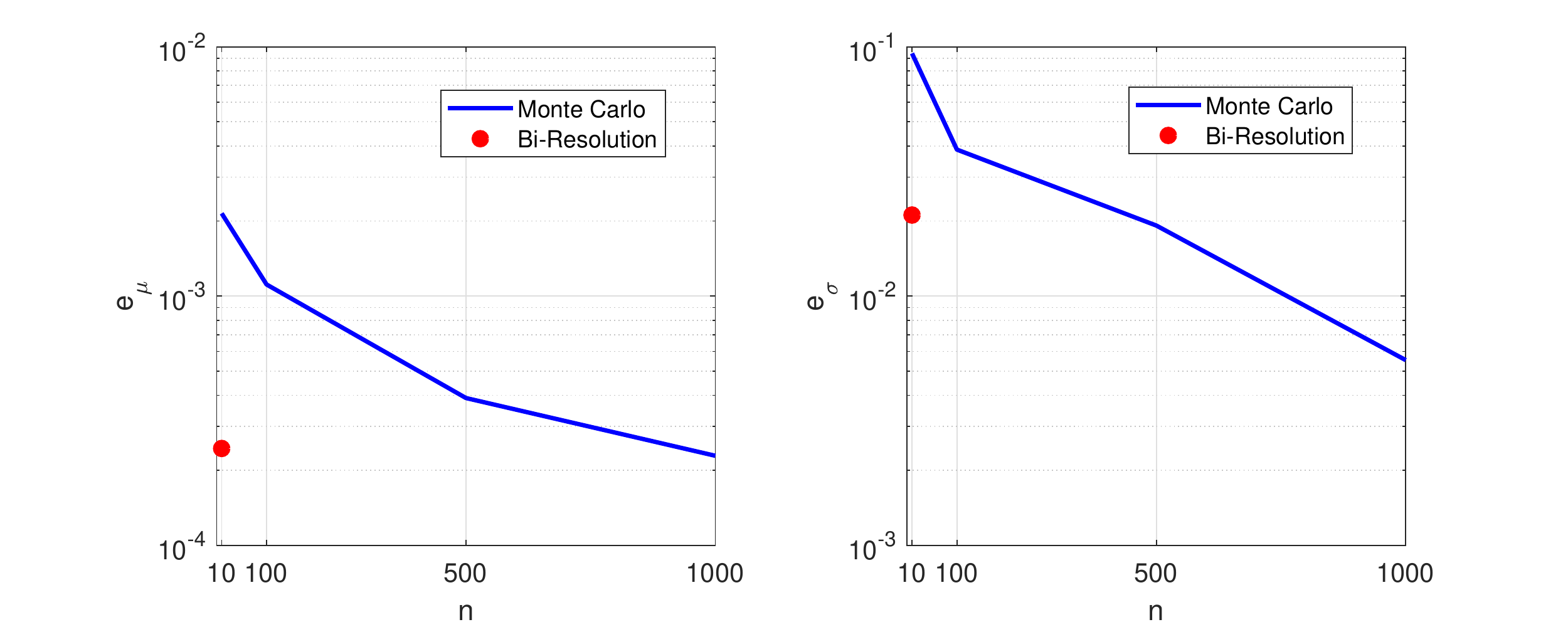}
\caption{{Error in mean (left) and standard deviation (right) of spatially averaged Von-Mises stress. The abscissa $n$ is the number of high-resolution simulations.}}\label{fign_stress_err}
\end{figure}

\section{Conclusion}\label{S5}

We present a systematic approach for parametric topology optimization with multi-resolution finite element models. The parametric variation is identified from an inexpensive low-resolution model where large number of simulations can be performed. The identified links among low-resolution samples are used to approximate the high-resolution parameter space which now only requires a limited number of high-fidelity simulations. We use the bi-fidelity surrogate of displacement for compliance-based topology optimization on benchmark problems with loading and geometric variabilities. An error estimate for bi-fidelity approximation of compliance is derived which certifies the convergence of approach. Numerical results are provided to delineate the convergence analysis. It is shown that the bi-resolution approach yields almost identical design to single resolution optimization with significantly smaller computational cost especially in expensive problems such as topology optimization under manufacturing uncertainty.

\section*{acknowledgements}
This research was sponsored by ARL under Cooperative Agreement Number W911NF-12-2-0023. The views and conclusions contained in this document are those of the authors and should not be interpreted as representing the official policies, either expressed or implied, of ARL or the U.S. Government. The U.S. Government is authorized to reproduce and distribute reprints for Government purposes notwithstanding any copyright notation herein. The first and third authors are partially supported by AFOSR FA9550-15-1-0467. The third author is partially supported by DARPA EQUiPS N660011524053 and NSF DMS 1720416.


\end{document}